\documentclass[12pt,reqno]{amsart}
\usepackage{amssymb, amscd, stmaryrd}
\usepackage{enumerate}
\usepackage{cases}
\usepackage{amsfonts}
\usepackage[dvips,dvipdf]{graphicx}
\usepackage[usenames,dvipsnames]{color}

\usepackage{color}
\usepackage{color, colortbl}
\usepackage{multirow}
\usepackage[table]{xcolor}
\usepackage{comment}
\usepackage{subfigure}
\usepackage{epstopdf}
\usepackage{tikz}
\usepackage{chngcntr}
\usepackage[stable]{footmisc}
\usepackage{booktabs}

\usepackage{geometry}
\makeatletter
\newcommand\figcaption{\def\@captype{figure}\caption}
\newcommand\tabcaption{\def\@captype{table}\caption}
\makeatother

\counterwithin*{equation}{section}

\newcommand{\bm}[1]{\mbox{\boldmath$#1$}}

\newcommand{\pa}{\partial}

\newcommand{\maT}{\mathcal T}

\newcommand{\be}{\begin{eqnarray}}
\newcommand{\ee}{\end{eqnarray}}
\newcommand{\ben}{\begin{eqnarray*}}
\newcommand{\een}{\end{eqnarray*}}

\def\be{\begin{equation}}
\def\ee{\end{equation}}
\def\bes{\begin{equation*}}
\def\ees{\end{equation*}}
\def\beq{\begin{eqnarray}}
\def\eeq{\end{eqnarray}}
\def\beqs{\begin{eqnarray*}}
\def\eeqs{\end{eqnarray*}}
\def\bal{\begin{aligned}}
\def\eal{\end{aligned}}
\def\bsqs{\begin{subequations}}
\def\esqs{\end{subequations}}

\newcommand{\maK}{\mathcal K}

\newtheorem{theorem}{Theorem}[section]

\newtheorem{corollary}[theorem]{Corollary}

\newtheorem{lemma}[theorem]{Lemma}
\newtheorem{lem}[theorem]{Lemma}

\theoremstyle{definition}
\newtheorem{definition}[theorem]{Definition}
\theoremstyle{definition}
\newtheorem{algorithm}[theorem]{Algorithm}
\theoremstyle{remark}
\newtheorem{remark}[theorem]{Remark}
\newtheorem{example}[theorem]{Example}

\makeatletter
\newcommand{\definetitlefootnote}[1]{%
  \newcommand\addtitlefootnote{%
    \makebox[0pt][l]{$^{\bigstar}$}%
    \footnote{\protect\@titlefootnotetext}
  }%
  \newcommand\@titlefootnotetext{\spaceskip=\z@skip $^{\bigstar}$#1}%
}
\makeatother

\title[FEM for biharmonic]{A $C^0$ finite element method for the biharmonic problem with Dirichlet boundary conditions in a polygonal domain \addtitlefootnote}

\author[H. Li, C. D. Wickramasinghe, P. Yin]{Hengguang Li$^\dagger$, Charuka D. Wickramasinghe$^\dagger$, Peimeng Yin$^{*}$}
\address{$^\dagger$ Department of Mathematics, Wayne State University, Detroit, MI 48202, USA}
\email{li@wayne.edu; gi6036@wayne.edu}
\address{$^*$ Multiscale Methods and Dynamics Group, Computer Science and Mathematics Division, Oak Ridge National Laboratory,  Oak Ridge, Tennessee 37831, USA.
}\email{yinp@ornl.gov}


\thanks{$^*$ Corresponding author.}

\definetitlefootnote{This manuscript has been authored in part by UT-Battelle, LLC, under contract DE-AC05-00OR22725 with the US Department of Energy (DOE). The US government retains and the publisher, by accepting the article for publication, acknowledges that the US government retains a nonexclusive, paid-up, irrevocable, worldwide license to publish or reproduce the published form of this manuscript, or allow others to do so, for US government purposes. DOE will provide public access to these results of federally sponsored research in accordance with the DOE Public Access Plan (http://energy.gov/downloads/doe-public-access-plan).}

\begin{document}

\date{\today}

\begin{abstract}
In this paper, we study the biharmonic equation with Dirichlet boundary conditions in a polygonal domain. In particular, we propose a method that effectively decouples the fourth-order problem into a system of two Poison equations and one Stokes equation, or a system of one Stokes equation and one Poisson equation. It is shown that the solution of each system is equivalent to that of the original fourth-order problem on both convex and non-convex polygonal domains. Two finite element algorithms are in turn proposed to solve the decoupled systems. In addition, we show the regularity of the solutions in each decoupled system in both the Sobolev space and the weighted Sobolev space, and we derive the optimal error estimates for the numerical solutions on both quasi-uniform meshes and graded meshes. Numerical test results are presented to justify the theoretical findings.
\end{abstract}

\maketitle


\section{Introduction}
We are interested in $C^0$ finite element method for the biharmonic problem
\begin{eqnarray}\label{eqnbi}
\Delta^2\phi=f \quad {\rm{in}} \ \Omega,\qquad \quad \phi=0 \quad {\rm{and}} \quad \partial_\mathbf{n} \phi=0 \quad {\rm{on}} \ \pa\Omega,
\end{eqnarray}
where $\Omega\subset \mathbb R^2$ is a polygonal domain, $\mathbf{n}$ is the outward normal derivation, and $f$ is a given function.
The boundary conditions in (\ref{eqnbi}) are referred to as the homogeneous Dirichlet boundary conditions or clamped boundary conditions \cite{MR1422248, MR2499118} that occur for example in fluid mechanics \cite{girault1986} and linear elasticity \cite{Ciarlet}.
Equation (\ref{eqnbi}) is a fourth-order elliptic equation, there are three classical approaches to discretizing the biharmonic equation in the literature. The first type is conforming finite element methods, such as the Argyris finite element method \cite{argyris1968tuba}, which requires globally $C^1$ finite element spaces.
The second type is the nonconforming method, such as Morley finite element methods \cite{morley1968}, which generally involve low order polynomials. The third type is the mixed finite element methods that require only Lagrange finite element spaces.

To obtain the mixed finite element methods, one first decomposes the fourth-order equation into a lower order system, then applies the finite element method to the system. 
Typically, problem (\ref{eqnbi}) is decoupled into two Poisson problems by introducing one intermediate function $\psi=-\Delta \phi$. Different from biharmonic equation with Navior boundary conditions ($u=\Delta u =0$ on $\partial \Omega$) that allows one
to obtain two Poisson equations that are completely decoupled \cite{lyz2020}, applying such decomposition to (\ref{eqnbi}) leads to two Poisson equation with either undertermined or overdetermined boundary condition,
\begin{eqnarray*}
\left\{\begin{array}{ll}
-\Delta \psi=f \quad {\rm{in}} \ \Omega,\\
\text{no data} \quad {\rm{on}} \ \pa\Omega;
\end{array}\right.
\qquad \qquad {\rm{and}}\qquad \qquad\left\{\begin{array}{ll}
-\Delta \phi=\psi \quad {\rm{in}} \ \Omega,\\
\phi=\partial_\mathbf{n} \phi=0 \quad {\rm{on}} \ \pa\Omega.
\end{array}\right.
\end{eqnarray*}

To overcome this difficulty, 
Ciarlet and Raviart \cite{MR0657977} introduced a mixed finite element method for problem (\ref{eqnbi}) in a smooth domain by introducing a conditioned function space. Later, Monk improved the mixed finite element method in \cite{monk1987} to allow the intermediate function $\psi$ in $H^1(\Omega)$. 
Their main difference lies in the smoothness of the intermediate function $\psi$, namely $\psi$ is in different Sobolev spaces, but their solutions could be equivalent under some smoothness assumptions \cite{de2015}. However, these two mixed finite element methods can only be applied to smooth domains and convex polygonal domains. For non-convex polygonal domains, the solution of the Ciarlet-Raviart mixed method is equivalent to the weak solution of (\ref{eqnbi}), but the corresponding mixed finite element solutions won't converge to the exact solution due to the low regularity \cite{MR0657977, de2015}, while the Monk mixed method could result in spurious solutions as  discussed in \cite{de2015}, in which an augmented Monk method was also introduced to remove the spurious solutions.

There is rich literature on mixed finite element methods for the biharmonic problem (\ref{eqnbi}). For example,  the work in \cite{davini2000} studied a mixed method, and in \cite{zulehner2015} preconditioning techniques were investigated for the discrete system, both of which are based on the Ciarlet-Raviart formulation \cite{MR0657977}. In addition, more mixed variational formulations can be found in \cite{krendl2016,gallistl2017}. 

We concern with the efficient numerical method for (\ref{eqnbi}) in a polygonal domain, especially, the non-convex polygonal domain. 
Motivated by the fact that 
a Stokes problem could be reduced to a biharmonic problem on convex polygonal domains or smooth domains \cite{Ciarlet, brezzi1986, brezzi1991, bacuta2003, gallistl2017}, we resort to decompose problem (\ref{eqnbi}) in the following way. For given source terms $f$ of the biharmonic problem (\ref{eqnbi}), we first identify a source term of the Stokes problem, though there are infinity many choices for the source term of the Stokes problem, the velocity of the Stokes problem is uniquely determined by $f$. Then we decompose the biharmonic problem (\ref{eqnbi}) into a system of two Poisson problems and one Stokes problem  (P-S-P sequentially), or a system of one Stokes problem and one Poisson problem  (S-P sequentially), and we prove that the solution of each decoupled system, namely the solution of the last Poisson problem, is equivalent to the solution of the biharmonic problem (\ref{eqnbi}) in both convex and non-convex polygonal domains. 

Many fast solvers or finite element methods are available for the involved Stokes problem and the Poisson problems, see \cite{Ciarlet} for a review of the finite element methods. For each decoupled system, we provide a finite element algorithm, in which the Mini element method \cite{arnold1984} or the Taylor-Hood element method  \cite{girault1986, stenberg1990, BS02} is used to solve the Stokes problem, and the usual finite element method \cite{Ciarlet, Brenner02} is used for the Poisson problems.
For each algorithm, we carry out the error analysis for both the finite element approximations of the biharmonic problem and the Taylor-Hood element approximations of the involved Stokes problem.

For the algorithm involving only one Stokes problem and one Poisson problem (S-P), the error estimates 
for the Stokes problem are standard. 
The error estimate in $H^1$ norm for the biharmonic problem has an extra bound by the $L^2$ error for the velocity of the Stokes problem, and it has a convergence rate $h^{\min\{k,\alpha_0+1, 2\alpha_0\}}$, where $k$ is the degree of polynomials for finite element approximations of the biharmonic problem, and $\alpha_0$ depending on the largest interior angle $\omega$ is given by (\ref{alpha0}), which is visualized in Figure \ref{Regularity}.
The error estimate in $L^2$ norm has an extra bound depending on $\omega$ by either the $L^2$ or the $H^{-1}$ error for the velocity of the Stokes problem, and it has a convergence rate $h^{\min\{k+1,\alpha_0+2, 2\alpha_0\}}$.

For the algorithm involving two Poisson problems and one Stokes problem (P-S-P), the error estimates for the Poisson problem (the first Poisson problem) sharing the same source term as the biharmonic equation are standard. The $H^1$ error estimate for the velocity of the Stokes problem has an extra bound by the $L^2$ error for the first Possion problem. The error estimates in $L^2$ and $H^{-1}$ norms essentially use the property that the finite element approximations for the source term of the Stokes problem and the source term itself are invariant in inner product with any functions in $\textbf{curl}$ of finite element space (see Lemma \ref{fcurlgal} below). Finally, the $H^1$ error for the biharmonic problem has a convergence rate $h^{\min\{k,\beta_0+2,\alpha_0+1, 2\alpha_0\}}$ with $\beta_0=\frac{\pi}{\omega}$, and the error in $L^2$ norm has a convergence rate  $h^{\min\{k+1,\beta_0+3,\alpha_0+2, 2\alpha_0\}}$. 
For $k\leq 4$, we find $\beta_0+2$ in $H^1$ convergence rate and $\beta_0+3$ in $L^2$ convergence rate cannot achieve the minimum so the convergence rates are the same as those in the previous algorithm.
Moreover, two algorithms give the same convergence rates in non-convex polygonal domains for any $k \geq 1$.


In addition, we derive the regularity estimates for each proposed system in a class of Kondratiev-type weighted Sobolev space. Based on the regularity results, we in turn propose graded mesh refinement algorithms, such that the associated finite element algorithms recover the optimal convergence rates for both the finite element approximations of the biharmonic problem, and Mini element approximation or Taylor-Hood approximations of the involved Stokes problem. We shall point out that the thresholds of the grading parameter for the graded meshes could depend on the norm of the error and the regularity of the solution to a specific equation.

The rest of the paper is organized as follows.  In Section \ref{sec-2}, based on the general regularity theory for second-order elliptic equations and the Stokes equation, we review the weak solutions of all involved equations and show the equivalence of the solution of the proposed system to that of the original problem. 
In Section \ref{sec-3}, we propose two finite element algorithms and obtain  error estimates on quasi-uniform meshes.
In Section \ref{sec-4}, we introduce a weighted Sobolev space and derive regularity estimates for the solution. We also present the graded mesh algorithm and provide optimal error estimates on graded meshes.
We report  numerical test results in Section \ref{sec-5} to validate the theory.

Throughout the paper, the generic constant $C>0$ in our estimates may be different at different occurrences. It will depend on the computational domain, but not on the functions involved nor on the mesh level in the finite element algorithms.

\section{The biharmonic problem and its decoupled formulation}\label{sec-2}

\subsection{Well-posdedness and regularity of the solution}

Denote by $H^m(\Omega)$, $m\geq 0$,  the Sobolev space that consists of functions whose $i$th ($0\leq i\leq m$) derivatives are square integrable. 
Denote by $H^1_0(\Omega)\subset H^1(\Omega)$  the subspace consisting of functions with  zero trace on the boundary  $\partial\Omega$.
Let $L^2(\Omega):=H^0(\Omega)$, we denote the norm $\|\cdot\|_{L^2(\Omega)}$ by $\|\cdot\|$ when there is no ambiguity about the underlying domain.
For $s>0$, let $s=m+t$, where $m\in \mathbb Z_{\geq 0}$ and $0<t<1$.
Recall that for   $D\subseteq \mathbb{R}^d$,
the fractional order Sobolev space $H^s(D)$ consists of distributions $v$ in $D$ satisfying
$$
\|v\|^2_{H^s(D)}:=\|v\|^2_{H^m(D)} + \sum_{|\alpha|= m}\int_{D}\int_{D} \frac{|\partial^\alpha v(x) - \partial^\alpha v(y)|^2 }{|x-y|^{d+2t}} dxdy <\infty,
$$
where $\alpha=(\alpha_1, \ldots, \alpha_d)\in\mathbb Z^d_{\geq0}$ is a multi-index such that $\partial^\alpha=\partial_{x_1}^{\alpha_1}\cdots\partial^{\alpha_d}_{x_d}$ and $|\alpha|=\sum_{i=1}^d\alpha_i$.
We denote by $H_0^s(D)$ the closure of $C_0^\infty(D)$ in $H^s(D)$, and $H^{-s}(D)$ the dual space of $H_0^s(D)$.
The notation $[\cdot]^2$ represents the vector space. For example, $\mathbf{v}=(v_1,v_2)^T \in [H_0^1(\Omega)]^2$ represents $v_i \in H_0^1(\Omega)$, $i=1,2$, where $T$ is the transposition of a matrix or a vector. For $\mathbf{v}=(v_1,v_2)^T$, we denote
($\text{curl } \mathbf{v}) := \frac{\partial v_2}{\partial x_1} - \frac{\partial v_1}{\partial x_2}$. For a scalar function $\psi$, we denote $(\textbf{curl }{\psi}) := ({\psi}_{x_2}, -{\psi}_{x_1})^T$.

By applying Green's formulas, the variational formulation for biharmonic problem (\ref{eqnbi}) can be written as:
\begin{eqnarray}\label{eqn.firstbi}
a(\phi, \psi):=\int_\Omega \Delta \phi\Delta \psi dx=\int_\Omega f\psi dx=(f, \psi) \quad \forall \psi\in H_0^2(\Omega).
\end{eqnarray}
For a function $\psi \in H_0^2(\Omega)$, applying the Poincar\'e-type inequality \cite{Grisvard1992} twice, it follows
$$
a(\psi,\psi) = \|\Delta \psi\|^2 = |\psi|^2_{H^2(\Omega)} \geq C\|\psi\|^2_{H^2(\Omega)}.
$$
Thus, for
any $f \in H^{-2}(\Omega)$, we have by the Lax-Milgram Theorem that (\ref{eqn.firstbi}) admits a unique weak solution
$
\phi \in H_0^2(\Omega).
$

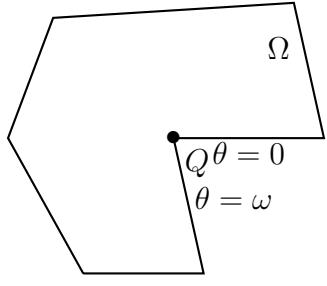
\begin{figure}
\begin{center}
\begin{tikzpicture}[scale=0.2]
\draw[thick]
(-6,-11) -- (2,-11) -- (0,-2) -- (10,-2) -- (8,7) -- (-8,6) -- (-11,-2) -- (-6,-11);
\draw (5,-3) node {$\theta = 0$};
\draw (4,-6) node {$\theta = \omega$};
\draw (7,4) node {$\Omega$};
\draw[thick] (0,-2) node {$\bullet$} node[anchor = north west] {$Q$};
\end{tikzpicture}
\end{center}
\vspace*{-15pt}
    \caption{Domain $\Omega$ containing one reentrant corner.}
    \label{fig:Omega}
\end{figure}

The regularity of the solution $\phi$ depends on the given data $f$ and the domain geometry \cite{agmon1959, blum1980boundary}.
In order to decouple  (\ref{eqnbi}), we assume that the polygonal domain $\Omega$ consists of $N$ vertices $Q_i$, $i=1,\cdots, N$, and the corresponding interior angles are $\omega_i\in (0, 2\pi)$. The largest interior angle $\omega = \max_i \omega_i \in [\frac{\pi}{3}, 2\pi)$ associated with the vertex $Q$. A sketch of the domain is given in Figure \ref{fig:Omega}.
We set $z_j$, $j=1,2,\cdots,n$ the solutions of the following characteristic equation corresponding to the the biharmonic problem (\ref{eqnbi}) (see, e.g.,  \cite{Grisvard1992}),
$$
\sin^2(z\omega) = z^2\sin^2(\omega),
$$
then there exists a threshold
\be\label{alpha0}
\alpha_0 := \min\{\rm{Re}(z_j), \  j=1,2,\cdots, n\} > \frac{1}{2},
\ee
such that when $0\leq \alpha< \alpha_0$, the biharmonic problem (\ref{eqnbi}) holds the regularity estimate \cite{kozlov2001, bacuta2002, bourlard1992}
$$
\|\phi\|_{H^{2+\alpha}(\Omega)} \leq C \|f\|_{H^{-2+\alpha}(\Omega)}.
$$
The sketches of the threshold $\alpha_0$ in terms of the interior angle $\omega$ is shown in Figure \ref{Regularity}a. In Figure \ref{Regularity}a, as a comparison we also show $\beta_0 = \frac{\pi}{\omega}$, which is the threshold of the characteristic equation for the Poisson equation with homogeneous Dirichlet boundary condition in the same polygonal domain. In Figure \ref{Regularity}b, we show the difference $\alpha_0-\beta_0$ in term of the interior angle $\omega$, from which we find $\beta_0+1 > \alpha_0$, when the largest interior angle $\omega$ satisfies $\frac{\pi}{3}<\omega'<\omega<\pi$ for some $\omega'$. In Table \ref{alpha0tab}, we present some numerical values of $\alpha_0$ in terms of different interior angles $\omega$.

\begin{figure}
\centering
\subfigure[]{\includegraphics[width=0.45\textwidth]{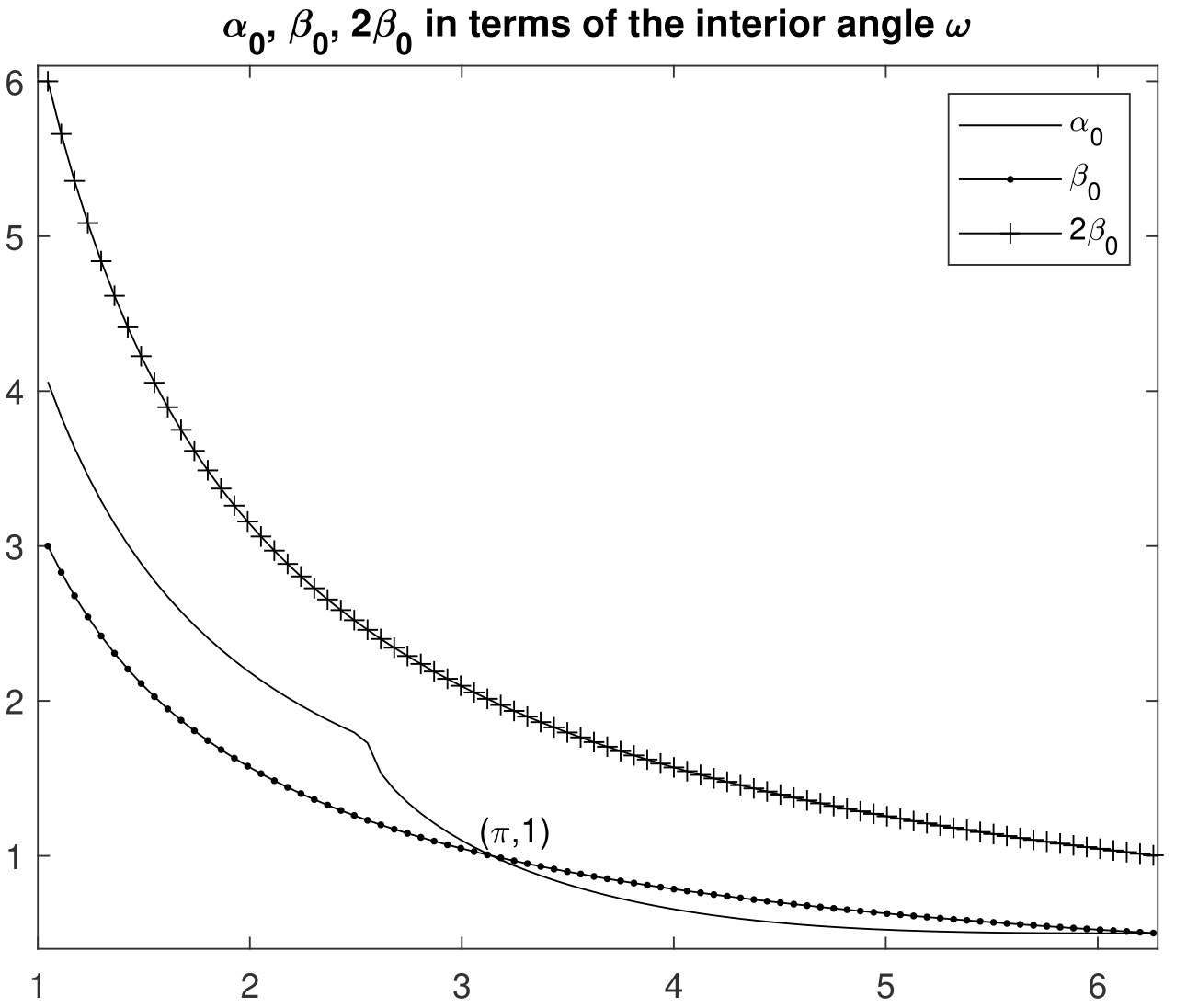}}
\subfigure[]{\includegraphics[width=0.49\textwidth]{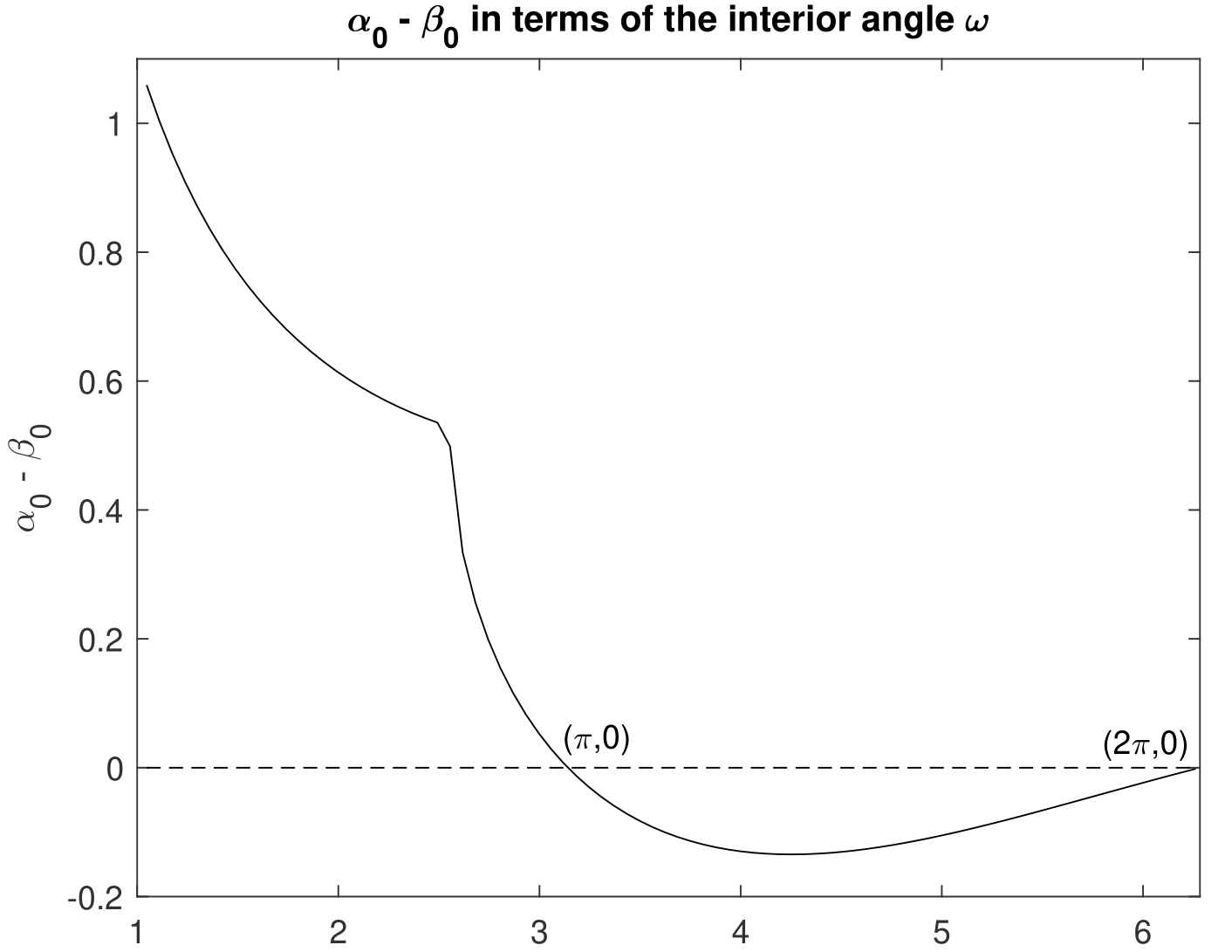}}
\caption{(a) $\alpha_0, \beta_0$ in terms of $\omega$; (b) $\alpha_0-\beta_0$ in terms of $\omega$. }\label{Regularity}
\end{figure}

\begin{table}[!htbp]\tabcolsep0.04in
\caption{Some numerical values of $\alpha_0$ in terms of different interior angles $\omega$.}
\begin{tabular}[c]{||c|c||c|c||}
\hline
$\omega (<\pi)$  & $\alpha_0 \approx$ & $\omega(>\pi)$  & $\alpha_0 \approx$ \\
\hline
$\frac{\pi}{3}$ & 4.059329012151345 & $\frac{7\pi}{6}$ & 0.751974545407645 \\
\hline
$\frac{\pi}{2}$ & 2.739593356324596 & $\frac{6\pi}{5}$ & 0.717799308407060 \\
\hline
$\frac{2\pi}{3}$ & 2.094139108847751 & $\frac{5\pi}{4}$ & 0.673583432221468 \\
\hline
$\frac{3\pi}{4}$ & 1.885371778114173 & $\frac{4\pi}{3}$ & 0.615731059491289 \\
\hline
$\frac{5\pi}{6}$ & 1.533859976323978 & $\frac{3\pi}{2}$ & 0.544483736993940 \\
\hline
$\frac{11\pi}{12}$ & 1.200631594651580 & $\frac{7\pi}{4}$ & 0.505009699452470\\
\hline
\end{tabular}\label{alpha0tab}
\end{table}

For $\alpha_0$ and $\beta_0$, we have the following result from \cite[Theorem 7.1.1]{kozlov2001}.
\begin{lemma}\label{betaalpha}
If  $\omega\in (0,\pi)$, it follows that $\alpha_0$ in (\ref{alpha0}) satisfies
\be\label{betaalphapi}
\beta_0 < \alpha_0 < 2\beta_0,
\ee
and if $\omega\in(\pi,2\pi)$, it follows
\be\label{betaalpha2pi}
\frac{1}{2}<\alpha_0<\beta_0,
\ee
where $\beta_0=\frac{\pi}{\omega}$.
\end{lemma}
By Lemma \ref{betaalpha}, when $\omega <\pi$, it follows
\be\label{alpha2beta}
\frac{1}{2} <\frac{\beta_0}{\alpha_0}<1,
\ee
and when $\omega>\pi$, it follows
\be\label{alpha2beta2}
\frac{\beta_0}{\alpha_0}>1.
\ee

\subsection{Decoupled formulation of the biharmonic problem}

It is known that solving high order problems numerically, such as  (\ref{eqnbi}), is much harder than solving the lower order problem. To decouple (\ref{eqnbi}) into lower order problems,
we first introduce a steady-state Stokes problem
\be\label{stokes}
\bal
- \Delta \mathbf{u} + \nabla p = & \mathbf{F} \quad \text{in } \Omega,\\
\text{div } \mathbf{u} = & 0 \quad \text{in } \Omega,\\
\mathbf{u} = & 0  \quad \text{on } \pa\Omega,
\eal
\ee
where $\mathbf{u} = (u_1, u_2)^T$ is the velocity field of an incompressible fluid motion, $p$ is the associated pressure.  The source term $\mathbf{F} = (f_1, f_2)^T$ satisfies
\be\label{curlFf}
\text{curl } \mathbf{F} = \frac{\partial f_2}{\partial x_1} - \frac{\partial f_1}{\partial x_2} = f,
\ee
for $f$ given in (\ref{eqnbi}).

The weak formulation of the Stokes equations  (\ref{stokes}) is to find $\mathbf{u} \in [H_0^1(\Omega)]^2$ and $p\in L_0^2(\Omega)$ such that
\be\label{stokesweak}
\bal
(\nabla \mathbf{u}, \nabla \mathbf{v}) - (\text{div } \mathbf{v},p) = & \langle\mathbf{F}, \mathbf{v} \rangle \quad \forall \mathbf{v} \in [H_0^1(\Omega)]^2,\\
-(\text{div } \mathbf{u}, q) = & 0 \quad \forall q \in L_0^2(\Omega),
\eal
\ee
where
$$
L_0^2(\Omega) = \{q \in L^2(\Omega),  \int_\Omega q dx =0\}.
$$

For the bilinear forms in weak formulation (\ref{stokesweak}), we have the following Ladyzhenskaya-Babuska-Breezi (LBB) or inf-sup conditions,
\be\label{skinfsup}
\bal
& \inf_{q\in L^2_0(\Omega)} \sup_{\mathbf{v}\in [H_0^1(\Omega)]^2}  \frac{-(\text{div } \mathbf{v},q)}{\|\mathbf{v}\|_{[H_0^1(\Omega)]^2}\|q\|} \geq \gamma_1>0,\\
& \inf_{\mathbf{u}\in [H_0^1(\Omega)]^2} \sup_{\mathbf{v}\in [H_0^1(\Omega)]^2} \frac{(\nabla \mathbf{u}, \nabla \mathbf{v})}{\|\mathbf{u}\|_{[H_0^1(\Omega)]^2}\|\mathbf{v}\|_{[H_0^1(\Omega)]^2}}
\geq \gamma_2 >0,
\eal
\ee
and the boundedness
\be\label{skubb}
\bal
& (\nabla \mathbf{u}, \nabla \mathbf{v}) \leq C_1 \|\mathbf{u}\|_{[H_0^1(\Omega)]^2}\|\mathbf{v}\|_{[H_0^1(\Omega)]^2},\\
& (\text{div } \mathbf{v},q) \leq C_2\|\mathbf{v}\|_{[H_0^1(\Omega)]^2} \|q\|,
\eal
\ee
where $\gamma_1, \gamma_2, C_1, C_2$ are constants.

Given that $\mathbf{F} \in [H^{-1}(\Omega)]^2$, under conditions (\ref{skinfsup}) and (\ref{skubb}), the weak formulation (\ref{stokesweak}) admits a unique solution $(\mathbf{u}, p) \in [H_0^1(\Omega)]^2 \times L_0^2(\Omega)$ (see, e.g. \cite{ladyzhenskaya1969,Temam1977,girault1979}).
Moreover, if $\mathbf{F} \in [H^{-1+\alpha}(\Omega)]^2$ for $\alpha<\alpha_0$, the Stokes problem holds the regularity estimate \cite{bernardi1981, Grisvard1992, pyo2015},
\be\label{stokesreg}
\|\mathbf{u}\|_{[H^{1+\alpha}(\Omega)]^2} + \|p\|_{H^{\alpha}(\Omega)} \leq C \|\mathbf{F}\|_{[H^{-1+\alpha}(\Omega)]^2}.
\ee

Next, we introduce
$$
\mathbf H(\text{curl}; \Omega) := \{ \mathbf{F} \in [L^2(\Omega)]^2 : \text{curl }\mathbf{F} \in L^2(\Omega)\}.
$$
For given $f$, we introduce a Poisson problem
\begin{equation}\label{eqnpoissonf}
-\Delta w  = f \quad \text{in }\Omega, \quad w=0 \quad {\rm{on}} \ \pa\Omega.
\end{equation}
Then we have the following results.
\begin{lem}\label{Fbblem}
For $f\in H^{-1}(\Omega)$, assume that $w$ is the solution of (\ref{eqnpoissonf}), then it follows
\be\label{Fthcurl}
\mathbf{F} = \textbf{curl }w \in \mathbf H(\text{curl}; \Omega) \subset [L^2(\Omega)]^2
\ee
satisfies (\ref{curlFf}) and
\be\label{Fbdd}
\|\mathbf F\|_{[L^2(\Omega)]^2} \leq C\|f\|_{H^{-1}(\Omega)}.
\ee
\end{lem}
\begin{proof}
It is easy to verify that $\mathbf{F}$ satisfies (\ref{curlFf}), i.e.,
$$
\text{curl } \mathbf F =  \text{curl }( \textbf{curl }w) = - \Delta w = f.
$$
The Poisson problem (\ref{eqnpoissonf}) admits a unique $w\in H_0^1(\Omega)$, which satisfies
$$\|w\|_{H^1(\Omega)} \leq C\|f\|_{H^{-1}(\Omega)}.$$
Note that $\|\textbf{curl }w\|_{[L^2(\Omega)]^2} = |w|_{H^1(\Omega)}$, so we have
\be\label{curlfbdd}
\|\textbf{curl }w\|_{[L^2(\Omega)]^2} \leq C\|f\|_{H^{-1}(\Omega)}.
\ee
Thus, the estimate (\ref{Fbdd}) holds.
\end{proof}
In a polygonal domain $\Omega$, if $f \in H^l(\Omega)$ for $l \geq -1$, the regularity estimate \cite{Grisvard1985, Grisvard1992} for the Poisson problem (\ref{eqnpoissonf}) gives \be\label{regpoisson}
\|w\|_{H^{\min\{1+\beta,l+2\}}(\Omega)} \leq \|f\|_{H^l(\Omega)},
\ee
where $\beta<\beta_0=\frac{\pi}{\omega}$ with $\omega$ being the largest interior angles of $\Omega$. So for $\mathbf{F}$ obtained from (\ref{Fthcurl}), we have $\mathbf{F} \in [H^{\min\{\beta,l+1\}}(\Omega)]^2$. 


Since no boundary data is enforced to (\ref{curlFf}), so $\mathbf{F}$ obtained through (\ref{curlFf}) is not unique. 
Assume that $\mathbf{F}_0 \in [L^2(\Omega)]^2$ is a solution of (\ref{curlFf}), then it follows that
$$
\mathbf{F} = \mathbf{F}_0 + \nabla q \quad \forall q\in H^1(\Omega),
$$
is also a solution of (\ref{curlFf}) in $[L^2(\Omega)]^2$, 
since for $q\in H^1(\Omega)$, we have
$
(\text{curl }\nabla q) \equiv 0.
$

In addition to obtaining $\mathbf{F}$ by Lemma \ref{Fbblem}, we provide another way to obtain $\mathbf{F}$.
\begin{lemma}\label{Fint}
Assume that $f \in L^2(\Omega)$.\\
(i) For any fixed $x_2$, if $f(\xi, x_2) $ is integrable on $[c_1,x_1]$ for some constant $c_1$, then
\be\label{F1}
\mathbf{F} = \left[0, \ \int_{c_1}^{x_1} f(\xi, x_2) d\xi \right]^T
\ee
satisfies (\ref{curlFf}).\\
(ii) Similarly, for any fixed $x_1$, if  $f(x_1, \zeta)$ is integrable on $[c_2,x_2]$ for some constant $c_2$, then
\be\label{F2}
\mathbf{F} = \left[-\int_{c_2}^{x_2} f(x_1, \zeta) d\zeta, \ 0 \right]^T
\ee
also satisfies (\ref{curlFf}).\\
(iii) If both $f(\xi, x_2)$ and $f(x_1, \zeta)$ are integrable, then for any constant $\eta$,
\be\label{F3}
\mathbf{F} = \left[-\eta\int_{c_2}^{x_2} f(x_1, \zeta) d\zeta, \ (1-\eta)\int_{c_1}^{x_1} f(\xi, x_2) d\xi \right]^T
\ee
also satisfies (\ref{curlFf}).
\end{lemma}


It is obvious that $\mathbf{F}$ obtained from Lemma \ref{Fint} satisfies $\mathbf{F} \in [L^2(\Omega)]^2$. For all these $\mathbf{F} \in [L^2(\Omega)]^2$ satisfying (\ref{curlFf}), we have the following result.

\begin{theorem}\label{Stokeindep}
Assume that $\mathbf{F}_l \in [L^2(\Omega)]^2$, $l=1,2$ both satisfy (\ref{curlFf}).
Let $(\mathbf{u}_l, p_l)$ be solutions of (\ref{stokes}) or (\ref{stokesweak}) corresponding to $\mathbf{F}_l$, then it follows that
\be
\bal
\mathbf{u}_1 = & \mathbf{u}_2 \quad \text{in } [H_0^1(\Omega)]^2 \cap [H^{1+\alpha}(\Omega)]^2,\\
p_1 = & p_2 + q \quad \text{in } L^2_0(\Omega)\cap H^{\alpha}(\Omega),
\eal
\ee
where $q \in L_0^2(\Omega)\cap H^1(\Omega)$ satisfies $\nabla q = \mathbf{F}_1 - \mathbf{F}_2$.
\end{theorem}
\begin{proof}
We take $\bar{\mathbf{F}} = \mathbf{F}_1 - \mathbf{F}_2\in [L^2(\Omega)]^2$, then by Helmholtz decomposition \cite{girault1979}, there exist a stream-function $\psi$ and a potential-function $q \in H^1(\Omega)$ uniquely up to a constant such that
\be\label{Helmdep}
\bar{\mathbf{F}} = \nabla q + \textbf{curl }\psi,
\ee
and
\be\label{Helmbc}
(\bar{\mathbf{F}}-\nabla q)\cdot\mathbf{n} = (\textbf{curl }\psi)\cdot\mathbf{n} = 0 \quad \text{in } H^{-\frac{1}{2}}(\partial \Omega).
\ee
From (\ref{Helmbc}), we have
\bes
\frac{\partial \psi}{\partial \mathbf{\tau}} = (\textbf{curl }\psi)\cdot\mathbf{n} = 0\quad \text{in } H^{-\frac{1}{2}}(\partial \Omega),
\ees
where $\tau$ is the unit tangential vector on $\partial \Omega$, thus we have
\be\label{HelmDbc}
\psi = C_0 \quad \text{in } H^{\frac{1}{2}}(\partial \Omega),
\ee
where $C_0$ is a constant.
Take $\text{curl }$ on (\ref{Helmdep}),  we have
\be\label{poicurl}
-\Delta \psi = \text{curl }(\textbf{curl } \psi) = \text{curl }  \bar{\mathbf{F}} = 0,
\ee
where the last equality is based on the fact that $\mathbf{F}_1, \mathbf{F}_2$ satisfy (\ref{curlFf}).
By the Lax-Milgram Theorem, the Poisson equation (\ref{poicurl}) with the boundary condition (\ref{HelmDbc}) admits a unique solution $\psi=C_0$ in $H^1(\Omega)$. Therefore, the decomposition (\ref{Helmdep}) is equivalent to
\be\label{Helmdep1}
\bar{\mathbf{F}} = \nabla q.
\ee
Let $\bar{\mathbf{u}} = \mathbf{u}_1 - \mathbf{u}_2$ and $\bar{p} = p_1-p_2$, then $(\bar{\mathbf{u}}, \bar{p})$ satisfies
\be\label{stokesbar}
\bal
- \Delta \bar{\mathbf{u}} + \nabla (\bar{p}-q) = & 0 \quad \text{in } \Omega,\\
\text{div } \bar{\mathbf{u}} = & 0 \quad \text{in } \Omega,\\
\bar{\mathbf{u}} = & 0  \quad \text{on } \pa\Omega.
\eal
\ee
By the regularity of the Stokes problem (\ref{stokesbar}),  the conclusion holds.
\end{proof}

\begin{lemma}\label{coro1}
Assume that the source term $\mathbf{F}\in [L^2(\Omega)]^2$ of the Stokes problem (\ref{stokes}) is any vector function determined by $f \in H^{-1}(\Omega)$ satisfying (\ref{curlFf}), then (\ref{stokes}) admits a unique solution $\mathbf{u} \in [H^{1+\alpha}(\Omega)]^2$ and satisfies
\be
\|\mathbf{u}\|_{[H^{1+\alpha}(\Omega)]^2} \leq C\|f\|_{H^{-1}(\Omega)}.
\ee
\end{lemma}
\begin{proof}
Given $f \in H^{-1}(\Omega)$, we can always find a vector function $\mathbf{F}_0 \in [L^2(\Omega)]^2$ following Lemma \ref{Fbblem} such that the corresponding Stokes problem (\ref{stokes}) admits a unique solution $\mathbf{u}_0 \in [H^{1+\alpha}(\Omega)]^2$ satisfying
\be
\|\mathbf{u}_0\|_{[H^{1+\alpha}(\Omega)]^2} \leq C \|\mathbf{F}_0\|_{[H^{-1+\alpha}(\Omega)]^2} \leq C \|\mathbf{F}_0\|_{[L^2(\Omega)]^2} \leq C\|f\|_{H^{-1}(\Omega)}.
\ee
For any source term $\mathbf{F}$ also satisfying (\ref{Fthcurl}), it follows by Theorem \ref{Stokeindep} that the corresponding solution $\mathbf{u}=\mathbf{u}_0$, so the conclusion holds.
\end{proof}

To show the connection of the solution $\mathbf{u}$ to the Stokes problem (\ref{stokes}) with the biharmonic problem (\ref{eqnbi}), we introduce the following result from \cite[Theorem 3.1]{girault1986}.
\begin{lem}\label{scurlv}
A function $\mathbf{v} \in [H^m(\Omega)]^2$ for integer $m\geq 0$ satisfies
$$
\text{div } \mathbf{v} = 0, \quad \langle \mathbf{v} \cdot \mathbf{n}, 1 \rangle|_{\partial \Omega} =0,
$$
then there exists a stream function $\psi \in H^{m+1}(\Omega)$ uniquely up to an additive constant satisfying
$$
\mathbf{v} = \textbf{curl } \psi.
$$
\end{lem}

Since $\mathbf{u} \in [H_0^1(\Omega)]^2 \cap [H^{1+\alpha}(\Omega)]^2$ and
$\text{div } \mathbf{u} = 0$, so we have by Lemma \ref{scurlv} that there exists $\bar{\phi} \in  H^{2}(\Omega)$ uniquely up to an additive constant
satisfying
\be\label{utophi}
(u_1, u_2)^T = \mathbf{u} = \textbf{curl }\bar{\phi} = (\bar{\phi}_{x_2}, -\bar{\phi}_{x_1})^T,
\ee
which further implies $|\nabla \bar\phi| \in H^{1+\alpha}(\Omega)$, thus we have
\be\label{phiH2}
\bar{\phi} \in  H^{2+\alpha}(\Omega).
\ee

\begin{lem}\label{barphibi0}
There exists a unique
\be\label{phiH02}
\bar{\phi} \in H_0^2(\Omega)\cap  H^{2+\alpha}(\Omega).
\ee
satisfying (\ref{utophi}).
\end{lem}
\begin{proof}
By calculation,
\bes
\bar{\phi}_\mathbf{\tau} = \textbf{curl }\bar{\phi} \cdot \mathbf{n} = \mathbf{u} \cdot \mathbf{n} = 0,
\ees
where $\mathbf{\tau}$ is the unit tangent to $\partial \Omega$, thus it follows
$$
\bar{\phi} = \text{constant} \quad \text{on } \partial \Omega.
$$
Without loss of generality, we can take
\be\label{barphibc}
\bar{\phi} = 0 \quad \text{on } \partial \Omega.
\ee
From (\ref{utophi}), we also have
\be\label{gradphi}
\nabla \bar{\phi} = (\bar{\phi}_{x_1}, \bar{\phi}_{x_2})^T = (-u_2, u_1)^T = \mathbf{0}\quad \text{on } \pa\Omega.
\ee
Thus, the conclusion (\ref{phiH02}) follows from (\ref{barphibc}), (\ref{gradphi}) and (\ref{phiH2}).
\end{proof}

Instead of solving for $\bar{\phi} \in H_0^2(\Omega)\cap  H^{2+\alpha}(\Omega)$ from (\ref{utophi}) directly, we apply the operator $\text{curl}$ on (\ref{utophi}) to obtain the following Poisson problem
\begin{equation}\label{eqnpoisson1}
-\Delta \bar{\phi}  = \text{curl }\mathbf{u} \quad \text{in }\Omega \quad \bar{\phi}=0 \quad {\rm{on}} \ \pa\Omega.
\end{equation}
The weak formulation of (\ref{eqnpoisson1}) is to find $\bar{\phi} \in H_0^1(\Omega)$, such that
\be\label{poissonweak1}
(\nabla \bar{\phi}, \nabla \psi) = (\text{curl } \mathbf{u}, \psi) \quad \forall \psi \in H_0^1(\Omega).
\ee
Since $(\text{curl } \mathbf{u} )\in L^2(\Omega)$, so we have by the Lax-Milgram Theorem that (\ref{poissonweak1}) admits a unique solution $\bar{\phi} \in H_0^1(\Omega)$.

\begin{lem}
The Poisson problem (\ref{eqnpoisson1}) admits a unique solution $\bar{\phi} \in H_0^2(\Omega)\cap  H^{2+\alpha}(\Omega)$.
\end{lem}
\begin{proof}
Since $\bar{\phi} \in H_0^2(\Omega)\cap  H^{2+\alpha}(\Omega) \subset H_0^1(\Omega)$ is a solution of (\ref{utophi}), so it is also a solution of the Poisson problem (\ref{eqnpoisson1}). By the uniqueness of the solution of (\ref{eqnpoisson1}) in $H_0^1(\Omega)$, the conclusion holds.
\end{proof}

\begin{lem} \label{barphibi}
The solution $\bar{\phi} \in H_0^2(\Omega)\cap  H^{2+\alpha}(\Omega)$ obtained through (\ref{utophi}) or the Poisson problem (\ref{eqnpoisson1})
satisfies the biharmonic problem
\begin{equation}\label{eqnbi2}
\Delta^2 \bar{\phi} = \text{curl } \mathbf{F} =f  \quad \text{in }\Omega, \qquad \quad \bar{\phi}=0 \quad {\rm{and}} \quad \partial_\mathbf{n} \bar{\phi}=0 \quad {\rm{on}} \ \pa\Omega.
\end{equation}
\end{lem}
\begin{proof}
Following (\ref{utophi}), we replace $\mathbf{u}$ by $\textbf{curl }\bar{\phi}$ in (\ref{stokes}) and obtain
\begin{subequations}\label{stokes2}
\begin{align}
-\Delta (\bar{\phi}_{x_2}) + p_{x_1} =&  f_1 \quad \text{in }\Omega, \\
-\Delta (-\bar{\phi}_{x_1}) + p_{x_2} =&  f_2 \quad \text{in }\Omega.
\end{align}
\end{subequations}
Applying differential operators $-\frac{\partial}{\partial x_2}$ and $\frac{\partial}{\partial x_1}$ to (\ref{stokes2}a) and (\ref{stokes2}b), respectively, and taking the summation lead to the conclusion.
\end{proof}

From Lemma \ref{barphibi}, we find that $\bar{\phi}$ in (\ref{eqnbi2}) satisfies exactly the same problem as $\phi$ in (\ref{eqnbi}) in the following sense,
\be
\phi = \bar{\phi} \quad \text{in } H_0^2(\Omega)\cap  H^{2+\alpha}(\Omega).
\ee
From now on, we will use $\phi$ to replace the notation $\bar{\phi}$. Thus the Poisson problem (\ref{eqnpoisson1}) is equivalent to
\begin{equation}\label{eqnpoisson}
-\Delta \phi  = \text{curl }\mathbf{u} \quad \text{in }\Omega, \quad \phi=0 \quad {\rm{on}} \ \pa\Omega.
\end{equation}
The weak formulation (\ref{poissonweak}) is equivalent to $\phi \in H_0^1(\Omega)$ satisfying
\be\label{poissonweak}
(\nabla \phi, \nabla \psi) = (\text{curl } \mathbf{u}, \psi) \quad \forall \psi \in H_0^1(\Omega).
\ee
By the regularity of the Poisson problem (\ref{eqnpoisson}) and Lemma \ref{coro1}, we have that
\be
\|\phi\|_{H^{2+\alpha}(\Omega)} \leq C\|\text{curl }\mathbf{u}\|_{H^{\alpha}(\Omega)}\leq C\|\nabla \mathbf{u}\|_{[H^{\alpha}(\Omega)]^2} \leq C\|\mathbf{u}\|_{[H^{1+\alpha}(\Omega)]^2}\leq C\|f\|_{H^{-1}(\Omega)}.
\ee

In summary, we can obtain the solution $\phi$ of the biharmonic problem (\ref{eqnbi}) by solving the lower order problems in the following steps,
\begin{enumerate}
    \item Choose an appropriate $\mathbf{F}$ by Lemma \ref{Fbblem} or Lemma \ref{Fint};
    \item Solve $\mathbf{u}$ from the Stokes problem (\ref{stokes});
    \item Solve $\phi$ from the Poisson problem (\ref{eqnpoisson}).
\end{enumerate}

\section{The finite element method and error estimates}\label{sec-3}

In this section, we propose a linear $C^0$ finite element method for solving the biharmonic problem (\ref{eqnbi}) based on the results in the previous section.

\subsection{The finite element algorithm}\label{fem}
Let $\maT_n$ be a triangulation of $\Omega$ with shape-regular triangles and let $\mathcal P_k(\maT_n)$ be the $C^0$ Lagrange finite element space associated with $\maT_n$,
\be\label{eqn.fems}
\mathcal P_k(\maT_n):=\{v\in C^0(\Omega): \ v|_T\in P_k, \ \forall T \in \maT_n\},
\ee
where $P_k$ is the space of polynomials of degree no more than $k$.
Further, we introduce the following specific $C^0$ Lagrange finite element spaces associated with $\maT_n$,
\be\label{eqn.space}
\bal
V_n^k:=& \mathcal P_k(\maT_n) \cap H_0^1(\Omega),\\
S_{n}^{k}:=& \mathcal P_k(\maT_n) \cap L_0^2(\Omega),
\eal
\ee
and the bubble function space
$$
B_n^3:= \{v\in C^0(\Omega): \ v|_T\in \text{span}\{\lambda_1\lambda_2\lambda_3\} , \ \forall T \in \maT_n\},
$$
where $\lambda_i$, $i=1,2,3$ are the barycentric coordinates on $T$.

We define the finite element solution of  the biharmonic problem (\ref{eqnbi}) by utilizing the decomposition in the previous section as follows.

\begin{algorithm}\label{femalg+}
For any $f\in H^{-1}(\Omega)$ and $k \geq 1$, we consider the following steps.
\begin{itemize}
\item{Step 1.} 
Find $w_n \in V_n^{k}$ of the Poisson equation
\be\label{poissonfem0}
(\nabla w_n, \nabla \psi) = (f, \psi) \quad \forall \psi \in V_{n}^{k},
\ee
then take $\mathbf{F}_n=\textbf{curl }w_n$.
\item{Step 2.}  
If $k=1$, we find the Mini element approximation 
$\mathbf{u}_n \times p_n \in [V_n^1 \oplus B_n^3]^2 \times S_{n}^{1}$ of the Stokes equation
\be\label{stokesfem+}
\bal
(\nabla \mathbf{u}_n, \nabla \mathbf{v}) - (p_n, \text{div } \mathbf{v}) = & \langle\mathbf{F}_n, \mathbf{v} \rangle \quad \forall \mathbf{v} \in [V_n^1 \oplus B_n^3]^2,\\
-(\text{div } \mathbf{u}_n, q) = & 0 \quad \forall q \in S_{n}^{1}.
\eal
\ee
If $k\geq 2$, we find the Taylor-Hood element solution $\mathbf{u}_n \times p_n \in [V_n^k]^2 \times S_{n}^{k-1}$ of the Stokes equation
\be\label{stokesfem}
\bal
(\nabla \mathbf{u}_n, \nabla \mathbf{v}) - (p_n, \text{div } \mathbf{v}) = & \langle\mathbf{F}_n, \mathbf{v} \rangle \quad \forall \mathbf{v} \in [V_n^k]^2,\\
-(\text{div } \mathbf{u}_n, q) = & 0 \quad \forall q \in S_{n}^{k-1}.
\eal
\ee
\item{Step 3.} Find the finite element solution $\phi_n\in V_{n}^{k}$ of the Poisson equation
\be\label{poissonfem}
(\nabla \phi_n, \nabla \psi) = (\text{curl }\mathbf{u}_n, \psi) \quad \forall \psi \in V_{n}^{k}.
\ee
\end{itemize}
\end{algorithm}

For some source term $f$, Algorithm \ref{femalg+} could be updated as follows.

\begin{algorithm}\label{femalg}
If $f\in L^2(\Omega)$ satisfies the condition of Lemma \ref{Fint}, we do the following updates.
\begin{itemize}
\item{Step 1.} This step is the same as Algorithm \ref{femalg+} Step 2 with $\langle\mathbf{F}_n, \mathbf{v} \rangle$ replaced by $\langle\mathbf{F}, \mathbf{v} \rangle$, where $\mathbf{F}$ obtained following Lemma \ref{Fint}.
\item{Step 2.} The same as Algorithm \ref{femalg+} Step 3.
\end{itemize}
\end{algorithm}

The finite element approximations for the Poisson problems in both Algorithm \ref{femalg+} and Algorithm \ref{femalg} are well defined by the Lax-Milgram Theorem. We take the Mini element method \cite{arnold1984} and the Taylor-Hood element method \cite{verfurth1984, Brezzi1991TH} for solving the Stokes problem, other methods could also be used. 
The Mini element approximations or the Taylor-Hood element approximations are well defined,
if i) the bilinear forms in Mini element method satisfy the following LBB condition,
\begin{subequations}\label{skinfsupMini}
\begin{align}
& \inf_{q\in S_{n}^{1}} \sup_{\mathbf{v}\in [V_n^1 \oplus B_n^3]^2}  \frac{-(\text{div } \mathbf{v},q)}{\|\mathbf{v}\|_{[H_0^1(\Omega)]^2}\|q\|} \geq \tilde \gamma_1>0,\\
& \inf_{\mathbf{u}\in [V_n^1 \oplus B_n^3]^2} \sup_{\mathbf{v}\in [V_n^1 \oplus B_n^3]^2} \frac{(\nabla \mathbf{u}, \nabla \mathbf{v})}{\|\mathbf{u}\|_{[H_0^1(\Omega)]^2}\|\mathbf{v}\|_{[H_0^1(\Omega)]^2}} \geq \tilde \gamma_2 >0,
\end{align}
\end{subequations}
and these in Taylor-Hood method satisfies the following LBB condition,
\begin{subequations}\label{skinfsupTH}
\begin{align}
& \inf_{q\in S_{n}^{k-1}} \sup_{\mathbf{v}\in [V_n^k]^2}  \frac{-(\text{div } \mathbf{v},q)}{\|\mathbf{v}\|_{[H_0^1(\Omega)]^2}\|q\|} \geq \tilde\gamma_1>0,\\
& \inf_{\mathbf{u}\in [V_n^k]^2} \sup_{\mathbf{v}\in [V_n^k]^2} \frac{(\nabla \mathbf{u}, \nabla \mathbf{v})}{\|\mathbf{u}\|_{[H_0^1(\Omega)]^2}\|\mathbf{v}\|_{[H_0^1(\Omega)]^2}} \geq \tilde\gamma_2 >0,
\end{align}
\end{subequations}
where $\tilde \gamma_1, \tilde \gamma_2$ are some constants; ii) the bilinear forms are bounded.

\begin{remark}
The Algorithm \ref{femalg+} is similar to method in \cite{brezzi1986, brezzi1991, gallistl2017} for fourth order problems in smooth domains or convex polygonal domains. Error estimates were derived in \cite{brezzi1991} 
for $P_k$ element approximations 
in a convex polygonal domain by assuming that the solutions are smooth enough. In \cite{gallistl2017}, error estimates  were given 
for $P_1$ element approximations 
based on the regularity assumption $|\nabla \mathbf{u}|, |\mathbf{F}|, p, |\nabla\phi| \in H^{s+1}(\Omega)$ for $-1<s\leq 0$. In this work, we carry out the error analysis of Algorithm \ref{femalg+} and Algorithm \ref{femalg} for biharmonic problem (\ref{eqnbi}) in both convex and non-convex polygonal domains for $P_k$ polynomials based on the regularity estimates obtained in Section \ref{sec-2}. 
\end{remark}

\subsection{Optimal error estimates on quasi-uniform meshes}

Suppose that the mesh $\maT_n$ consists of quasi-uniform triangles with size $h$.
The interpolation error estimate on $\maT_n$ (see e.g., \cite{Ciarlet}) for any $v \in H^\sigma(\Omega)$, $\sigma>1$,
\be\label{interr}
\| v - v_I \|_{H^\tau(\Omega)} \leq Ch^{\sigma-\tau}\|v\|_{H^\sigma(\Omega)},
\ee
where $\tau= 0, 1$ and $v_I\in V_n^k$ represents the nodal interpolation of $v$.

To make the analysis simple and clear, we 
we assume that $f\in H^{\max\{\alpha_0,\beta_0\}-1}(\Omega) \cap L^2(\Omega)$, where $\alpha_0$ given in (\ref{alpha0}), and $\beta_0=\frac{\pi}{\omega}$.
Since Algorithm \ref{femalg} involves only one Stokes equation and one Poisson equation, we first give the error estimate of Algorithm \ref{femalg}.

For $f\in H^{\max\{\alpha_0,\beta_0\}-1}(\Omega) \cap L^2(\Omega)$, if $\mathbf{F}$ is given by Lemma \ref{Fint}, we have $\mathbf{u} \in [H^{1+\alpha}(\Omega)]^2,  p \in H^{\alpha}(\Omega)$.
Note that the bilinear forms in the Mini element method ($k=1$) or Taylor-Hood element method ($k\geq 2$) satisfying the LBB condition  (\ref{skinfsupMini}) or (\ref{skinfsupTH}) on quasi-uniform meshes \cite{arnold1984, verfurth1984, Brezzi1991TH}, 
then the standard arguments for error estimate (see e.g., \cite{girault1986, stenberg1990, BS02}) give the following error estimate.

\begin{lem}\label{stokesestlem}
Let $(\mathbf{u}, p)$ be the solution of the Stokes problem (\ref{stokesweak}), and $(\mathbf{u}_n, p_n)$ be the Mini element solution ($k=1$) or Taylor-Hood element solution ($k\geq 2$) in Algorithm \ref{femalg} on quasi-uniform meshes, then it follows
\begin{subequations}\label{stokeserr}
\begin{align}
\|\mathbf{u}-\mathbf{u}_n\|_{[H^1(\Omega)]^2} + \|p-p_n\| \leq Ch^{\min\{k,\alpha\}},\\
\|\mathbf{u}-\mathbf{u}_n\|_{[L^2(\Omega)]^2} \leq Ch^{\min\{k+1, \alpha+1, 2\alpha\}},\\
\|\mathbf{u}-\mathbf{u}_n\|_{[H^{-1}(\Omega)]^2} \leq Ch^{\min\{2k,k+2, k+\alpha, \alpha+2, 2\alpha\}}.
\end{align}
\end{subequations}
\end{lem}
If the largest interior angle $\omega<\pi$, it follows  $\min\{\alpha+1, 2\alpha\} = \alpha+1$, and if $\omega>\pi$, we have $\min\{\alpha+1, 2\alpha\} = 2\alpha$.

If $\mathbf{F}$ is given by Lemma \ref{Fbblem}, then
for the Poisson equations (\ref{eqnpoissonf}) in a polygonal domain with $f\in H^{\max\{\alpha_0,\beta_0\}-1}(\Omega) \cap L^2(\Omega)$, the regularity estimate gives $v \in H^{1+\beta}(\Omega)$ for $\beta<\beta_0=\frac{\pi}{\omega}$ (see e.g., \cite{Grisvard1985, Grisvard1992}), it implies that 
$$
\mathbf{F}=\textbf{curl }w \in [ H^{\beta}(\Omega)]^2 \subset [H^{\alpha-1}(\Omega)]^2\cap [H^{\beta}(\Omega)]^2,
$$ 
where $\alpha<\alpha_0$. Therefore, we also have $\mathbf{u} \in [H^{\alpha+1}(\Omega)]^2\cap [H^{\beta+2}(\Omega)]^2,  p \in H^{\alpha}(\Omega)\cap H^{\beta+1}(\Omega)$, and $\phi \in H^{\alpha+2}(\Omega)\cap H^{\beta+3}(\Omega)$.
For the finite element approximations $w_n$ in (\ref{poissonfem0}),
the standard error estimate  \cite{Ciarlet} yields
\be\label{wh1err}
\|w-w_n\|_{H^1(\Omega)} \leq Ch^{\min\{k,\beta\}}, \quad \|w-w_n\|\leq Ch^{\min\{k+1, \beta+1, 2\beta
\} },
\ee
which implies that
\be\label{Ferrs}
\bal
\|\mathbf{F}-\mathbf{F}_n\|_{[L^2(\Omega)]^2} = &  \|\textbf{curl }w - \textbf{curl }w_n\|_{[L^2(\Omega)]^2} \leq \|w-w_n\|_{H^1(\Omega)} \leq Ch^{\min\{k,\beta\}},\\
\|\mathbf{F}-\mathbf{F}_n\|_{[H^{-1}(\Omega)]^2} = &  \|\textbf{curl }w - \textbf{curl }w_n\|_{[H^{-1}(\Omega)]^2} \leq C\|w-w_n\| \leq Ch^{\min\{k+1, \beta+1, 2\beta\} }.
\eal
\ee

For $\mathbf{F}_n$ in Algorithm \ref{femalg+}, we further have the following result.
\begin{lem}\label{fcurlgal}
If $\mathbf{F}_n$ is given by $\mathbf{F}_n=\textbf{curl }w_n$  in Step 1 of Algorithm \ref{femalg+}, and $\mathbf{F}=\textbf{curl }w$ is given in (\ref{Fthcurl}), then it follows
\be
\langle \mathbf{F}-\mathbf{F}_n, \textbf{curl }\psi \rangle = 0 \quad \forall \psi \in V_n^k.
\ee
\end{lem}
\begin{proof}
Subtract (\ref{poissonfem0}) from the weak formulation of (\ref{eqnpoissonf}), then we have the Galerkin orthogonality,
\be
\bal
(\nabla (w-w_n), \nabla \psi) = (w-w_n)_{x_1}\psi_{x_1}+(w-w_n)_{x_2}\psi_{x_2}  = 0 \quad \forall \psi \in V_{n}^{k},
\eal
\ee
which implies that
\be
\bal
\langle \mathbf{F}-\mathbf{F}_n, \textbf{curl }\psi \rangle = \langle \textbf{curl }(w-w_n), \textbf{curl }\psi \rangle = (w-w_n)_{x_2}\psi_{x_2} + (w-w_n)_{x_1}\psi_{x_1}  = 0.
\eal
\ee
\end{proof}
Next, we consider the error estimates of Taylor-Hood element approximations.
Subtract (\ref{stokesfem}) from (\ref{stokesweak}), we have the following equality,
\begin{subequations}\label{notgo}
\begin{align}
(\nabla (\mathbf{u}-\mathbf{u}_n), \nabla \mathbf{v}) - (\text{div } \mathbf{v}, p-p_n) = & \langle\mathbf{F}-\mathbf{F}_n, \mathbf{v} \rangle \quad \forall \mathbf{v} \in [V_n^k]^2,\\
-(\text{div } (\mathbf{u}-\mathbf{u}_n), q) = & 0 \quad \forall q \in S_{n}^{k-1}.
\end{align}
\end{subequations}

We introduce the adjoint problem of the
the Stokes equations  (\ref{stokes}),
\be\label{stokeadjoint}
\bal
- \Delta \mathbf{r} + \nabla s = & \mathbf g \quad \text{in } \Omega,\\
\text{div } \mathbf{r} = & 0 \quad \text{in } \Omega,\\
\mathbf{r} = & 0  \quad \text{on } \pa\Omega,
\eal
\ee
where $\mathbf g \in [H_0^{l}(\Omega)]^2$ for some $l = 0,1$. Here, the notation $H_0^0(\Omega):=H^0(\Omega)=L^2(\Omega)$.
The weak formulation of (\ref{stokeadjoint}) is to find $\mathbf{r} \in [H_0^1(\Omega)]^2$ and $s\in L_0^2(\Omega)$ such that
\begin{subequations}\label{stokesadweak}
\begin{align}
(\nabla \mathbf{r}, \nabla \mathbf{v}) - (\text{div } \mathbf{v},s) = & \langle\mathbf g, \mathbf{v} \rangle \quad \forall \mathbf{v} \in [H_0^1(\Omega)]^2,\\
-(\text{div } \mathbf{r}, q) = & 0 \quad \forall q \in L_0^2(\Omega).    
\end{align}
\end{subequations}
We have the following regularity result,
\be\label{stokesdualreg}
\|\mathbf{r}\|_{[H^{1+\min\{\alpha,l+1\}}(\Omega)]^2} + \|s\|_{H^{\min\{\alpha,l+1\}}(\Omega)} \leq C \|\mathbf g\|_{[H^{\min\{\alpha,l+1\}-1}(\Omega)]^2}\leq C\|\mathbf g\|_{[H^{l}(\Omega)]^2},
\ee
where $\alpha<\alpha_0$.

Note that $\mathbf{r}\in [H^{1+\min\{\alpha,l+1\}}(\Omega)]^2$ satisfying (\ref{stokesdualreg}) and (\ref{stokeadjoint}), we have by Lemma \ref{scurlv} that there exists $\psi \in H^{2+\min\{\alpha,l+1\}}(\Omega)\cap H_0^1(\Omega)$ such that
\be\label{rtopis}
\mathbf{r} = \textbf{curl }\psi.
\ee
We also have that $\|\psi\|_{H^{2+\min\{\alpha,l+1\}}(\Omega)} \leq C\|\mathbf{r}\|_{[H^{1+\min\{\alpha,l+1\}}(\Omega)]^2}$.

Let $(\mathbf r_n, s_n)$ be the Taylor-Hood solution of (\ref{stokeadjoint}), we have
\begin{subequations}\label{notgofem}
\begin{align}
(\nabla \mathbf{r}_n, \nabla \mathbf{v}) - (\text{div } \mathbf{v}, s_n) = & \langle \mathbf{g}, \mathbf{v} \rangle \quad \forall \mathbf{v} \in [V_n^k]^2,\\
-(\text{div } \mathbf{r}_n, q) = & 0 \quad \forall q \in S_{n}^{k-1}.
\end{align}
\end{subequations}
Subtracting (\ref{notgofem}) from (\ref{stokesadweak}), we have the Galerkin orthogonality,
\begin{subequations}\label{notgors}
\begin{align}
(\nabla (\mathbf{r}-\mathbf{r}_n), \nabla \mathbf{v}) - (\text{div } \mathbf{v}, s-s_n) = & 0 \quad \forall \mathbf{v} \in [V_n^k]^2,\\
-(\text{div } (\mathbf{r}-\mathbf{r}_n), q) = & 0 \quad \forall q \in S_{n}^{k-1}.
\end{align}
\end{subequations}

\begin{lem}\label{stokesestlem+}
Let $(\mathbf{u}, p)$ be the solution of the Stokes problem (\ref{stokesweak}), and $(\mathbf{u}_n, p_n)$ be the Mini element solution ($k=1$)  or Taylor-Hood element solution ($k\geq 2$) in Algorithm \ref{femalg+} on quasi-uniform meshes, then it follows the error estimates
\begin{subequations}\label{stokeserr+}
\begin{align}
\|\mathbf{u}-\mathbf{u}_n\|_{[H^1(\Omega)]^2} + \|p-p_n\| \leq Ch^{\min\{k,\alpha,\beta+1\}},\\
\|\mathbf{u}-\mathbf{u}_n\|_{[L^2(\Omega)]^2} \leq Ch^{\min\{k+1, \alpha+1, \beta+2, 2\alpha\}},\\
\|\mathbf{u}-\mathbf{u}_n\|_{[H^{-1}(\Omega)]^2} \leq Ch^{\min\{2k,k+2, k+\beta, \alpha+2, \beta+3, 2\alpha\}}.
\end{align}
\end{subequations}
\end{lem}
\begin{proof} We will only present the proof of error estimates of the Taylor-Hood element approximations, the error estimates of the Mini element approximations can be proved similarly.
By Lemma \ref{stokesTHbdd}, we have
\bes\bal
\|\mathbf{u}-\mathbf{u}_n\|_{[H^1(\Omega)]^2} + \|p-p_n\| \leq & C \left( \inf_{\mathbf{v} \in [V_n^{k}]^2 } \|\mathbf{u}-\mathbf{v}\|_{[H^1(\Omega)]^2} + \inf_{ q \in S_{n}^{k-1} } \|p-q\| + \|\mathbf{F}-\mathbf{F}_n\|_{[H^{-1}(\Omega)]^2} \right)\\
\leq & Ch^{\min\{k,\alpha,\beta+1\}} + Ch^{\min\{k+1, \beta+1, 2\beta
\} } \leq Ch^{\min\{k, \alpha,\beta+1\}}.
\eal\ees
We take $\mathbf{v} = \mathbf u -\mathbf u_n$, $q=p-p_n$ in (\ref{stokesadweak}), 
then we have 
\bes
\bal
& \|\mathbf u -\mathbf u_n\|_{[H^{-l}(\Omega)]^2} =  \sup_{g\in [H_0^l(\Omega)]^2} \frac{\langle\mathbf g, \mathbf u -\mathbf u_n \rangle}{\|\mathbf g\|_{[H^{l}(\Omega)]^2}},
\eal
\ees
where $l=0,1$.
Subtracting (\ref{notgofem}) with $\mathbf{v}=\mathbf{u}_n$ from (\ref{stokesadweak}) with $\mathbf{v}=\mathbf{u}$, we have
\bes
\bal
\langle\mathbf g, \mathbf u -\mathbf u_n \rangle = & (\nabla \mathbf{r}, \nabla \mathbf u )-(\nabla \mathbf{r}_n, \nabla \mathbf u_n ) - (\text{div } \mathbf u,s) + (\text{div } \mathbf u_n,s_n) \\
= & (\nabla (\mathbf{u}-\mathbf{u}_n), \nabla \mathbf r)
- (\text{div } (\mathbf u- \mathbf u_n),s) + (\nabla (\mathbf{r}-\mathbf{r}_n), \nabla \mathbf u_n )  -  (\text{div } \mathbf u_n,s-s_n) \\
= & (\nabla (\mathbf{u}-\mathbf{u}_n), \nabla \mathbf r)
- (\text{div } (\mathbf u- \mathbf u_n),s)
\eal
\ees
where we have used (\ref{notgors}a) with $\mathbf{v} = \mathbf{u}_n$ in the last equality. Subtracting (\ref{notgo}a) with $\mathbf{v}=\mathbf{r}_n$, (\ref{notgo}b) with $q=s_n$, and (\ref{stokesadweak}b) with $q=-(p-p_n)$ from the equation above, respectively, we have
\be\label{gunn}
\bal
\langle\mathbf g, \mathbf u -\mathbf u_n \rangle 
= &  (\nabla (\mathbf u -\mathbf u_n), \nabla (\mathbf{r} -\mathbf{r}_n ))  - (\text{div } (\mathbf{r}-\mathbf{r}_n), p-p_n) \\
& - (\text{div }(\mathbf u -\mathbf u_n),s - s_n) + \langle \mathbf{F}-\mathbf{F}_n, \mathbf{r}_n\rangle\\
:= & T_1+T_2+T_3+T_4,
\eal
\ee
For $\mathbf r-\mathbf r_n$ and $s-s_n$, we have the following estimate
\be\label{dualest0}
\bal
& \|\mathbf r -\mathbf r_n\|_{[H^1(\Omega)]^2} + \|s-s_n\| \leq C \left(\inf_{\mathbf{r}_I \in [V_n^k(\Omega)]^2}\|\mathbf r -\mathbf r_I\|_{[H^1(\Omega)]^2} + \inf_{s_I\in S_n^{k-1}} \|s-s_I\| \right)\\
& \leq Ch^{\min\{k,\alpha, l+1\}}
(\|\mathbf{r}\|_{[H^{\min\{k+1,\alpha+1, l+2\}}(\Omega)]^2} + \|s\|_{H^{\min\{k,\alpha, l+1\}}(\Omega)}).
\eal
\ee
Then we have the following estimates for terms $T_i$, $i=1,\cdots,4$ in (\ref{gunn}).
\bes
\bal
|T_1| \leq &  |\mathbf u -\mathbf u_n|_{[H^1(\Omega)]^2}|\mathbf r -\mathbf r_n|_{[H^1(\Omega)]^2} \leq Ch^{\min\{k, \alpha, \beta+1\} + \min\{k,\alpha, l+1\} } \|\mathbf{r}\|_{[H^{1+\min\{k,\alpha, l+1\}}(\Omega)]^2}\\
=& Ch^{\min\{2k, k+l+1, k+\alpha, \alpha+l+1,k+\beta+1,\alpha+\beta+1,\beta+l+2, 2\alpha\}}\|\mathbf{r}\|_{[H^{1+\min\{k,\alpha, l+1\}}(\Omega)]^2}.
\eal
\ees
\bes
\bal
|T_2| \leq  \|p-p_n\||\mathbf{r}-\mathbf{r}_n|_{[H^1(\Omega)]^2} \leq Ch^{\min\{2k, k+l+1, k+\alpha, \alpha+l+1,k+\beta+1,\alpha+\beta+1,\beta+l+2, 2\alpha\}}\|\mathbf{r}\|_{[H^{1+\min\{k,\alpha, l+1\}}(\Omega)]^2}.
\eal
\ees
\bes
\bal
|T_3| \leq  |\mathbf u -\mathbf u_n|_{[H^1(\Omega)]^2}|\|s-s_n\| \leq Ch^{\min\{2k, k+l+1, k+\alpha, \alpha+l+1,,k+\beta+1,\alpha+\beta+1,\beta+l+2, 2\alpha\}}\|s\|_{H^{\min\{k,\alpha, l+1\}}(\Omega)}.
\eal
\ees
For $T_4$, we have
\bes
\bal
T_4 = \langle \mathbf{F}-\mathbf{F}_n, \mathbf{r}_n-\mathbf{r}\rangle + \langle \mathbf{F}-\mathbf{F}_n, \mathbf{r}\rangle:=T_{41}+T_{42}.
\eal
\ees
\bes
\bal
|T_{41}| \leq &  \|\mathbf{F}-\mathbf{F}_n\|_{[H^{-1}(\Omega)]^2}\|\mathbf{r}-\mathbf{r}_n\|_{[H^{1}(\Omega)]^2} \leq Ch^{\min\{k+1, \beta+1, 2\beta\}+\min\{k,\alpha, l+1 \} }\|\mathbf{r}\|_{[H^{1+\min\{\alpha, l+1\}}(\Omega)]^2}\\
=& Ch^{\min\{2k+1,k+l+2, k+\beta+1,\beta+l+2, \alpha+\beta+1, 2\beta+\alpha\} }\|\mathbf{r}\|_{[H^{1+\min\{\alpha, l+1\}}(\Omega)]^2}.
\eal
\ees
By (\ref{rtopis}) and Lemma \ref{fcurlgal}, we have 
\bes
\bal
|T_{42}| = & |\langle \mathbf{F}-\mathbf{F}_n,  \textbf{curl }\psi\rangle|= |\langle \mathbf{F}-\mathbf{F}_n,  \textbf{curl }(\psi-\psi_I)\rangle|\leq \|\mathbf{F}-\mathbf{F}_n\|_{[L^2(\Omega)]^2}\|\textbf{curl }(\psi-\psi_I)\|_{[L^2(\Omega)]^2}\\
\leq & \|\mathbf{F}-\mathbf{F}_n\|_{[L^2(\Omega)]^2}\|\psi-\psi_I\|_{H^1(\Omega)} \leq Ch^{\min\{k, \beta\} + \min\{k, \alpha+1, l+2\} }\|\mathbf{r}\|_{[H^{1+\min\{\alpha,l+ 1\}}(\Omega)]^2}\\
\leq & Ch^{\min\{2k,k+l+2,k+\beta,\beta+l+2, \alpha+\beta+1\} }\|\mathbf{r}\|_{[H^{1+\min\{\alpha,l+ 1\}}(\Omega)]^2}.
\eal
\ees
where $\psi_I$ is the nodal interpolation of $\psi$.
It can be verified that
$$
|T_4| \leq |T_{41}|+|T_{42}| \leq  Ch^{\min\{2k,k+l+2,k+\beta,\beta+l+2, \alpha+\beta+1,\alpha+2\beta\} }\|\mathbf{r}\|_{[H^{1+\min\{\alpha,l+ 1\}}(\Omega)]^2}.
$$

By the regularity (\ref{stokesdualreg}) and the summation of  estimates $|T_i|$, $i=1, \cdots, 4$, the error estimate $\|\mathbf u - \mathbf u_n\|_{[H^{-l}(\Omega)]^2}$ holds.

\end{proof}

\begin{theorem}\label{phierrthm3.1}
Let $\phi_n\in V_{n}^{k}$ be the solution of finite element solution of (\ref{poissonfem}) from Algorithm \ref{femalg} or Algorithm \ref{femalg+}, and $\phi$ is the solution of the biharmonic problem (\ref{eqnbi}), if $\phi_n$ is the solution of Algorithm \ref{femalg}, then we have
\be\label{phiH1err}
\|\phi-\phi_n\|_{H^1(\Omega)} \leq Ch^{\min\{k, \alpha+1,2\alpha\}};
\ee
if it is the solution of of Algorithm \ref{femalg+}, then we have
\be\label{phiH1err+}
\|\phi-\phi_n\|_{H^1(\Omega)} \leq Ch^{\min\{k, \beta+2, \alpha+1,2\alpha\}}.
\ee
\end{theorem}
\begin{proof}
Subtracting (\ref{poissonfem}) from (\ref{poissonweak}) gives
$$
(\nabla (\phi-\phi_n), \nabla \psi) =  (\text{curl } (\mathbf{u}-\mathbf{u}_n), \psi) \quad \forall \psi \in V_{n}^{k}.
$$
Denote by $\phi_I \in V_{n}^{k}$ the nodal interpolation of $\phi$. Set $\epsilon = \phi_I-\phi$, $e=\phi_I - \phi_n$ and take $\psi = e$, then we have
\bes
(\nabla e, \nabla e) = (\nabla \epsilon, \nabla e) + (\text{curl } (\mathbf{u}-\mathbf{u}_n), e) = (\nabla \epsilon, \nabla e) + (\mathbf{u}-\mathbf{u}_n, \textbf{curl } e),
\ees
which gives
\be\label{eerr}
\bal
\|e\|^2_{H^1(\Omega)} \leq & \|\epsilon\|_{H^1(\Omega)}\|e\|_{H^1(\Omega)}+\|\mathbf{u}-\mathbf{u}_n\|_{[L^2(\Omega)]^2}\|\textbf{curl } e\|_{[L^2(\Omega)]^2} \\
\leq & C (\|\epsilon\|_{H^1(\Omega)} +\|\mathbf{u}-\mathbf{u}_n\|_{[L^2(\Omega)]^2} )\|e\|_{H^1(\Omega)},
\eal
\ee
By the triangle inequality, we have
\be
\bal
\|\phi-\phi_n\|_{H^1(\Omega)} \leq & \|\epsilon\|_{H^1(\Omega)}+\|e\|_{H^1(\Omega)} \leq C\left( \|\epsilon\|_{H^1(\Omega)}+  \|\mathbf{u}-\mathbf{u}_n\|_{[L^2(\Omega)]^2} \right)
\eal
\ee
Recall that $\phi \in H^{2+\alpha}(\Omega)$, so it follows
$$
\|\epsilon\|_{H^1(\Omega)} \leq Ch^{\min\{k,1+\alpha\}},
$$
which together with (\ref{stokeserr}) for Algorithm \ref{femalg} and (\ref{stokeserr+}) for Algorithm \ref{femalg+} leads to the conclusion.
\end{proof}

\begin{theorem}\label{phierrL2thm}
Let $\phi_n$ be the solution of finite element solution of (\ref{poissonfem}) from Algorithm \ref{femalg} or Algorithm \ref{femalg+}, and $\phi$ be the solution of the biharmonic problem (\ref{eqnbi}). If $\phi_n$ is the solution of Algorithm \ref{femalg}, then we have
\be\label{l2erru}
\|\phi-\phi_n\| \leq Ch^{\min\{k+1,\alpha+2, 2\alpha\}};
\ee
if it is the solution of Algorithm \ref{femalg+}, then we have
\be\label{l2erru+}
\|\phi-\phi_n\| \leq Ch^{\min\{k+1, \beta+3, \alpha+2,  2\alpha\}}.
\ee
\end{theorem}
\begin{proof}
Consider the Poisson problem \be\label{dualVL2}
-\Delta v = \phi - \phi_n \text{ in } \Omega \quad v =0 \text{ on } \partial \Omega.
\ee
Then we have
\be\label{dualL2}
\|\phi - \phi_n\|^2=(\nabla (\phi - \phi_n), \nabla v).
\ee
By Subtracting (\ref{poissonfem}) from (\ref{poissonweak}), it follows
\be\label{galorth}
(\nabla (\phi - \phi_n), \nabla \psi) = (\text{curl }(\mathbf{u} - \mathbf{u}_n), \psi) \quad \forall \psi \in V_{n}^{k}.
\ee
Set $\psi = v_I \in V_{n}^{k}$ the nodal interpolation of $v$ and subtract (\ref{galorth}) from (\ref{dualVL2}), we have
\bes
\bal
\|\phi - \phi_n\|^2= & (\nabla (\phi - \phi_n), \nabla (v - v_I)) + (\text{curl }(\mathbf{u} - \mathbf{u}_n), v_I), \\
= & (\nabla (\phi - \phi_n), \nabla (v - v_I)) + (\text{curl }(\mathbf{u} - \mathbf{u}_n), v_I-v)+ (\text{curl }(\mathbf{u} - \mathbf{u}_n), v), \\
= & (\nabla (\phi - \phi_n), \nabla (v - v_I)) + (\mathbf{u} - \mathbf{u}_n, \textbf{curl }(v_I-v) )+ (\mathbf{u} - \mathbf{u}_n, \textbf{curl }v), \\
\leq & \|\phi-\phi_n\|_{H^1(\Omega)} \|v - v_I\|_{H^1(\Omega)} + \|\mathbf{u} - \mathbf{u}_n\|_{[L^2(\Omega)]^2} \|v-v_I\|_{H^1(\Omega)} \\
& + \|\mathbf{u} - \mathbf{u}_n\|_{[H^{-\min\{\lfloor\beta\rfloor, 1\}}(\Omega)]^2}\|\textbf{curl }v\|_{H^{\min\{\lfloor\beta\rfloor,1\}}(\Omega)},
\eal
\ees
where $\lfloor \cdot \rfloor$ represents the floor function.

The regularity result \cite{Grisvard1985, Grisvard1992} of the Poisson problem (\ref{dualL2}) gives
\be\label{poireg}
\|v\|_{H^{\min\{1+\beta,2\}}(\Omega)} \leq C \|\phi-\phi_n\|_{H^{\min\{\beta-1,0\}}(\Omega)}\leq C\|\phi-\phi_n\|.,
\ee
where $\beta<\frac{\pi}{\omega}$.
From (\ref{interr}), we have
\bes
\| v - v_I \|_{H^1(\Omega)} \leq Ch^{\min\{\beta,1\}}\|v\|_{H^{\min\{1+\beta,2\}}(\Omega)}.
\ees
For Algorithm \ref{femalg}, we have the following result by (\ref{stokeserr}).
Since $\beta<\frac{\pi}{\omega}$, so if $\omega>\pi$ we have $\lfloor\beta\rfloor=0$,  and
\be\label{betaleqpi} 
\|\mathbf{u} - \mathbf{u}_n\|_{[H^{-\min\{\lfloor\beta\rfloor, 1\}}(\Omega)]^2}=\|\mathbf{u} - \mathbf{u}_n\|_{[L^2(\Omega)]^2} \leq Ch^{2\alpha}.
\ee
and if $\omega<\pi$, we have $\lfloor\beta\rfloor=1$, and
\be\label{betageqpi}
\|\mathbf{u} - \mathbf{u}_n\|_{[H^{-\min\{\lfloor\beta\rfloor, 1\}}(\Omega)]^2} \leq Ch^{\min\{2k, k+2, k+\alpha, \alpha+2,2\alpha\}}.
\ee
For $\omega\in(0,2\pi)\setminus\{\pi\}$, (\ref{betaleqpi}) and (\ref{betageqpi}) imply that
\be\label{Hbeta}
\|\mathbf{u} - \mathbf{u}_n\|_{[H^{-\min\{\lfloor\beta\rfloor, 1\}}(\Omega)]^2} \leq Ch^{\min\{2k, k+2, k+\alpha, \alpha+2,2\alpha\}}.
\ee
Thus, we have by (\ref{phiH1err}), (\ref{stokeserr}), and (\ref{Hbeta}),
\be\label{dualL22}
\bal
\|\phi - \phi_n\|^2
\leq & Ch^{\min\{k+1, \alpha+2,k+\beta,2\alpha+\beta\}}\|v\|_{H^{\min\{1+\beta,2\}}(\Omega)}
 + Ch^{\min\{2k,k+2,k+\alpha,\alpha+2, 2\alpha\}}   \|v\|_{H^{\min\{1+\beta,2\}}(\Omega)}\\
\leq & Ch^{\min\{k+1, \alpha+2, 2\alpha\}}\|v\|_{H^{\min\{1+\beta,2\}}(\Omega)}.
\eal
\ee
By (\ref{poireg}) and (\ref{dualL22}),  the estimate (\ref{l2erru}) holds.

Similarly, for Algorithm \ref{femalg+} we have
\be\label{dualL22+}
\bal
\|\phi - \phi_n\|^2
\leq & Ch^{\min\{k+1, \alpha+2,\beta+3, k+\beta,2\alpha+\beta\}}\|v\|_{H^{\min\{1+\beta,2\}}(\Omega)}\\
& + Ch^{\min\{2k,k+2,k+\beta,\alpha+2,\beta+3, 2\alpha\}}   \|v\|_{H^{\min\{1+\beta,2\}}(\Omega)}\\
\leq & Ch^{\min\{k+1, \alpha+2, \beta+3, 2\alpha\}}\|v\|_{H^{\min\{1+\beta,2\}}(\Omega)}.
\eal
\ee
By (\ref{poireg}) and (\ref{dualL22+}),  the estimate (\ref{l2erru+}) holds.
\end{proof}

\begin{remark}
For results in Theorem \ref{phierrthm3.1} and Theorem \ref{phierrL2thm}, we have the following comparisons.
\begin{itemize}
\item From Figure \ref{Regularity}(b), we find $\alpha>\beta+1$ only when the largest interior angle $\omega$ of the domain $\Omega$ is close to $\frac{\pi}{3}$, so there exists $\omega_0 \in (\frac{\pi}{3}, \frac{\pi}{2})$ such that when $\omega>\omega_0$ the Algorithm \ref{femalg+} and the Algorithm \ref{femalg} give the same convergence rates. In particular, these two algorithms give the same convergence rates on non-convex domains.
\item For $k\leq 4$, from Figure \ref{Regularity} we can find $\beta+2$ in estimate (\ref{phiH1err+}) and $\beta+3$ in estimate (\ref{l2erru+}) cannot achieve the minimum, so these two algorithms also give the same convergence rates. 
\end{itemize}
\end{remark}

\section{Optimal error estimates on graded meshes}\label{sec-4}


To improve the convergence rate, we consider the Algorithm \ref{femalg} on graded meshes. We start with the regularity in weighted Sobolev space.

\subsection{Weighted Sobolev space}

Recall that $Q_i$, $i=1,\cdots, N$ are the vertices of domain $\Omega$.
Let $r_i=r_i(x,Q_i)$ be the distance from $x$ to $Q_i$ and let
\begin{eqnarray}\label{eqn.rho}
\rho(x)=\Pi_{1\leq i \leq N} r_i(x,Q_i).
\end{eqnarray}
Let $\mathbf a = (a_1, \cdots,a_i, \cdots, a_N)$ be a vector with $i$th component associated with  $Q_i$. We denote $t+\mathbf a = (t+a_1, \cdots, t+a_N)$, so we have
$$
\rho(x)^{(t+\mathbf a)}=\Pi_{1\leq i \leq N} r_i^{(t+\mathbf a)}(x,Q_i) = \Pi_{1\leq i \leq N} r_i^t(x,Q_i) \Pi_{1\leq i \leq N} r_i^{\mathbf a}(x,Q_i) = \rho(x)^t \rho(x)^{\mathbf{a}}.
$$
Then, we introduce the Kondratiev-type weighted Sobolev spaces for the analysis of the Stokes problem (\ref{stokesweak}) and the Poisson problem (\ref{eqnpoisson}).
\begin{definition} \label{wss} (Weighted Sobolev spaces)
For $a\in\mathbb R$, $m\geq 0$, and $G\subset \Omega$,  we define the weighted Sobolev space
$$
\maK_{\mathbf a}^m(G) := \{v|\ \rho^{|\nu|-\mathbf a}\partial^\nu v\in L^2(G), \forall\ |\nu|\leq m \},
$$
where the multi-index $\nu=(\nu_1,\nu_2)\in\mathbb Z^2_{\geq 0}$, $|\nu|=\nu_1+\nu_2$, and $\partial^\nu=\partial_x^{\nu_1}\partial_y^{\nu_2}$.
The $\maK_{\mathbf a}^m(G)$ norm for $v$  is defined by
$$
\|v\|_{\maK_{\mathbf a}^m(G)}=\big(\sum_{|\nu|\leq m}\iint_{G} |\rho^{|\nu|-\mathbf a}\partial^\alpha v|^2dxdy\big)^{\frac{1}{2}}.
$$
\end{definition}


\begin{remark}\label{KHeq}
According to Definition \ref{wss}, in the region that is away from the corners, the weighted space $\maK^m_{\mathbf a}$ is equivalent to the Sobolev space $H^m$. In the neighborhood of $Q_i$, the space $\maK^m_{\mathbf a}(B_i)$ is the equivalent to the Kondratiev space  \cite{Dauge88,Grisvard1985,Kondratiev67},
$$
\maK_{a_i}^m(B_i) := \{v|\ r_i^{|\nu|-a_i}\partial^\alpha v\in L^2(B_i), \forall\ |\nu|\leq m \},
$$
where $B_i \subset \Omega$ represents the neighborhood of $Q_i$ satisfying $B_i \cap B_j = \emptyset$ for $i\not = j$.
\end{remark}

\subsection{Graded meshes}

We now present the construction of graded meshes to improve the convergence rate of the numerical approximation from Algorithm \ref{femalg}.

\begin{algorithm}\label{graded} (Graded refinements) Let $\maT$ be a triangulation of $\Omega$ with shape-regular triangles.
Recall that $Q_i$, $i=1,\cdots, N$ are the vertices of $\Omega$.
Let ${AB}$ be an edge in the triangulation $\mathcal T$ with $A$ and $B$ as the endpoints. Then, in a graded refinement, a new node $D$ on $AB$ is produced according to the following conditions:
\begin{itemize}
\item[1.] (Neither $A$ nor $B$ coincides with $Q_i$.) We choose $D$  as the midpoint ($|AD|=|BD|$).
\item[2.] ($A$ coincides with $Q_i$.) We choose $r$  such that $|AD|=\kappa_{Q_i}|AB|$, where $\kappa_{Q_i}\in (0, 0.5)$ is a parameter that will be specified later. See Figure \ref{fig.2} for example.
\end{itemize}
Then, the graded refinement, denoted by $\kappa(\mathcal T)$, proceeds as follows.
For each triangle $T\in \mathcal T$, a new node is generated on each edge of $T$ as described above. Then, $T$ is decomposed into four small triangles by connecting these new nodes (Figure \ref{fig.333}). Given an initial mesh $\mathcal T_0$ satisfying the condition above, the associated family of graded meshes $\{\mathcal T_n,\ n\geq0\}$ is defined recursively $\mathcal T_{n+1}=\kappa(\mathcal T_{n})$.
\end{algorithm}

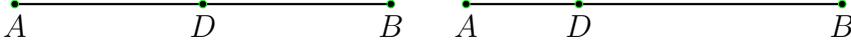
\begin{figure}
\begin{center}
\begin{tikzpicture}[scale=0.5]

\draw[thick]
(-1,1) node[anchor = north] {$B$}
-- (-11, 1) node[anchor = north] {$A$}
-- (-1,1);
\draw[green,fill=black] (-6,1) circle (.1);
\draw[green,fill=black] (-11,1) circle (.1);
\draw[green,fill=black] (-1,1) circle (.1);
\draw (-6,1) node[anchor = north] {$D$};

\draw[thick]
(11,1) node[anchor = north] {$B$}
-- (1, 1) node[anchor = north] {$A$}
-- (11,1);
\draw[green,fill=black] (4,1) circle (.1);
\draw[green,fill=black] (1,1) circle (.1);
\draw[green,fill=black] (11,1) circle (.1);
\draw (4,1) node[anchor = north] {$D$};

%

\end{tikzpicture}
\end{center}
\vspace*{-15pt}
\caption{The new node on an edge $AB$. (left): $A \neq Q_i$ and $B\neq Q_i$ (midpoint);  (right) $A=Q_i$  ($|AD|=\kappa_{Q_i}|AB|$,  $\kappa_{Q_i}<0.5$).}\label{fig.2}
\end{figure}


\begin{figure}
\begin{center}
\begin{tikzpicture}[scale=0.5]

\draw[thick]
(-1,1) node[anchor = north] {$x_2$}
-- (-4,7) node[anchor = south] {$x_0$}
-- (-11, 1) node[anchor = north] {$x_1$}
-- (-1,1);

\draw[thick]
(11,1) node[anchor = north] {$x_2$}
-- (8,7) node[anchor = south] {$x_0$}
-- (1, 1) node[anchor = north] {$x_1$}
-- (11,1);

\draw[thick]
(9.5,4)
-- (4.5,4)
-- (6,1) node[anchor = north] {$x_{12}$}
-- (9.5,4);
\draw (3.9,4.3) node {$x_{01}$};
\draw (10,4.3) node {$x_{02}$};

\draw[thick]
(-1,-7) node[anchor = north] {$x_2$}
-- (-4,-1) node[anchor = south] {$x_0$}
-- (-11,-7) node[anchor = north] {$x_1$}
-- (-1,-7);
\draw[green,fill=green] (-4,-1) circle (.2);

\draw[thick]
(-23/4, -5/2)
-- (-6,-7+0.04) node[anchor = north] {$x_{12}$}
-- (-13/4, -5/2)
-- (-23/4, -5/2);
\draw (-6.4,-2.2) node {$x_{01}$};
\draw (-2.7,-2.2) node {$x_{02}$};

\draw[thick]
(11,-7) node[anchor = north] {$x_2$}
-- (8,-1) node[anchor = south] {$x_0$}
-- (1, -7) node[anchor = north] {$x_1$}
-- (11,-7);
\draw[green,fill=green] (8,-1) circle (.2);

\draw[thick]
(25/4, -5/2)
-- (6,-7+0.04) node[anchor = north] {$x_{12}$}
-- (35/4, -5/2)
-- (25/4, -5/2);
\draw (5.6,-2.2) node {$x_{01}$};
\draw (9.3,-2.2) node {$x_{02}$};

\draw[thick] (29/8, -19/4) -- (7/2,-7) -- (49/8,-19/4) -- (29/8, -19/4);
\draw[thick] (49/8,-19/4) -- (59/8,-19/4) -- (15/2,-5/2) -- (49/8,-19/4);
\draw[thick] (59/8,-19/4) -- (17/2,-7) -- (79/8,-19/4) -- (59/8,-19/4);
\draw[thick] (121/16,-11/8) -- (15/2,-5/2) -- (131/16,-11/8) -- (121/16,-11/8);

%

\end{tikzpicture}
\end{center}
\vspace*{-15pt}
\caption{Refinement of a triangle $\triangle x_0x_1x_2$. First row: (left -- right): the initial triangle and the midpoint refinement; second row: two consecutive graded refinements toward $x_0=Q_i$, ($\kappa_{Q_i}<0.5$).}
\label{fig.333}
\end{figure}
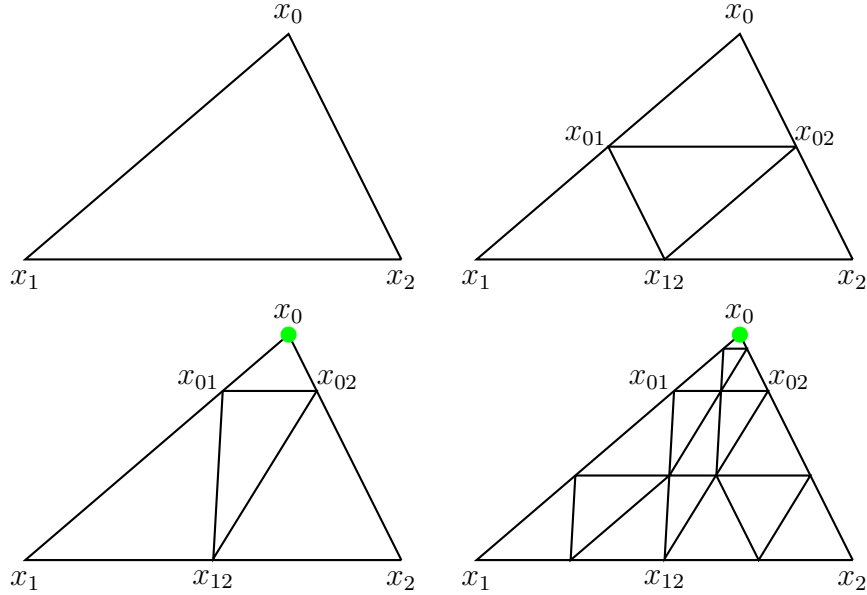

Given a grading parameter $\kappa_{Q_i}$, Algorithm \ref{graded} produces smaller elements near $Q_i$ for a better approximation of the singular solution. It is an explicit construction of graded meshes based on recursive refinements. See also \cite{ASW96,BNZ205, LMN10, LN09} and references therein for more discussions on the graded mesh.

Note that after $n$ refinements, the number of triangles in the mesh $\maT_n$ is $O(4^n)$, so we denote the ``mesh size" of $\maT_n$ by
\be\label{meshsize}
h = 2^{-n}.
\ee

In Algorithm \ref{graded}, we choose the parameter $\kappa_{Q_i}$ for each vertex $Q_i$ as follows. Recall that and $\alpha_0^i$ is the solution of (\ref{alpha0}) with $\omega$ being replaced by the interior angle $\omega_i$ at $Q_i$.
Given the degree of polynomials $k, k'$ in Algorithm \ref{femalg}, we choose
\be\label{kappainit}
\kappa_{Q_i}=2^{-\frac{\theta}{a_i}}\left(\leq \frac{1}{2}\right),
\ee
where $a_i>0$
and $\theta$ could be any possible constants satisfying
\be\label{thetarange}
a_i \leq \theta \leq \min\{k,m\}.
\ee
In (\ref{thetarange}), if we take $a_i=\theta$, the grading parameter $\kappa_{Q_i}=\frac{1}{2}$.

\subsection{Interpolation error estimates on graded meshes}




\begin{lem}\label{r1r3}
Let $T_{(0)}\in\mathcal T_{0}$ be an initial triangle of the triangulation $\mathcal T_n$ in Algorithm \ref{graded} with grading parameters $\kappa_{Q_i}$ given by (\ref{kappainit}). For $m\geq 1, k \geq 1$, we denote $v_I \in V_n^{k}$ (resp. $q_I \in S_n^{k-1}$) the nodal interpolation of $v \in \maK_{\mathbf a+1}^{m+1}(\Omega)$ (resp. $q \in \maK_{\mathbf a}^{m}(\Omega)$). If $\bar T_{(0)}$ does not contain any vertices $Q_i$, $i=1,\cdots,N$, then
\begin{equation*}
\|v-v_I\|_{H^1(T_{(0)})}\leq Ch^{\min\{k,m\}}, \quad \|q-q_I\|_{L^2(T_{(0)})} \leq Ch^{\min\{k,m\}},
\end{equation*}
where $h = 2^{-n}$.
\end{lem}
\begin{proof}
If $\bar T_0$ does not contain any vertices $Q_i$ of the domain $\Omega$, we have $v\in \maK_{\mathbf a+1}^{m+1}(\Omega) \subset H^{m+1}(T_{(0)})$ (see Remark \ref{KHeq}) and the mesh on $T_{(0)}$ is quasi-uniform (Algorithm \ref{graded}) with size $O(2^{-n})$. Therefore, based on the standard interpolation error estimate, we have
\begin{eqnarray}\label{1.1}
\|v-v_I\|_{H^1(T_{(0)})}\leq Ch^{\min\{k,m\}}\|v\|_{H^{m+1}(T_{(0)})}.
\end{eqnarray}
Note that $q \in \maK_{\mathbf a}^{m}(\Omega) \subset H^{m}(\Omega)$, we can similar obtain the estimate for $\|q-q_I\|$.
\end{proof}

We now study the interpolation error in the neighborhood $Q_i$, $i=1,\cdots, N$.
In the rest of this subsection, we assume $T_{(0)}\in\mathcal T_0$ is an initial triangle such that the $i$th vertex $Q_i$ is a vertex of $T_{(0)}$.
We first define mesh layers on $T_{(0)}$ which are collections of triangles in $\mathcal T_n$.
\begin{definition} (Mesh layers) Let $T_{(t)}\subset T_{(0)}$ be the triangle in $\mathcal T_t$, $0\leq t\leq n$, that is attached to the singular vertex $Q_i$ of $T_{(0)}$. For $0\leq t<n$, we define the $t$th mesh layer of $\mathcal T_n$ on $T_{(0)}$ to be the region $L_{t}:=T_{(t)}\setminus T_{(t+1)}$; and for $t=n$, the $n$th layer is $L_{n}:=T_{(n)}$.  See Figure \ref{fig.layer} for example.
\end{definition}


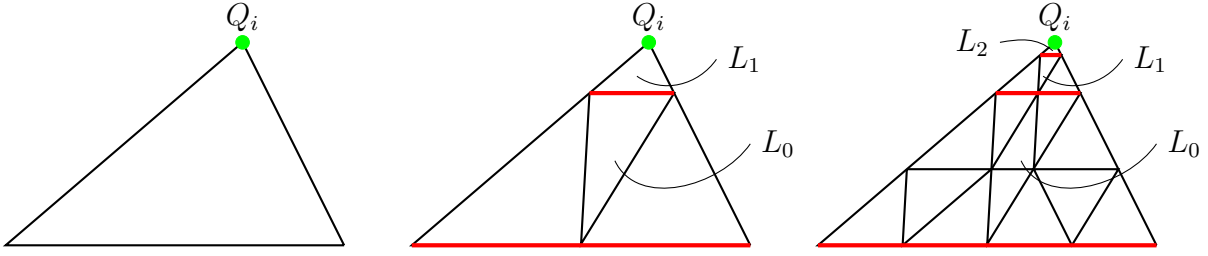
\begin{figure}
\begin{center}
\begin{tikzpicture}[scale=0.45]

\draw[thick]
(-13,-7)
-- (-16,-1) node[anchor = south] {$Q_i$}
-- (-23, -7)
-- (-13,-7);
\draw[green,fill=green] (-16,-1) circle (.2);

\draw[thick]
(-1,-7)
-- (-4,-1) node[anchor = south] {$Q_i$}
-- (-11,-7)
-- (-1,-7);
\draw[green,fill=green] (-4,-1) circle (.2);

\draw[thick]
(-23/4, -5/2)
-- (-6,-7+0.04)
-- (-13/4, -5/2)
-- (-23/4, -5/2);

\draw[red, ultra thick] (-13/4, -5/2) -- (-23/4, -5/2);
\draw[red, ultra thick] (-11,-7) -- (-1,-7);
\draw (-5,-4.5) to[out=-65,in=235] (-1,-4) node[anchor = west] {$L_0$};
\draw (-4.35,-2) to[out=-35,in=235] (-2,-1.5) node[anchor = west] {$L_1$};

\draw[thick]
(11,-7)
-- (8,-1) node[anchor = south] {$Q_i$}
-- (1, -7)
-- (11,-7);
\draw[green,fill=green] (8,-1) circle (.2);

\draw[thick]
(25/4, -5/2)
-- (6,-7+0.04)
-- (35/4, -5/2)
-- (25/4, -5/2);

\draw[thick] (29/8, -19/4) -- (7/2,-7) -- (49/8,-19/4) -- (29/8, -19/4);
\draw[thick] (49/8,-19/4) -- (59/8,-19/4) -- (15/2,-5/2) -- (49/8,-19/4);
\draw[thick] (59/8,-19/4) -- (17/2,-7) -- (79/8,-19/4) -- (59/8,-19/4);
\draw[thick] (121/16,-11/8) -- (15/2,-5/2) -- (131/16,-11/8) -- (121/16,-11/8);

\draw[red, ultra thick] (131/16,-11/8) -- (121/16,-11/8);
\draw[red, ultra thick] (-13/4+12, -5/2) -- (-23/4+12, -5/2);
\draw[red, ultra thick] (-11+12,-7) -- (-1+12,-7);
\draw (-5+12,-4.5) to[out=-65,in=235] (-1+12,-4) node[anchor = west] {$L_0$};
\draw (-4.35+12,-2) to[out=-35,in=235] (-2+12,-1.5) node[anchor = west] {$L_1$};
\draw (127/16,-1.25) to[out=500,in=10] (102/16,-1) node[anchor = east] {$L_2$};

\end{tikzpicture}
\end{center}
\vspace*{-5pt}
\caption{Mesh layers (left -- right): the initial triangle $T_{(0)}$ with a vertex $Q_i$; two layers after one refinement; three layers after two refinements.}
\label{fig.layer}
\end{figure}

\begin{remark}
The triangles in $\mathcal T_n$ constitute $n$ mesh layers on $T_{(0)}$. According to Algorithm \ref{graded} and the choice of grading parameters $\kappa_{Q_i}$ given by (\ref{kappainit}), the mesh size in the $t$th  layer $L_t$ is
\begin{equation}\label{eqn.size}O(\kappa_{Q_i}^t2^{t-n}). \end{equation}
Meanwhile, the weight function $\rho$ in (\ref{eqn.rho}) satisfies
\begin{eqnarray}\label{eqn.dist}
\rho=O(\kappa_{Q_i}^t) \ \ \ {\rm{in\ }} L_t\ (0\leq t< n) \qquad {\rm{and}}  \qquad \rho \leq C\kappa_{Q_i}^n \ \ \ {\rm{in\ }} L_n.
\end{eqnarray}
\end{remark}
Although the mesh size varies in different layers, the triangles in $\mathcal T_n$ are shape regular. In addition, using the local Cartesian coordinates such that $Q$ is the origin,  the mapping
\begin{eqnarray}\label{eqn.map}
\mathbf B_{t}= \begin{pmatrix}
  \kappa_{Q_i}^{-t}   &   0 \\
  0    &   \kappa_{Q_i}^{-t} \\
\end{pmatrix}\qquad 0\leq t\leq n,
\end{eqnarray}
is a bijection between $L_t$ and $L_0$ for $0\leq t<n$ and a bijection between $L_n$ and $T_{(0)}$. We call $L_0$ (resp. $T_{(0)}$) the reference region associated to $L_t$ for $0\leq t<n$ (resp. $L_n$).

With the mapping (\ref{eqn.map}), we have that for any point $(x, y)\in L_t$, $0\leq t\leq n$, the image point $(\hat x, \hat y):=\mathbf B_t(x,y)$ is in its reference region. We introduce the following dilation result.

\begin{lem}\label{dilation}
For $0\leq t\leq n$,
given a function $v(x, y) \in \maK_{a}^{l}(L_t)$, the function $\hat v(\hat x, \hat y):=v(x, y)$ belongs to $\maK_{a}^{l}(\hat L)$, where $(\hat x, \hat y):=\mathbf B_t(x,y)$, $\hat L=L_0$ for $0\leq t< n$, and $\hat L=T_{(0)}$ for $t=n$. Then, it follows
\begin{equation*}
\|\hat v(\hat x, \hat y)\|_{\maK_{a}^{l}(\hat L)} = \kappa_{Q_i}^{t(a-1)} \|v(x,y)\|_{\maK_{a}^{l}(L_i)}.
\end{equation*}
\end{lem}
\begin{proof}
The proof can be found in \cite[Lemma 4.5]{li2021}.
\end{proof}

We then derive the interpolation error estimate in each layer.

\begin{lem}\label{TNtri}
For $k \geq 1, m \geq 1$, set $\kappa_{Q_i}$ in (\ref{kappainit}) with $\theta$ satisfying (\ref{thetarange}) for the graded mesh on $T_{(0)}$.
Let $h:=2^{-n}$, then in the $t$th layer $L_t$ on $T_{(0)}$, $0\leq t<n$,\\
(i) if $v_I \in V_n^k$ be the nodal interpolation of $v\in \maK_{\mathbf a+1}^{m+1}(\Omega)$, it follows
\be\label{projgh1}
|v-v_{I}|_{H^1(L_t)} \leq Ch^{\theta}\|v\|_{\maK_{a_i+1}^{m+1}(L_t)};
\ee
(ii) if $q_I \in V_n^{k-1}$ be the nodal interpolation of $q\in \maK_{\mathbf a}^{m}(\Omega)$, it follows
\be\label{projgl2}
\|q-q_{I}\|_{L^2(L_t)} \leq Ch^{\theta}\|q\|_{\maK_{a_i}^{m}(L_t)}.
\ee
\end{lem}
\begin{proof}
For $L_t$ associated with $Q_i$, $0\leq t<n$, the space $\maK_{a_i+1}^{m+1}(L_t)$ (resp. $\maK_{a_i}^{m}(L_t)$) is equivalent to $H^{m+1}(L_t)$ (resp. $H^{m}(L_t)$). Therefore, $v$ (resp. $q$) is a continuous function in $L_t$.
For any point $(x, y)\in L_t$, let $(\hat x, \hat y)=\mathbf B_t(x,y)\in L_0$. For $v(x, y)$ (resp. $q(x, y)$) in $L_t$, we define $\hat v(\hat x, \hat y):=v(x, y)$ (resp. $\hat q(\hat x, \hat y):=q(x, y)$) in $L_0$.

(i). Using the standard interpolation error estimate, the scaling argument,  the estimate in (\ref{eqn.size}), and the mapping in (\ref{eqn.map}), we have
\begin{eqnarray*}
 |v-v_I|_{H^1(L_t)}&=& |\hat v-\hat v_I|_{H^1(L_0)}\leq C 2^{(t-n)\mu}\|\hat v\|_{\maK_{a_i+1}^{m+1}(L_0)}
 \leq C 2^{(t-n)\mu}\kappa_{Q_i}^{a_it}\|v\|_{\maK_{a_i+1}^{m+1}(L_t)},
\end{eqnarray*}
where we have used Lemma \ref{dilation} in the last inequality.
Since $\kappa_{Q_i} = 2^{-\frac{\theta}{a_i}}$,
so we have
$\kappa_{Q_i}^{a_it} = 2^{-\theta t}$.
Set $\mu=\min\{k,m\}$, by $\theta \leq \mu$ from (\ref{thetarange}) and $t<n$, we have $2^{(n-t)(\theta-\mu)}<2^0=1$.
Therefore, we have the estimate
\begin{eqnarray*}
 |v-v_I|_{H^1(L_t)}
 &\leq& C 2^{(t-n)\mu-\theta t}\|v\|_{\maK_{a_i+1}^{m+1}(L_t)} = C2^{-n\theta} 2^{(n-t)(\theta-\mu)}\|v\|_{\maK_{a_i+1}^{m+1}(L_t)}\\
 &\leq& C2^{-n\theta} \|v\|_{\maK_{a_i+1}^{m+1}(L_t)} \leq Ch^{\theta} \|v\|_{\maK_{a_i+1}^{m+1}(L_t)}.
\end{eqnarray*}

(ii). We can show that
\begin{eqnarray*}
\|q-q_I\|_{L^2(L_t)}&=& \kappa_{Q_i}^t\|\hat q-\hat q_I\|_{L^2(L_0)}\leq C \kappa_{Q_i}^t 2^{(t-n)\mu}\|\hat q\|_{\maK_{a_i}^{m}(L_0)} \leq C 2^{(t-n)\mu}\kappa_{Q_i}^{a_it}\|q\|_{\maK_{a_i}^{m}(L_t)},
\end{eqnarray*}
where again we used Lemma \ref{dilation} in the last inequality.
Using the similar argument as in (i), the estimate (\ref{projgl2}) holds.
\end{proof}

Before deriving the interpolation error estimate in the last layer $L_n$ on $T_{(0)}$, we first introduce the following results.

\begin{lemma}\label{Lnrela}
For $\forall v \in \maK_a^l(L_n)$, if $0\leq l'\leq l$ and $a'\leq a$, then it follows
\be
\|v\|_{\maK_{a'}^{l'}(L_n)}\leq C \kappa_{Q_i}^{n(a-a')} \|v\|_{\maK_{a}^{l}(L_n)}.
\ee
\end{lemma}
\begin{proof}
This is a direct application of \cite[Lemma 2.6]{LN09} under condition (\ref{eqn.dist}) on $L_n$.
\end{proof}

\begin{lemma}\label{HtoKbdd}
For $\forall v \in \maK_a^l(L_n)$ , if $a\geq l$, then it follows that
\be\label{HtoKbddfor}
\|v\|_{H^l(L_n)} \leq C \kappa_{Q_i}^{n(a-l)} \|v\|_{\maK_a^l(L^n)}.
\ee
\end{lemma}
\begin{proof}
This is a direct application of \cite[Lemma 2.8]{LN09} under condition (\ref{eqn.dist}) on $L_n$.
\end{proof}

\begin{lem}\label{TNtri2}
For $k \geq 1, m \geq 1$, set $\kappa_{Q_i}$ in (\ref{kappainit}) with $\theta$ satisfying (\ref{thetarange}) for the graded mesh on $T_{(0)}$.
Let $h:=2^{-n}$, then in the $n$th layer $L_n$ on $T_{(0)}$ for $n$ sufficiently large,\\
(i) if $v_I \in V_n^k$ be the nodal interpolation of $v\in \maK_{\mathbf a+1}^{m+1}(\Omega)$, it follows
\be\label{projgh12}
|v-v_{I}|_{H^1(L_n)} \leq Ch^{\theta}\|v\|_{\maK^{m+1}_{{a_i}+1}(L_{n})};
\ee
(ii) if $q_I \in V_n^{k-1}$ be the nodal interpolation of $q\in \maK_{\mathbf a}^{m}(\Omega)$, it follows
\be\label{projgl22}
\|q-q_{I}\|_{L^2(L_n)} \leq Ch^{\theta}\|q\|_{\maK_{a_i}^{m}(L_n)}.
\ee
\end{lem}
\begin{proof}
Recall the mapping $\mathbf B_n$ in (\ref{eqn.map}). For any point $(x, y)\in L_n$, let $(\hat x, \hat y)=\mathbf B_n(x,y)\in T_{(0)}$.

(i). Let $\eta: T_{(0)} \rightarrow [0, 1]$ be a  smooth function that
is equal to $0$ in a neighborhood of $Q_i$, but is equal
to 1 at all the other nodal points in $\mathcal T_0$.
For a function $v(x, y)$ in $L_n$, we define $\hat v(\hat x, \hat y):=v(x, y)$ in $T_{(0)}$. We take $w=\eta \hat v$ in $T_{(0)}$. Consequently, we have for $l\geq 0$
\begin{equation}\label{eqn.aux111}
\bal
\|w\|^2_{\maK^{l}_{1}(T_{(0)})} & = & \|\eta
\hat v\|^2_{\maK^{l}_{1}(T_{(0)})} \leq C
\|\hat v\|^2_{\maK^{l}_{1}(T_{(0)})},
\eal
\end{equation}
where $C$ depends on $l$ and the smooth function $\eta$. Moreover, the condition $\hat v\in \maK_{a^i+1}^{m+1}(T_{(0)})$ with and $m\geq 2$ implies $\hat v(Q)=0$ (see, e.g., \cite[Lemma 4.7]{LN09}).
Let $w_{\hat I}$ be the nodal interpolation of $w$ associated with the mesh $\mathcal T_0$ on $T_{(0)}$.
Therefore, by the definition of $w$, we have
\begin{eqnarray}\label{wi}
w_{\hat I}=\hat v_{\hat I} = \widehat{v_{I}} \quad {\rm{in}}\ T_{(0)}.
\end{eqnarray}

Note that the $\maK^{l}_{1}$ norm and the $H^l$ norm are equivalent for $w$ on $T_{(0)}$, since $w=0$ in the neighborhood of the vertex $Q_i$. Let $r$ be the distance from $(x,y)$ to $Q_i$, and $\hat r$ be the distance from $(\hat x,\hat y)$ to $Q_i$. Then, by the definition of the weighted space, the scaling argument, (\ref{eqn.aux111}),  (\ref{wi}),  and (\ref{eqn.dist}), we have
\begin{eqnarray*}
|v-v_{I}|_{H^1(L_{n})}^2 &\leq& C\|v-v_{I}\|_{\maK^1_{{1}}(L_{n})}^2 \leq C\sum_{|\nu|\leq 1}\|r(x,y)^{|\nu|-1}\partial^\nu (v- v_I)\|_{L^2(L_{n})}^2\\
&\leq &C\sum_{|\nu|\leq 1}\|\hat r(\hat{x},\hat{y})^{|\nu|-1}\partial^\nu (\hat v- \widehat{v_I})\|_{L^2(T_{(0)})}^2\leq C\|\hat v- w+w-\widehat{v_I}\|_{\maK^1_{1}( T_{(0)})}^2 \\
& \leq &C\big( \|\hat v-w\|^2_{\maK^1_{1}(T_{(0)})} +
\|w-\widehat{v_I}\|^2_{\maK^1_{1}( T_{(0)}  )}\big) = C\big( \|\hat v-w\|^2_{\maK^1_{1}(T_{(0)})} +
\|w-w_{\hat I}\|^2_{\maK^1_{1}( T_{(0)} )}\big)  \\
& \leq & C\big(   \|\hat v\|^2_{\maK^1_{1}(T_{(0)} )} +
\|w\|^2_{\maK^{m+1}_{1}( T_{(0)}  )}\big) \leq C\big(   \|\hat v\|^2_{\maK^1_{1}(T_{(0)})} +
\|\hat v\|^2_{\maK^{m+1}_{1}( T_{(0)} )}\big)\\%
& = & C\big(   \|v\|^2_{\maK^1_{1}(L_n)} +
\|v\|^2_{\maK^{m+1}_{1}( L_n)}\big)\leq C
\kappa_{Q_i}^{2na_i}\|v\|_{\maK^{m+1}_{{a_i}+1}(L_{n})}^2\\
& \leq & C
2^{-2n\theta}\|v\|_{\maK^{m+1}_{{a_i}+1}(L_{n})}^2\leq C
h^{2\theta}\|v\|_{\maK^{m+1}_{{a_i}+1}(L_{n})}^2,
\end{eqnarray*}
where the ninth and tenth relationships are based on Lemma \ref{dilation} and Lemma \ref{Lnrela}, respectively.
This completes the proof of (\ref{projgh1}).

(ii). Since $q\in L^2(\Omega)$, we have that the interpolation operator is $L^2$ stable
\be\label{ProfL2Stab}
\|q_I\|_{L^2(L_n)} \leq C \|q\|_{L^2(L_n)}.
\ee
Thus, by (\ref{ProfL2Stab}) and (\ref{HtoKbddfor}), we have
\bes
\bal
\|q-q_I\|_{L^2(L_n)} \leq C \|q\|_{L^2(L_n)} \leq C \kappa_{Q_i}^{na_i} \|q\|_{\maK_{a_i}^{m}(L_n)} \leq C2^{-n\theta}\|q\|_{\maK_{a_i}^{m}(L_n)} \leq Ch^{\theta}\|q\|_{\maK_{a_i}^{m}(L_n)}.
\eal
\ees
\end{proof}

\begin{theorem}\label{gradprojerr}
Let $\mathcal T_{0}$ be an initial triangle of the triangulation $\mathcal T_n$ in Algorithm \ref{graded} with grading parameters $\kappa_{Q_i}$ in (\ref{kappainit}).
For $k\geq1, m \geq 1$,
if $v_I \in V_n^k$ (resp. $q_I \in V_n^{k-1}$) be the nodal interpolation of $v\in \maK_{\mathbf a+1}^{m+1}(\Omega)$ (resp. $q\in \maK_{\mathbf a}^{m}(\Omega)$). Then, it follows the following interpolation error
\be\label{projgherr}
\|v-v_{I}\|_{H^1(\Omega)} \leq Ch^{\theta} \|v\|_{\maK_{\mathbf a+1}^{m+1}(\Omega)}, \quad \|q-q_I\| \leq Ch^{\theta}\|q\|_{\maK_{\mathbf a}^{m}(\Omega)},
\ee
where $h:=2^{-n}$, and $\theta$ satisfying (\ref{thetarange}).
\end{theorem}
\begin{proof}
By summing the estimates in Lemma \ref{r1r3}, Lemma \ref{TNtri}, and Lemma \ref{TNtri2}, we have
\bes
\bal
\|v-v_{I}\|^2_{H^1(\Omega)} =& \sum_{T_{(0)} \in \mathcal T_{0}} \|v-v_{I}\|^2_{H^1(T_{(0)})} \leq Ch^{2\theta}\|v\|^2_{\maK_{\mathbf a+1}^{m+1}(\Omega)}, \\
\|q-q_{I}\|^2 =& \sum_{T_{(0)} \in \mathcal T_{0}} \|v-v_{I}\|^2_{L^2(T_{(0)})} \leq Ch^{2\theta}\|q\|^2_{\maK_{\mathbf a}^{m}(\Omega)}.
\eal
\ees
\end{proof}

\subsection{Regularity in weighted Sobolev space}

Let $\alpha_0^i$ the solution of  (\ref{alpha0}) with $\omega$ being replaced by the interior angle $\omega_i$ at $Q_i$. We denote the vector
$\bm{\alpha}_0=(\alpha_0^1, \cdots, \alpha_0^i, \cdots, \alpha_0^N)$.
By Lemma \ref{Fbblem}, for $f\in H^{-1}(\Omega)$, there exists $\mathbf{F} \in [\maK_{\mathbf a-1}^{0}(\Omega)]^2$ with $a_i < \min_i\{\alpha_0^i\}$ satisfying (\ref{curlFf}) and
\be\label{Fdecom1}
\|\mathbf{F}\|_{[\maK_{\mathbf a-1}^{0}(\Omega)]^2} \leq C \|f\|_{H^{-1}(\Omega)}.
\ee
If $f \in \maK_{\mathbf a-2}^{m-2}(\Omega) \cap H^{-1}(\Omega)$ with $m\geq 1$ and $0\leq \mathbf a < \bm\alpha_0$, then we can find $\mathbf{F} \in [\maK_{\mathbf a-1}^{m-1}(\Omega)]^2$ satisfying (\ref{curlFf}) and
\be\label{Fdecom2}
\|\mathbf{F}\|_{[\maK_{\mathbf a-1}^{m-1}(\Omega)]^2} \leq C \|f\|_{\maK_{\mathbf a-2}^{m-2}(\Omega)}.
\ee

For the Stokes problem (\ref{stokes}), we have the following regularity estimate in weighted Sobolev space \cite{bernardi1981}.
\begin{lemma}\label{stokeswreg}
Let $(\mathbf{u},p)\in [H_0^1(\Omega)]^2\times L_0^2(\Omega)$ be the solution of the Stokes problem (\ref{stokes}). For $m\geq 1$ and $0\leq \mathbf a \leq \bm\alpha_0$, if $\mathbf{F}\in [\maK_{\mathbf a-1}^{m-1}(\Omega)]^2$, then it follows
\be\label{regweak}
\|\mathbf{u}\|_{[\maK_{\mathbf a+1}^{m+1}(\Omega)]^2} + \|p\|_{\maK_{\mathbf a}^{m}(\Omega)} \leq C \|\mathbf{F}\|_{[\maK_{\mathbf a-1}^{m-1}(\Omega)]^2}.
\ee
\end{lemma}

We then have the following result.
\begin{lemma}\label{biregwsb}
Given $f\in \maK_{\mathbf a-2}^{m-2}(\Omega)\cap H^{-1}(\Omega)$ for $0\leq \mathbf a <\bm \alpha_0$ and $m\geq 1$, let $\phi\in H_0^2(\Omega)$ be the solution of the Possible problem (\ref{eqnpoisson}), then it follows
\be
\|\phi\|_{\maK_{\mathbf a+2}^{m+2}(\Omega)} \leq C \|f\|_{\maK_{\mathbf a-2}^{m-2}(\Omega)}.
\ee
\end{lemma}
\begin{proof}
By Lemma \ref{stokeswreg}, we have
$\mathbf{u} \in [\maK_{\mathbf a+1}^{m+1}(\Omega)]^2 \cap [H_0^1(\Omega)]^2$, thus we have
$(\text{curl }\mathbf{u}) \in \maK_{\mathbf a}^{m}(\Omega)$ and
\be\label{curltoF}
\|\text{curl }\mathbf{u}\|_{\maK_{\mathbf a}^{m}(\Omega)} \leq C\|\nabla \mathbf{u}\|_{[\maK_{\mathbf a}^{m}(\Omega)]^2} \leq C\|\mathbf{u}\|_{[\maK_{\mathbf a+1}^{m+1}(\Omega)]^2}\leq C\|\mathbf{F}\|_{[\maK_{\mathbf a-1}^{m-1}(\Omega)]^2}.
\ee
By the regularity estimate \cite{Grisvard1985, LMN10} for the Poisson problem (\ref{eqnpoisson}), we have
\be\label{Poissonwreg}
\|\phi\|_{\maK_{\mathbf a+2}^{m+2}(\Omega)} \leq \|\text{curl }\mathbf{u}\|_{\maK_{\mathbf a}^{m}(\Omega)}.
\ee
The conclusion holds by combining (\ref{Fdecom2}), (\ref{curltoF}) and (\ref{Poissonwreg}).
\end{proof}

\subsection{Optimal error estimates on graded meshes}

To better observe the threshold of grading parameter $\kappa_{Q_i}$ in obtaining the optimal convergence rates, we always assume $1\leq k\leq m$ in the following discussions, otherwise, we just replace $k$ by $\min\{k,m\}$.
In this section, we assume that $f\in \maK_{\mathbf{a}-1}^{m-1}(\Omega)\cap\maK_{\mathbf{b}-1}^{m-1}(\Omega)$ with $0< \mathbf a < \bm\alpha_0$, and  $0<\mathbf{b}<\bm{\beta}_0$,
where $\bm{\beta}_0 = (\frac{\pi}{\omega_1}, \cdots, \frac{\pi}{\omega_N})$.

If $\mathbf{F}$ is given by Lemma \ref{Fint}, then we have $\mathbf{F} \in [\maK_{\mathbf a-1}^{m-1}(\Omega)]^2$, and the regularities in Lemma \ref{stokeswreg} and Lemma \ref{biregwsb} hold.

If $\mathbf{F}$ is given by Lemma \ref{Fbblem}, by the regularity estimate \cite{BNZ205} for the Poisson problem (\ref{eqnpoissonf}) on weighted Sobolev space, it follows that 
\be\label{wreggrad}
\|w\|_{\maK_{\mathbf{b}+1}^{m+1}(\Omega)} \leq C\|f\|_{\maK_{\mathbf{b}-1}^{m-1}(\Omega)},
\ee
which implies $\mathbf{F} = \textbf{curl }w \in [\maK_{\mathbf b}^{m}(\Omega)]^2 \subset [\maK_{\mathbf b}^{m-1}(\Omega)]^2$. Then the solution of the Stokes problem (\ref{stokes}) satisfies
\be\label{regwsb3.2}
\|\mathbf{u}\|_{[\maK_{\mathbf{c}+1}^{m+1}(\Omega)]^2} + \|p\|_{\maK_{\mathbf{c}}^{m}(\Omega)} \leq C\|\mathbf{F}\|_{[\maK_{{\mathbf c}-1}^{m-1}(\Omega)]^2}\leq C\|\mathbf{F}\|_{[\maK_{{\mathbf b}-1}^{m-1}(\Omega)]^2},
\ee
where $\mathbf{c}=(c_1, \cdots, c_N)$ with $c_i=\min\{b_{i}+1,a_i\}$.
Thus, the solution of the Poisson problem (\ref{eqnpoisson}) satisfies $\phi \in \maK_{\mathbf{c}+2}^{m+2}(\Omega)$.

Since the bilinear functional in (\ref{poissonfem0}) is coercive and continuous on $V_n^k$, so we have by C\'ea's Theorem, 
\be\label{ceathmgrade}
\|w-w_n\|_{H^1(\Omega)} \leq C \inf_{v \in V_n^k} \|w-v\|_{H^1(\Omega)}.
\ee
Recall that $\alpha_0 = \min_i\{\alpha_0^i\}$ given by (\ref{alpha0}), and $\beta_0 = \min_i\{\beta_0^i\}=\frac{\pi}{\omega}$ are the thresholds corresponding to the largest interior angle $\omega$,
then we have the following result.
\begin{lemma}\label{possion1error}
Set the grading parameters $\kappa_{Q_i}=2^{-\frac{\theta}{a_i}}\left(=2^{-\frac{\theta'}{b_i}}\right)$ with $0<a_i<\alpha_0^i$, $0< b_i <\beta_0^i$, $\theta$ being any constant satisfying $a_i\leq \theta \leq k$, and $\theta'=\min\left\{\frac{\beta_0}{\alpha_0}\max\{\theta, \alpha_0\}, k\right\}$ satisfying $b_i \leq \theta' \leq k$.
Let $w_n\in V_{n}^{k}$ be the solution of finite element solution of (\ref{poissonfem0}), and $w$ is the solution of the Poisson problem (\ref{eqnpoissonf}), then it follows
\be\label{phiH1errg3.1}
\|w-w_n\|_{H^1(\Omega)} \leq Ch^{\theta'}, \quad \|w-w_n\|\leq Ch^{\min\left\{2\theta', \theta'+1\right\}},
\ee
where $h:=2^{-n}$.
\end{lemma}
\begin{proof}
By (\ref{ceathmgrade}) and the interpolation error estimates in Lemma \ref{gradprojerr} under the regularity result in (\ref{wreggrad}) and $\kappa_{Q_i}=2^{-\frac{\theta}{a_i}}=2^{-\frac{\theta'}{b_i}}$,
we have the estimate  
$$
\|w-w_n\|_{H^1(\Omega)} \leq C \|w-w_I\|_{H^1(\Omega)} \leq Ch^{\theta'}.
$$

Consider the Poisson problem 
\be\label{dualVL2w}
-\Delta v = w - w_n \text{ in } \Omega, \quad v =0 \text{ on } \partial \Omega.
\ee
Then we have
\be\label{dualL2w}
\|w - w_n\|^2=(\nabla (w - w_n), \nabla v).
\ee
Subtract (\ref{poissonfem0}) from (\ref{eqnpoissonf}), we have the Galerkin orthogonality,
\be\label{poigo}
(\nabla (w - w_n), \nabla \psi) = 0 \quad \forall \psi \in V_n^k.
\ee
Setting $\psi=v_I\in V_{n}^{k}$ the nodal interpolation of $v$ and subtract (\ref{poigo}) from  (\ref{dualL2w}), we have
\be\label{dualL2w2}
\bal
\|w - w_n\|^2=(\nabla (w - w_n), \nabla (v-v_I)) \leq \|w-w_n\|_{H^1(\Omega)} \|v-v_I\|_{H^1(\Omega)}.
\eal
\ee
Similarly, the solution
$v\in \mathcal K^{2}_{\mathbf b'+1}(\Omega)$ satisfies the regularity estimate
\be\label{poiregw1}
\|v\|_{K^{2}_{\mathbf b'+1}(\Omega)} \leq C\|w-w_n\|_{K^{0}_{\mathbf b'-1}(\Omega)} \leq C\|w-w_n\|,
\ee
where the $i$th entry of $\mathbf{b}'$ satisfying $b'_i=\min\left\{b_i,1\right\}$.
By Lemma \ref{gradprojerr} with grading parameter $\kappa_{Q_i} =2^{-\frac{\theta'}{b_i}}(= 2^{-\frac{\theta}{a_i}})$ again,
we have the interpolation error
\be\label{dualintererr3.1}
\| v - v_I \|_{H^1(\Omega)} \leq Ch^{\min\{\theta', 1\}}\|v\|_{K^{2}_{\mathbf b'+1}(\Omega)}.
\ee
The $L^2$ error estimate in (\ref{phiH1errg3.1}) can be obtained by combining (\ref{dualL2w2}), (\ref{poiregw1}), and (\ref{dualintererr3.1}).
\end{proof}
Thus, we have the following result,
\be\label{Ferrs+}
\bal
\|\mathbf{F}-\mathbf{F}_n\|_{[L^2(\Omega)]^2} = &  \|\textbf{curl }v - \textbf{curl }v_n\|_{[L^2(\Omega)]^2} \leq \|w-w_n\|_{H^1(\Omega)} \leq Ch^{\theta'},\\
\|\mathbf{F}-\mathbf{F}_n\|_{[H^{-1}(\Omega)]^2} = &  \|\textbf{curl }w - \textbf{curl }w_n\|_{[H^{-1}(\Omega)]^2} \leq C\|w-w_n\| \leq Ch^{\min\{\theta'+1, 2\theta'\} }.
\eal
\ee

Now we have the following error estimate of the Mini element approximation ($k=1$) or Taylor-Hood method ($k\geq 2$) in Algorithm \ref{femalg} on graded meshes for the Stokes problem (\ref{stokes}). 

\begin{lemma}\label{StLBB}
The bilinear forms in both Mini element method and the Taylor-Hood method on graded meshes satisfies the LBB or inf-sup condition (\ref{skinfsupMini}) or (\ref{skinfsupTH}).
\end{lemma}
\begin{proof}
For given $\kappa = \min_i\{\kappa_{Q_i}\}$, Algorithm \ref{graded} implies that there exists a constant $\sigma(\kappa)>0$ such that
\be\label{meshtolbb}
h_T \leq \sigma(\kappa) \rho_T, \quad \forall T \in \maT_n,
\ee
where $h_T$ is the diameter of $T$, and $\rho_T$ is the maximum diameter of all circles contained in $T$. Under condition (\ref{meshtolbb}) of the graded mesh, the conclusion follows from \cite[Theorem 3.1]{Stenberg}.
\end{proof}

\begin{theorem}\label{Stokesgraderr}
Set the grading parameters $\kappa_{Q_i}=2^{-\frac{\theta}{a_i}}$ with $0<a_i<\alpha_0^i$ and $\theta$ being any constants satisfying $a_i\leq \theta \leq k$.
Let $(\mathbf{u}, p)$ be the solution of the Stokes problem (\ref{stokesweak}), and $(\mathbf{u}_n, p_n)$ be the Mini element solution ($k=1$)  or Taylor-Hood element solution ($k\geq 2$) on graded meshes $\mathcal T_n$.
If $(\mathbf{u}_n, p_n)$ is the solution of in Algorithm \ref{femalg}, then it follows
\begin{equation}\label{uh1errwsb}
\|\mathbf u-\mathbf  u_n\|_{[H^1(\Omega)]^2}+\|p-p_n\|\leq Ch^\theta;
\end{equation}
if it is the solution of in Algorithm \ref{femalg+}, then it follows
\begin{equation}\label{uh1errwsb+}
\|\mathbf u-\mathbf  u_n\|_{[H^1(\Omega)]^2}+\|p-p_n\|\leq Ch^{\min\{\theta, \theta'+1\}},
\end{equation}
where $\theta'$ is given in Lemma \ref{possion1error}.
\end{theorem}
\begin{proof}
For Algorithm \ref{femalg}, by Corollary \ref{stokesTHbdd+} and the interpolation error estimates in Lemma \ref{gradprojerr} under the regularity result in Lemma \ref{stokeswreg}, the estimate (\ref{uh1errwsb}) holds.

For Algorithm \ref{femalg+}, by Lemma \ref{stokesTHbdd} with the estimate (\ref{Ferrs+})
and the interpolation error estimates in Lemma \ref{gradprojerr} under the regularity result (\ref{regwsb3.2}), it follows
\bes
\bal
\|\mathbf{u}-\mathbf{u}_n\|_{[H^1(\Omega)]^2}+\|p-p_n\|\leq Ch^{\theta} + Ch^{\min\{\theta'+1, 2\theta'\} } \leq  Ch^{\min\{\theta, \theta'+1\}}.
\eal
\ees
Here, we have used the fact that if $\omega>\pi$, $\theta\leq \theta'<2\theta'$. Note that $\theta'=\min\left\{\frac{\beta_0}{\alpha_0}\max\{\theta, \alpha_0\}, k\right\}$, so if $\theta'=k$, then $\theta\leq k=\theta'$; otherwise
$$
\theta'=\frac{\beta_0}{\alpha_0}\max\{\alpha_0, \theta\}\geq \frac{\beta_0}{\alpha_0}\theta>\theta,
$$
where we have used (\ref{alpha2beta2}).

If $\omega<\pi$, by taking $1< b_i=\frac{\beta_0}{\alpha_0}  a_i<\beta_0\leq \beta_0^i$, it follows $\theta'=\frac{b_i}{a_i}\theta \geq b_i>1$, so that $\theta'+1<2\theta'$.
Thus, the estimate (\ref{uh1errwsb+}) holds.
\end{proof}

In weighted Sobolev space, the regularity result for (\ref{stokesdualreg}) with $l=0,1$ has the form
\be\label{stokesregwe}
\|\mathbf{r}\|_{[\maK_{\mathbf{a}'+1}^{l+2}(\Omega)]^2} + \|s\|_{\maK_{\mathbf{a}'}^{l+1}(\Omega)} \leq C \|\mathbf g\|_{[\maK_{\mathbf{a}'-1}^{l}(\Omega)]^2} \leq C\|\mathbf g\|_{[\maK_{\mathbf{b}}^{l}(\Omega)]^2},
\ee
where $0<\mathbf{a}'= \min\{\mathbf{a}, l+1\}$ with $0< \mathbf a < \bm\alpha_0$ and $0<\mathbf{b}<\bm{\beta}_0$ .
Then we have the following result.
\begin{theorem}\label{StokesgradL2err}
Set the grading parameters $\kappa_{Q_i}=2^{-\frac{\theta}{a_i}}$ with $0<a_i<\alpha_0^i$ and $\theta$ being any constants satisfying $a_i\leq \theta \leq k$.
Let $(\mathbf{u}, p)$ be the solution of the Stokes problem (\ref{stokesweak}), and $(\mathbf{u}_n, p_n)$ be the Mini element solution ($k=1$)  or Taylor-Hood element solution ($k\geq 2$) in Algorithm \ref{femalg+} on graded meshes $\mathcal T_n$.
If $(\mathbf{u}_n, p_n)$ is the solution of in Algorithm \ref{femalg}, then it follows
\begin{equation}\label{stokL2err}
\bal
\|\mathbf u-\mathbf  u_n\|_{[L^2(\Omega)]^2}\leq & Ch^{\min\{2\theta, \theta+1\}},\\
\|\mathbf u-\mathbf  u_n\|_{[(\maK_{\mathbf{b}}^{1}(\Omega))^*]^2}\leq & Ch^{\min\{2\theta, \theta+2\}};
\eal
\end{equation}
if it is the solution of in Algorithm \ref{femalg}, then it follows
\begin{equation}\label{stokL2err+}
\bal
\|\mathbf u-\mathbf  u_n\|_{[L^2(\Omega)]^2}\leq & Ch^{\min\{2\theta, \theta+1, \theta'+2\}},\\
\|\mathbf u-\mathbf  u_n\|_{[(\maK_{\mathbf{b}}^{1}(\Omega))^*]^2}\leq & Ch^{\min\{2\theta, \theta+2, k+\theta', \theta+\theta'+1, \theta'+3\}},
\eal
\end{equation}
where $\theta'$ is given in Lemma \ref{possion1error}.
Here, $(\cdot)^*$ represents the dual space.
\end{theorem}
\begin{proof}
We only prove (\ref{stokL2err+}) for Taylor-Hood method, all other cases can be proved similarly. 
Similar to Lemma \ref{stokesestlem+}, we take $\mathbf{v} = \mathbf u -\mathbf u_n$, $q=p-p_n$ in (\ref{stokesadweak}), 
then we have 
\bes
\bal
& \|\mathbf u -\mathbf u_n\|_{[(\maK_{\mathbf{b}}^{l}(\Omega))^*]^2} =  \sup_{g\in (\maK_{\mathbf{b}}^{l}(\Omega))]^2} \frac{\langle\mathbf g, \mathbf u -\mathbf u_n \rangle}{\|\mathbf g\|_{[(\maK_{\mathbf{b}}^{l}(\Omega))]^2}},
\eal
\ees
where $l=0,1$.
Let $(\mathbf r_n, s_n)$ be the Taylor-Hood solution of (\ref{stokeadjoint}), then it follows
\bes
\bal
\langle\mathbf g, \mathbf u -\mathbf u_n \rangle = & T_1+T_2+T_3+T_4,
\eal
\ees
where $T_i$, $i=1,\cdots, 4$ have the same expressions as those in Lemma \ref{stokesestlem+}.
For $\mathbf r_n$ and $s_n$, we have the following estimate in the weighted Sobolev space
\be\label{dualest}
\bal
& \|\mathbf r -\mathbf r_n\|_{[H^1(\Omega)]^2} + \|s-s_n\| \leq C \left(\inf_{\mathbf{r}_I \in [V_n^k(\Omega)]^2}\|\mathbf r -\mathbf r_I\|_{[H^1(\Omega)]^2} + \inf_{s_I\in S_n^{k-1}} \|s-s_I\| \right)\\
& \leq Ch^{\min\{\theta, l+1\}}
(\|\mathbf{r}\|_{[\maK_{\mathbf{a}'+1}^{l+2}(\Omega)]^2} + \|s\|_{\maK_{\mathbf{a}'}^{l+1}(\Omega)}).
\eal
\ee
Here we have
\bes
\bal
|T_j| \leq &  Ch^{\min\{\theta, \theta'+1\} + \min\{\theta, l+1\} } (\|\mathbf{r}\|_{[\maK_{\mathbf{a}'+1}^{l+2}(\Omega)]^2}+\|s\|_{\maK_{\mathbf{a}'}^{l+1}(\Omega)}) \\
= & Ch^{\min\{2\theta, \theta+l+1, \theta+\theta'+1, \theta'+l+2\}}(\|\mathbf{r}\|_{[\maK_{\mathbf{a}'+1}^{l+2}(\Omega)]^2}+ \|s\|_{\maK_{\mathbf{a}'}^{l+1}(\Omega)}),
\eal
\ees
where $j=1,2,3$.

Note that
\bes
\bal
T_4 = T_{41}+T_{42}=\langle \mathbf{F}-\mathbf{F}_n, \mathbf{r}_n-\mathbf{r}\rangle + \langle \mathbf{F}-\mathbf{F}_n, \mathbf{r}\rangle.
\eal
\ees
By (\ref{Ferrs+}) and (\ref{dualest}), we have
\bes
\bal
|T_{41}| \leq &  \|\mathbf{F}-\mathbf{F}_n\|_{[H^{-1}(\Omega)]^2}\|\mathbf{r}-\mathbf{r}_n\|_{[H^{1}(\Omega)]^2} \leq Ch^{\min\{\theta'+1, 2\theta'\}+\min\{\theta, l+1 \} }\|\mathbf{r}\|_{[\maK_{\mathbf{a}'+1}^{l+2}(\Omega)]^2}\\
=& Ch^{\min\{\theta+\theta'+1,\theta'+l+2, 2\theta'+l+1, \theta+2\theta'\} }\|\mathbf{r}\|_{[\maK_{\mathbf{a}'+1}^{l+2}(\Omega)]^2}.
\eal
\ees
By Theorem \ref{gradprojerr} for $\psi \in \maK_{\mathbf{a}'+2}^{l+3}(\Omega)$ satisfying (\ref{rtopis}) with $\kappa_{Q_i}=2^{-\frac{\theta}{a_i}}=2^{-\frac{\theta_1}{1+a'_i}}$, we have
\be\label{psierrwi}
\|\psi-\psi_I\|_{H^1(\Omega)} \leq Ch^{\min\{k, \theta_1, l+2\}} \|\mathbf{r}\|_{[\maK_{\mathbf{a}'+1}^{l+2}(\Omega)]^2},
\ee
where $\theta_1 = (1 + \frac{1}{a'_i})\theta \geq  (1 + \frac{1}{a_i})\theta \geq \theta +1$ and $\psi_I$ is the nodal interpolation of $\psi$.

By (\ref{Ferrs+}) and (\ref{psierrwi}), we have 
\bes
\bal
|T_{42}| \leq & \|\mathbf{F}-\mathbf{F}_n\|_{[L^2(\Omega)]^2}\|\psi-\psi_I\|_{H^1(\Omega)} \leq Ch^{\theta' + \min\{k, \theta+\theta'_1, l+2\} }\|\mathbf{r}\|_{[\maK_{\mathbf{a}'+1}^{l+2}(\Omega)]^2}\\
\leq & Ch^{\min\{k+\theta',\theta'+l+2, \theta+\theta'+\theta'_1\} }\|\mathbf{r}\|_{[\maK_{\mathbf{a}'+1}^{l+2}(\Omega)]^2},
\eal
\ees
where $\theta'_1 = \min_i\{\frac{1}{a'_i}\theta\} \geq 1$.
It can be verified that
$$
|T_4| \leq |T_{41}|+|T_{42}| \leq  Ch^{\min\{k+\theta', \theta+\theta'+1,\theta'+l+2, 2\theta'+l+1, \theta+2\theta'\} }\|\mathbf{r}\|_{[\maK_{\mathbf{a}'+1}^{l+2}(\Omega)]^2}.
$$

By the regularity (\ref{stokesregwe}) and the summation of estimates $|T_i|$, $i=1, \cdots, 4$, and $\theta<2\theta'$, we have the error estimate 
\be\label{dualuerr}
\|\mathbf u - \mathbf u_n\|_{[(\maK_{\mathbf{b}}^{1}(\Omega))^*]^2} \leq Ch^{\min\{2\theta, \theta+l+1, k+\theta', \theta+\theta'+1, \theta'+l+2\}}.
\ee
Recall that $k\geq 1$, $\theta\leq k$, and when $\omega>\pi$, we have $\theta<\theta'$, then it follows
\be\label{k1ktheta}
\theta+1 \leq k+\theta',
\ee
and when $\omega<\pi$, we have $\theta'>1$, so the inequality (\ref{k1ktheta}) still holds.
The estimates in (\ref{stokL2err+}) follows from (\ref{dualuerr}) with the fact (\ref{k1ktheta}).
\end{proof}

\begin{remark}\label{stokesoptimal}
By Theorem \ref{Stokesgraderr} and Theorem \ref{StokesgradL2err},
we can find that if we take
\be\label{thetastokes}
\theta = k
\ee
in the grading parameter $\kappa_{Q_i}$, then we can obtain the optimal convergence rate for the Stokes approximations in Algorithm \ref{femalg},
\begin{subequations}\label{stokeserr2}
\begin{align}
\|\mathbf{u}-\mathbf{u}_n\|_{[H^1(\Omega)]^2} + \|p-p_n\| \leq Ch^{k},\\
\|\mathbf{u}-\mathbf{u}_n\|_{[L^2(\Omega)]^2} \leq Ch^{k+1}.
\end{align}
\end{subequations}
For Algorithm \ref{femalg+}, it follows
\begin{subequations}\label{stokeserr2+}
\begin{align}
\|\mathbf{u}-\mathbf{u}_n\|_{[H^1(\Omega)]^2} + \|p-p_n\| \leq Ch^{\min\left\{\frac{\beta_0}{\alpha_0}\max\{k, \alpha_0\}+1, k\right\}},\\
\|\mathbf{u}-\mathbf{u}_n\|_{[L^2(\Omega)]^2} \leq Ch^{\min\left\{ k+1, \frac{\beta_0}{\alpha_0}\max\{k, \alpha_0\}+2 \right\} }.
\end{align}
\end{subequations}
However, to obtain the optimal convergence rate for the biharmonic approximation, the convergence rates of Mini element or Taylor-Hood element approximations don't have to be optimal. Therefore, we shall figure out the admissible parameters $\theta$ such that the convergence rate of the biharmonic approximation is optimal.
\end{remark}


\begin{theorem}\label{phierrthmg}
Set the grading parameters $\kappa_{Q_i}=2^{-\frac{\theta}{a_i}}$ with $0<a_i<\alpha_0^i$ and
\be\label{thetaopt}
\theta=\max\{k-1, a'_i\},
\ee
where $a'_i=\min\{\alpha_0, a_i\} \leq \alpha_0$ for $\alpha_0$ given by (\ref{alpha0}).
Let $\phi_n\in V_{n}^{k}$ be the solution of finite element solution of (\ref{poissonfem}), and $\phi$ is the solution of the biharmonic problem (\ref{eqnbi}). If $\phi_n$ is the solution in Algorithm \ref{femalg}, then it follows
\be\label{phiH1errg}
\|\phi-\phi_n\|_{H^1(\Omega)} \leq Ch^{k};
\ee
it is the solution in Algorithm \ref{femalg+}, then it follows
\be\label{phiH1errg+}
\|\phi-\phi_n\|_{H^1(\Omega)} \leq Ch^{\min\left\{k, \max\left\{\frac{\beta_0}{\alpha_0}(k-1)+2, \beta_0+2\right\}\right\}}.
\ee
\end{theorem}
\begin{proof}
Denote by $\phi_I \in V_{n}^{k}$ the nodal interpolation of $\phi$. Similar to Theorem \ref{phierrthm3.1}, we have
\be\label{phih1int}
\bal
\|\phi-\phi_n\|_{H^1(\Omega)} \leq & C\left( \|\phi-\phi_I\|_{H^1(\Omega)}+  \|\mathbf{u}-\mathbf{u}_n\|_{[L^2(\Omega)]^2} \right)
\eal
\ee
Recall that $\phi \in \maK_{\mathbf a+2}^{m+2}(\Omega)=\maK_{ (\mathbf a+1)+1}^{(m+1)+1}(\Omega)$ with $m\geq k$, so by Lemma \ref{gradprojerr} with grading parameter $\kappa_{Q_i}=2^{-\frac{\theta_1}{1+a_i}}(= 2^{-\frac{\theta}{a_i}})$ with  $\theta $ given in (\ref{thetaopt}) and $\theta_1=\frac{1+a_i}{a_i}\theta=\theta+\frac{1}{a_i}\theta\geq  \theta+1 \geq k$, we have
\be\label{phigradprojerr}
\|\phi-\phi_I\|_{H^1(\Omega)} \leq Ch^{\min\{k,\theta_1\}} = Ch^{k}.
\ee
For $\theta $ given in (\ref{thetaopt}), Theorem \ref{StokesgradL2err} indicates for Algorithm \ref{femalg},
\begin{equation}\label{stokL2err1}
\|\mathbf u-\mathbf  u_n\|_{[L^2(\Omega)]^2}\leq Ch^{\min\left\{\max\{2\alpha_0, 2(k-1)\},\max\{\alpha_0+1, k\}\right\}}.
\end{equation}
Plugging (\ref{phigradprojerr}) and (\ref{stokL2err1}) into (\ref{phih1int}), the estimate (\ref{phiH1errg}) holds.

For Algorithm \ref{femalg+}, we have
\begin{equation}\label{stokL2err1+}
\|\mathbf u-\mathbf  u_n\|_{[L^2(\Omega)]^2}\leq Ch^{\min\left\{\max\{2\alpha_0, 2(k-1)\},\max\{\alpha_0+1, k\}, \frac{\beta_0}{\alpha_0}\max\{k-1, \alpha_0\}+2          \right\} }.
\end{equation}
By plugging (\ref{phigradprojerr}) and (\ref{stokL2err1+}) into (\ref{phih1int}), it follows the estimate (\ref{phiH1errg}).
\end{proof}

\begin{theorem}\label{phierrL2thm1}
Set the grading parameters $\kappa_{Q_i}=2^{-\frac{\theta}{a_i}}$ with $0<a_i<\alpha_0^i$ and $\theta$ given by
\be\label{thetaopt2}
\theta=\max\left\{k-1, \frac{k+1}{2}, a'_i\right\},
\ee
where $a'_i=\min\{\alpha_0, a_i\} \leq \alpha_0$ for $\alpha_0$ given by (\ref{alpha0}).
Let $\phi_n$ be the solution of finite element solution of (\ref{poissonfem}), and $\phi$ be the solution of the biharmonic problem (\ref{eqnbi}). 
If $\phi_n$ is the solution in Algorithm \ref{femalg}, then it follows
\be\label{phiL2errg}
\|\phi-\phi_n\| \leq Ch^{k+1}.
\ee
it is the solution in Algorithm \ref{femalg+}, then it follows
\be\label{phiL2errg+}
\|\phi-\phi_n\| \leq Ch^{\min\left\{k+1, \max\left\{ \frac{\beta_0}{\alpha_0}(k-1)+2, \beta_0+2 \right\}+1\right\}}.
\ee
\end{theorem}
\begin{proof}
Set $\psi = v_I \in V_{n}^{k}$ the nodal interpolation of $v$ of the Poisson problem (\ref{dualL2}). Similar to Theorem \ref{phierrL2thm}, we have
\be\label{dualL21}
\bal
\|\phi - \phi_n\|^2
\leq & \|\phi-\phi_n\|_{H^1(\Omega)} \|v - v_I\|_{H^1(\Omega)} + \|\mathbf{u} - \mathbf{u}_n\|_{[L^2(\Omega)]^2} \|v-v_I\|_{H^1(\Omega)} \\
& + \|\mathbf{u} - \mathbf{u}_n\|_{[(\maK_{\mathbf b}^{1}(\Omega))^*]^2}\|\textbf{curl }v\|_{[\maK_{\mathbf b}^{1}(\Omega)]^2}:=T_1+T_2+T_3.
\eal
\ee
Based on the results in \cite{BNZ205}, the solution
$v\in \mathcal K^{2}_{\mathbf b'+1}(\Omega)$ satisfies the regularity estimate
\be\label{poiregw}
\|v\|_{K^{2}_{\mathbf b'+1}(\Omega)} \leq C\|\phi-\phi_n\|,
\ee
where the $i$th entry of $\mathbf{b}'$ is given by $b'_i=\min\left\{b_i, 1\right\}$ with $b_i<\frac{\pi}{\omega_i}$.
If $\omega>\pi$, we have 
$$
\theta'\geq \theta \geq\frac{k+1}{2}\geq 1,
$$
so it follows the interpolation error
\be\label{dualintererr}
\| v - v_I \|_{H^1(\Omega)} \leq Ch^{\min\{\theta', 1\}}\|v\|_{K^{2}_{\mathbf b'+1}(\Omega)}=Ch\|v\|_{K^{2}_{\mathbf b'+1}(\Omega)}.
\ee
If $\omega<\pi$, the interpolation error (\ref{dualintererr}) is obvious since $v\in H^2(\Omega)$.

For Algorithm \ref{femalg}, we have the following estimate for each $T_i$, $i=1,2,3$.
By Theorem \ref{phierrthmg} and (\ref{dualintererr}), it follows
$$
T_1 =\|\phi-\phi_n\|_{H^1(\Omega)} \|v - v_I\|_{H^1(\Omega)}\leq Ch^{k+1}\|v\|_{K^{2}_{\mathbf b'+1}(\Omega)}.
$$
By Theorem \ref{StokesgradL2err}  and (\ref{dualintererr}), it follows
$$
T_2 = \|\mathbf{u} - \mathbf{u}_n\|_{[L^2(\Omega)]^2} \|v-v_I\|_{H^1(\Omega)}\leq Ch^{k+2}\|v\|_{K^{2}_{\mathbf b'+1}(\Omega)}.
$$
Again, by Theorem \ref{StokesgradL2err}, it follows
$$
T_3 \leq C\|\mathbf{u} - \mathbf{u}_n\|_{[(\maK_{\mathbf b'}^{1}(\Omega))^*]^2}\|v\|_{\maK_{\mathbf b'+1}^{2}(\Omega)} \leq Ch^{k+1}\|v\|_{K^{2}_{\mathbf b'+1}(\Omega)}.
$$
Thus, the regularity estimate (\ref{poiregw}) and the summation of $T_i$, $i=1,2,3$ give the estimate (\ref{phiL2errg}).

For Algorithm \ref{femalg+}, we have the following estimate for $T_i$.
By Theorem \ref{phierrthmg} and (\ref{dualintererr}), it follows
$$
T_1 \leq Ch^{\min\left\{k+1, \max\left\{ \frac{\beta_0}{\alpha_0}(k-1)+2, \beta_0+2 \right\}+1\right\}}\|v\|_{K^{2}_{\mathbf b'+1}(\Omega)}.
$$
By Theorem \ref{StokesgradL2err}  and (\ref{dualintererr}), it follows
$$
T_2 \leq Ch^{\min\left\{k+2, \max\left\{ \frac{\beta_0}{\alpha_0}(k-1)+2, \beta_0+2 \right\}+2\right\}}\|v\|_{K^{2}_{\mathbf b'+1}(\Omega)}.
$$
Again,  it follows by Theorem \ref{StokesgradL2err},
$$
T_3 \leq Ch^{\min\left\{k+1, \max\left\{ \frac{\beta_0}{\alpha_0}(k-1)+2, \beta_0+2 \right\}+1\right\}}\|v\|_{K^{2}_{\mathbf b'+1}(\Omega)}.
$$
Thus, the regularity estimate (\ref{poiregw}) and the summation of $T_i$, $i=1,2,3$ again give the estimate (\ref{phiL2errg+}).

\end{proof}

\begin{remark}
For the results in Theorem \ref{phierrthmg} and Theorem \ref{phierrL2thm1}, we have the following facts.
\begin{itemize}
\item If $k=1$ in (\ref{thetaopt}), then $\theta=a_i$, $i=1,\cdots, N$ gives $\kappa_{Q_i}=\frac{1}{2}$, which indicates the mesh is exactly the uniform mesh.
\item With the grading parameters given, the finite element approximations $\phi_n$ from Algorithm \ref{femalg+} achieve the optimal convergence rates at least for $k= 1,2,3$ in both $H^1$ and $L^2$ norm. Moreover, the optimal convergence rates can be obtained for any $k\geq 1$  when $\omega>\pi$ by (\ref{alpha2beta2}). 

\item With the grading parameters given, the finite element approximations $\phi_n$ from Algorithm \ref{femalg} can achieve optimal convergence rate for any $k \geq 1$.

\item Based on the regularity (\ref{regwsb3.2}) for the involved Stokes problem in Algorithm \ref{femalg+}, if we take the grading parameter $\kappa_{Q_i}=2^{-\frac{\theta}{c_i}}\left(\leq 2^{-\frac{\theta}{a_i}}\right)$, then the corresponding error estimates (\ref{phiH1errg3.1}),  (\ref{uh1errwsb+}), (\ref{stokL2err+}) and (\ref{phiH1errg+}) still hold. 
\item By (\ref{eqnpoisson}) and $\phi \in H_0^2(\Omega)$, we have that
$$
\|\phi\|_{H^2(\Omega)} \leq C\|\mathbf{u}\|_{[H^1(\Omega)]^2},
$$
so the error $\|\mathbf{u}-\mathbf{u}_n\|_{[H^1(\Omega)]^2}$
is an estimate of the solution error to $\phi$ in $H^2$ norm.
\end{itemize}
\end{remark}

\section{Numerical illustrations}\label{sec-5}
In this section, we present numerical tests to validate our theoretical predictions for the proposed finite element algorithm for solving the biharmonic problem (\ref{eqnbi}).
If an exact solution (or vector) $v$ is unknown, we use the following numerical convergence rate
\begin{eqnarray}\label{rate}
{\mathcal R}=\log_2\frac{|v_j-v_{j-1}|_{[H^l(\Omega)]^{l'}}}{|v_{j+1}-v_j|_{[H^l(\Omega)]^{l'}}},
\end{eqnarray}
$l=0,1$ as an  indicator of the actual convergence rate \cite{LN18}.
Here $v_j$ denotes the finite element solution on the  mesh $\mathcal T_j$ obtained after $j$  refinements of the initial triangulation $\mathcal T_0$. For scalar functions, we take $l'=1$, otherwise, $l'=2$. So if $v_j = w_j, \phi_j, p_j$, we take $l'=1$; if $v_j=\mathbf{u}_j$, we take $l'=2$.

To test the performance of Algorithm \ref{femalg+} and Algorithm \ref{femalg} for solving the biharmonic problem (\ref{eqnbi}), we shall use the $H^2$-conforming Argyris finite element approximation \cite{argyris1968tuba} as a reference solution $\phi_R$, which is computed on the same mesh as that for Algorithm \ref{femalg}. Since the solution of the $H^2$-conforming finite element method converges to the exact solution $\phi$ regardless of the convexity of the domain as the mesh is refined, so we can use $\phi_R$ as a good approximation of the exact solution $\phi$.

Since the convergence rate of the finite element approximation $w_j$ of the Poisson equation in Algorithm \ref{femalg+} has been well investigated in many papers (see e.g., \cite{lyz2020, li2021}), so we will not report the convergence rates of $w_j$ in the following numerical tests.

\begin{example}\label{ex5.1}
We solve the bibarmonic problem (\ref{eqnbi}) with $f=1$ using Algorithm \ref{femalg} based on polynomials with $k=2$ on uniform meshes obtained by the midpoint refinements.
The source term of involved Stokes problem (\ref{stokes}) in Algorithm \ref{femalg} is taken as $\mathbf{F}=(0,x)^T$, which satisfies (\ref{curlFf}).

\noindent\textbf{Test case 1.} We first consider this problem in a square domain $\Omega=(-1,1)^2$ with the initial mesh given in Figure \ref{Mesh_InitS}a.
The errors in $L^\infty$ norm between the finite element solution $\phi_j$ and the reference solution $\phi_R$ are given in Table \ref{MaxErrSquare}. The finite element solution and its difference with the reference solution are shown in \ref{Mesh_InitS}b and \ref{Mesh_InitS}c, respectively. These results indicate that the solution of Algorithm \ref{femalg} converges to the exact solution.
In table \ref{Sqr_Poi_Rates}, we report the $H^1$ and $L^2$ convergence rates of the finite element solutions $\phi_j$, respectively. The results indicate that optimal convergence rates are obtained for finite element solutions of the biharmonic problem.

In addition, the Taylor-Hood element approximations $\mathbf{u}_7$ and $p_7$ for the involved Stokes problem are shown in Figure \ref{Mesh_InitS}d-\ref{Mesh_InitS}f.
In Table \ref{Sqr_Poi_Rates}, we also report the $H^1$ and/or $L^2$ convergence rates of the Taylor-Hood element approximations $\mathbf{u}_j$ and $p_j$, the results imply that optimal convergence rates are obtained for Taylor-Hood element approximations of the Stokes problem.

These results are consistent with our expectation in Lemma \ref{stokesestlem}, Theorem \ref{phierrthm3.1}, and Theorem \ref{phierrL2thm} for the biharmonic problem (\ref{eqnbi}) and the involved Stokes problem (\ref{stokes}) in a convex domain.

\begin{table}[!htbp]\tabcolsep0.04in
\caption{The $L^\infty$ error $\|\phi_R-\phi_j\|_{L^\infty(\Omega)}$ in the square domain  on quasi-uniform meshes.}
\begin{tabular}[c]{|c|c|c|c|c|c|c|}
\hline
\multirow{2}{*}{} & $j=3$ & $j=4$ & {$j=5$} & {$j=6$} &{$j=7$} &{$j=8$} \\
\hline
$k=2$  &  6.66250e-05 & 9.10404e-06 & 1.17798e-06 & 1.49996e-07 & 1.89341e-08 & 2.37885e-09 \\
\cline{1-7}
\end{tabular}\label{MaxErrSquare}
\end{table}

\begin{figure}[h]
\centering
\subfigure[]{\includegraphics[width=0.28\textwidth]{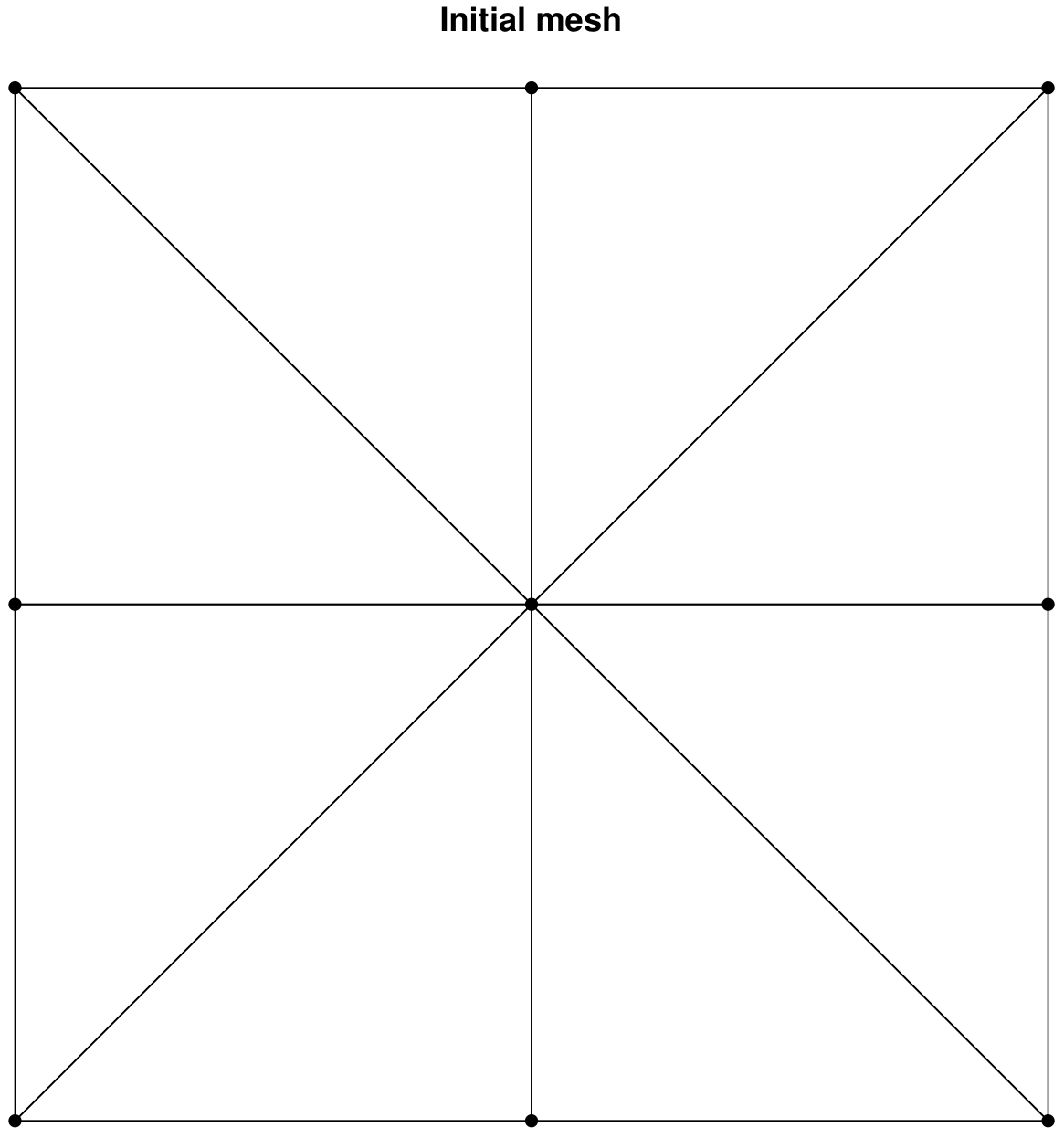}}
\subfigure[]{\includegraphics[width=0.325\textwidth]{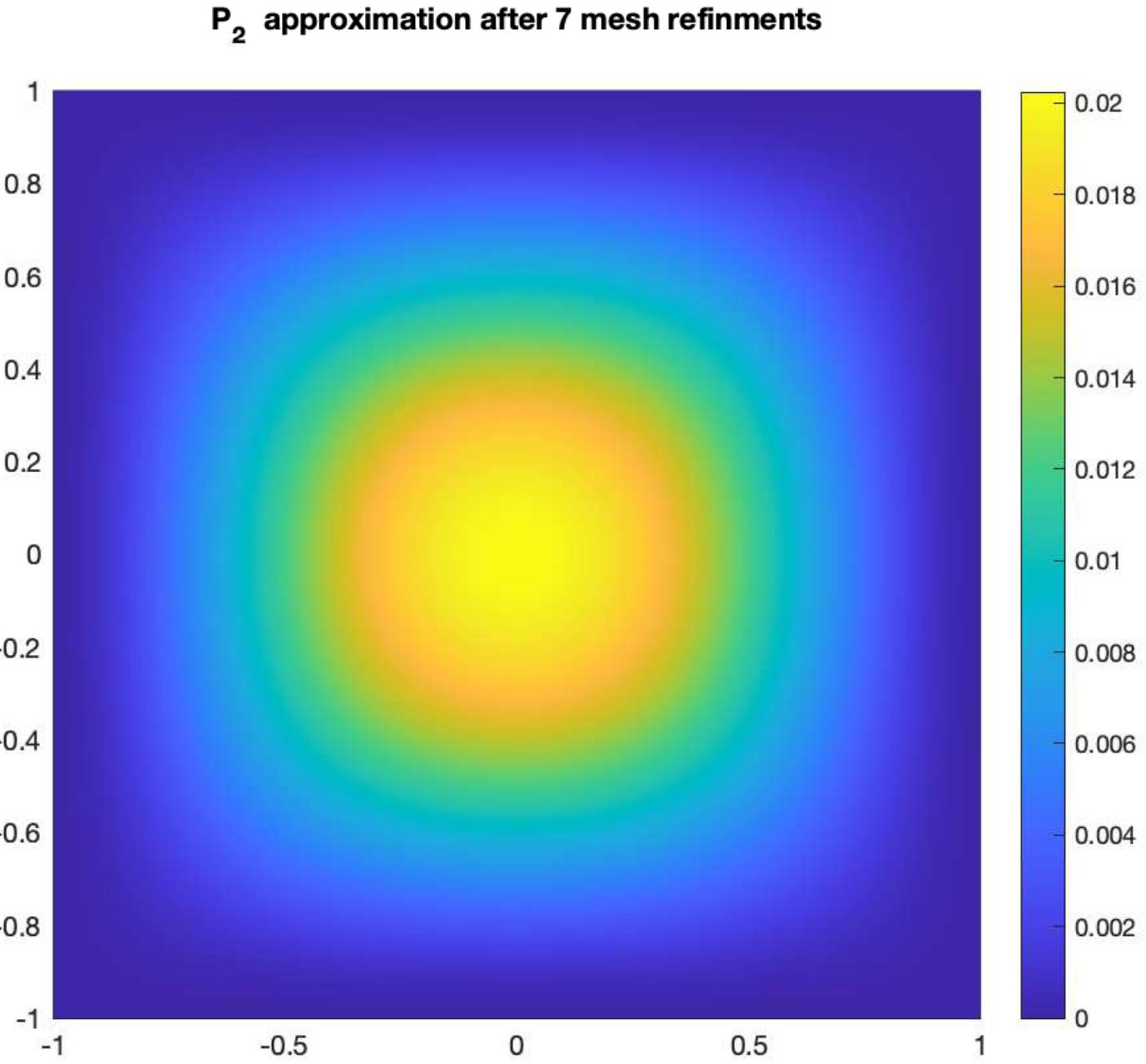}}
\subfigure[]{\includegraphics[width=0.325\textwidth]{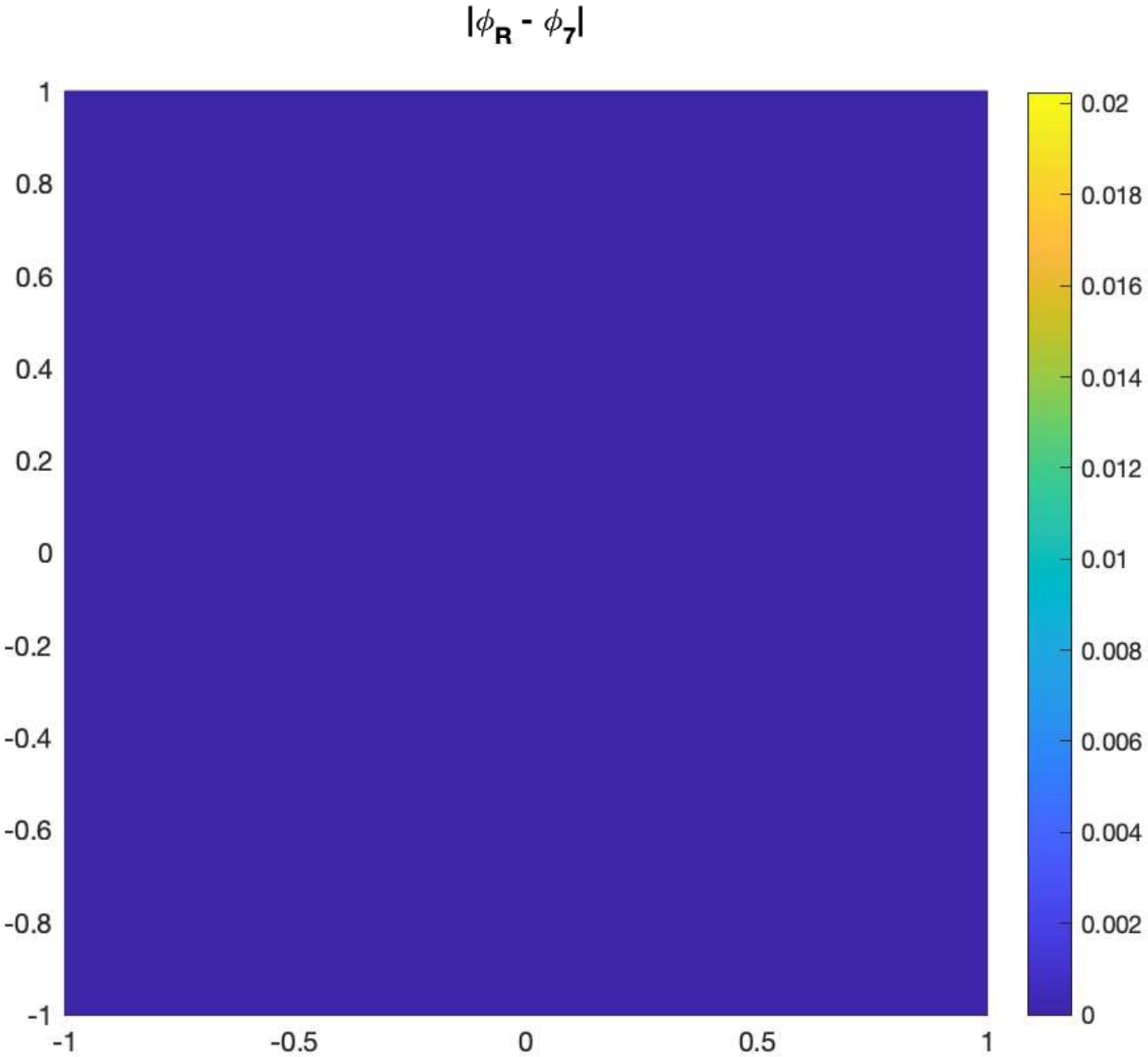}}
\subfigure[]{\includegraphics[width=0.325\textwidth]{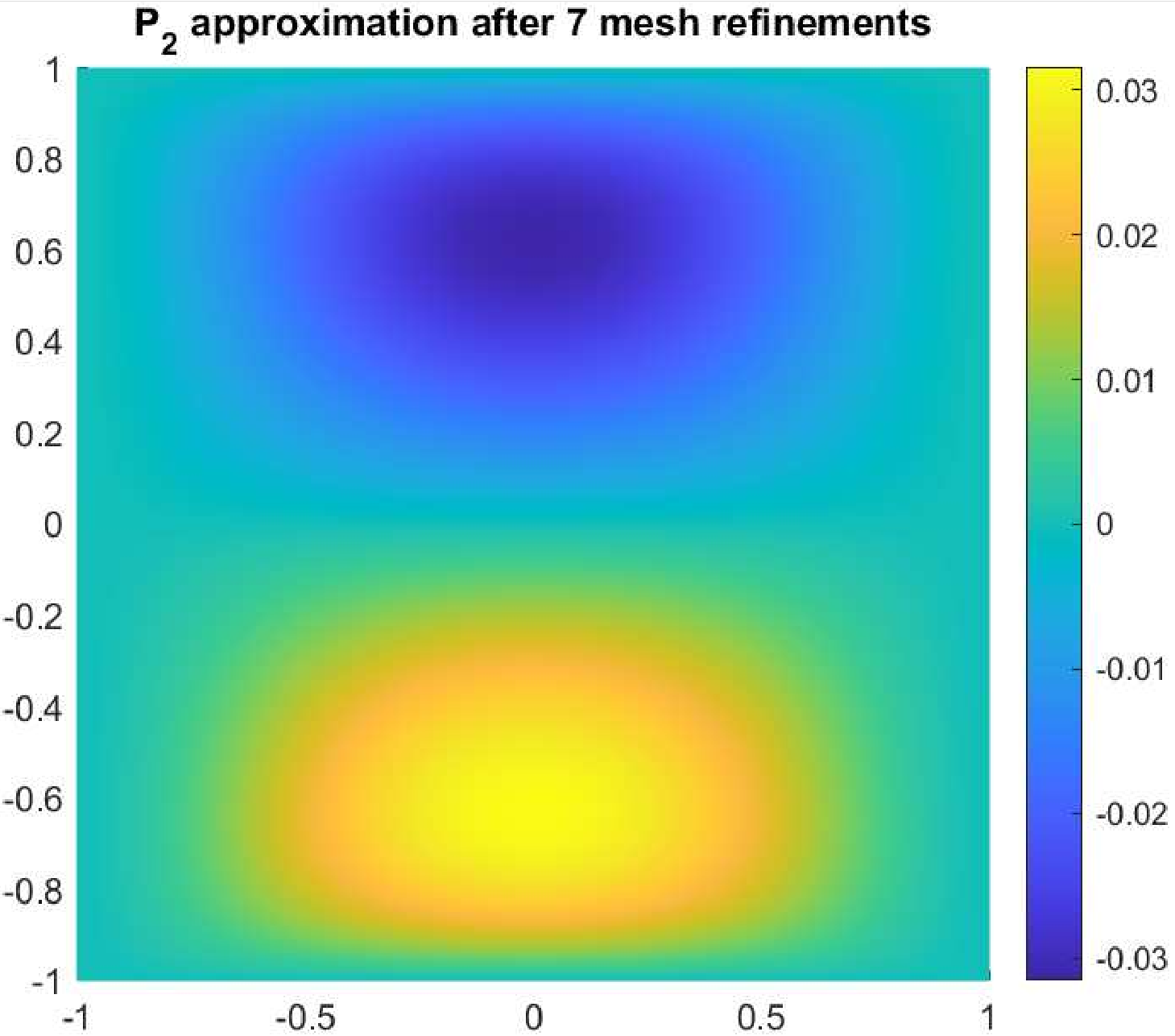}}
\subfigure[]{\includegraphics[width=0.325\textwidth]{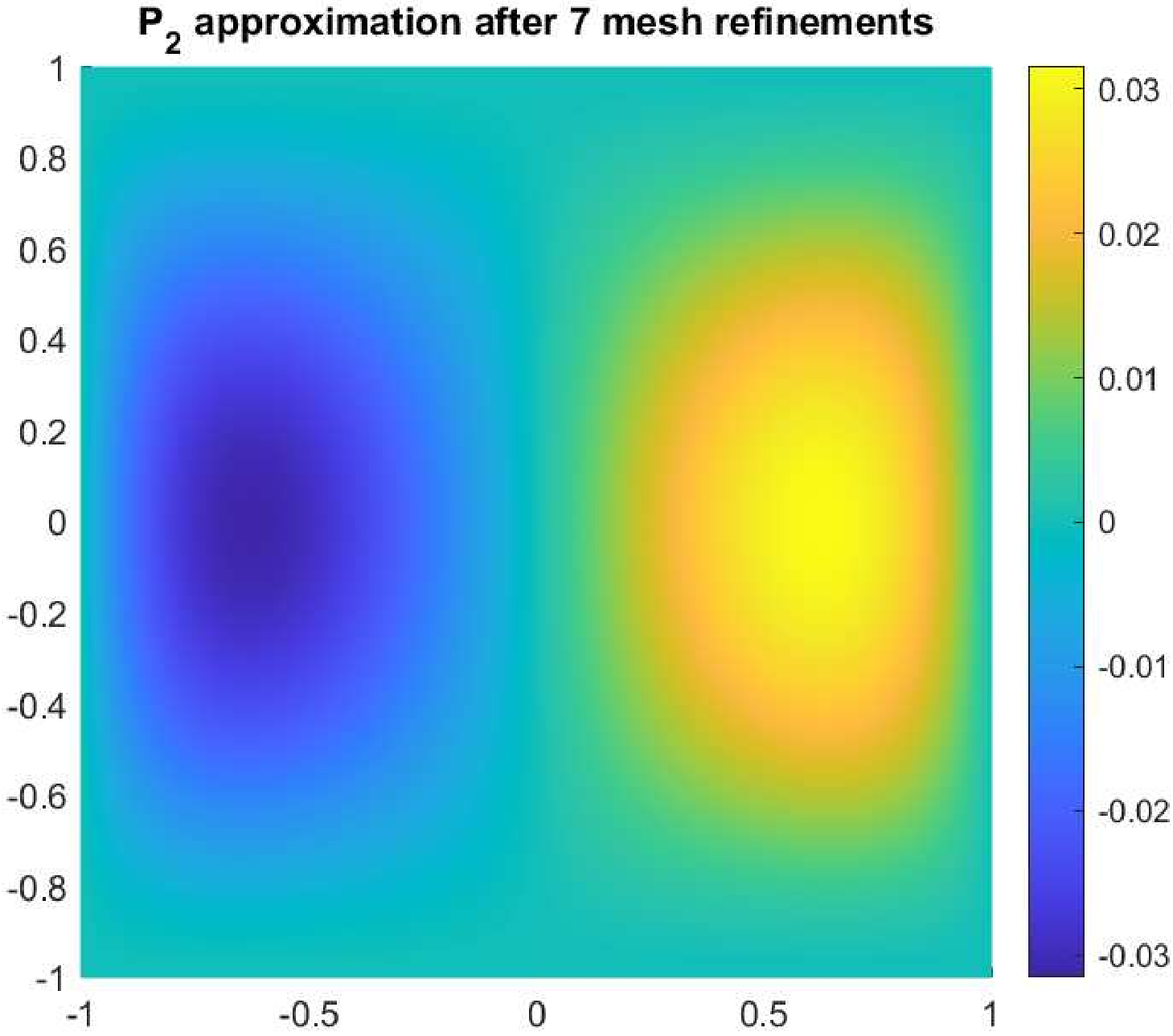}}
\subfigure[]{\includegraphics[width=0.325\textwidth]{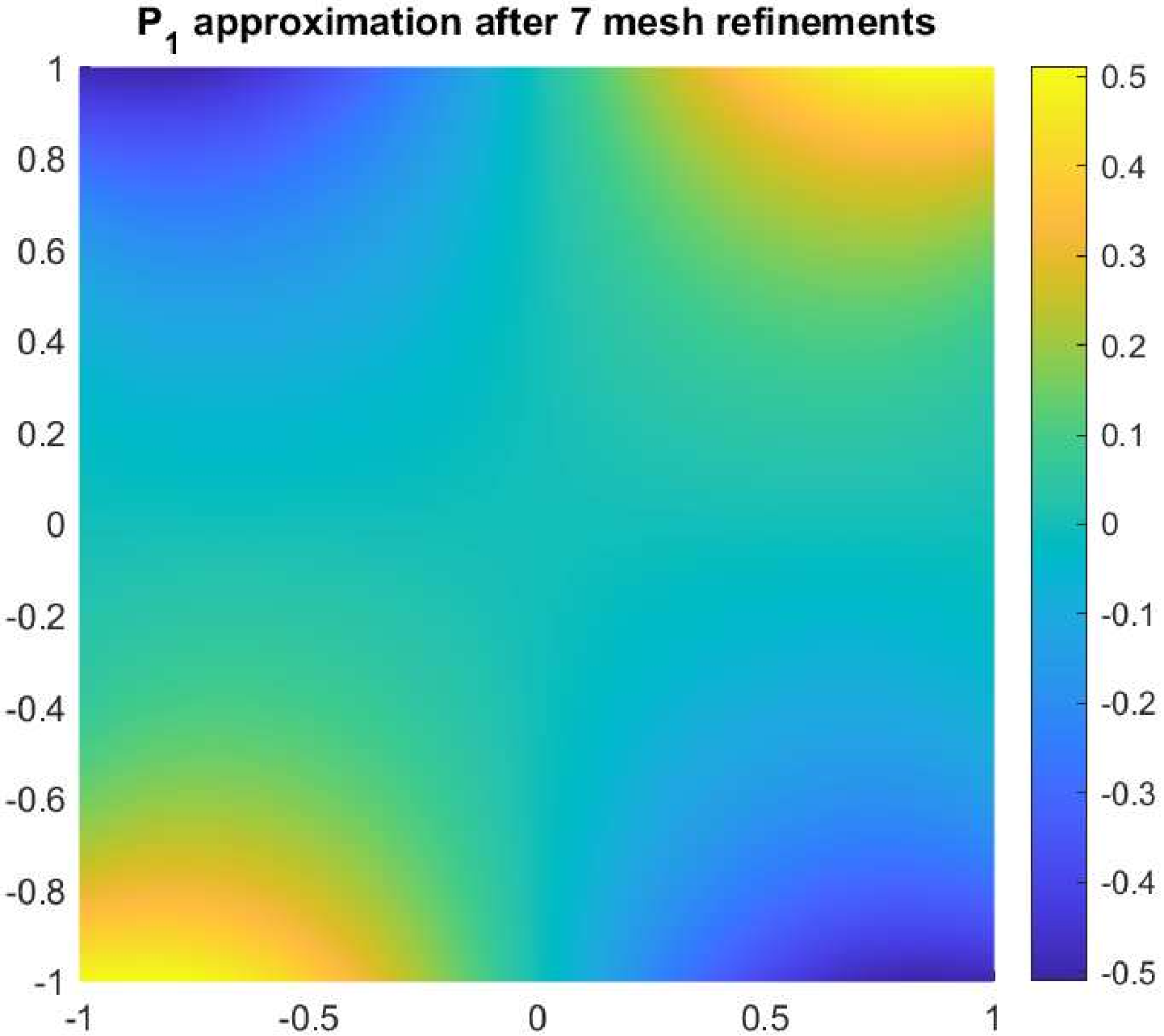}}
\caption{The Square domain (Example \ref{ex5.1}):  (a) the initial mesh; (b) $P_2$ finite element approximation
$\phi_7$; (c) the difference $|\phi_R-\phi_7|$; (d) Taylor-Hood element approximation $u_1$ of $\mathbf{u}_7$; (e) Taylor-Hood element approximation $u_2$ of $\mathbf{u}_7$; (d) Taylor-Hood element approximation $p_7$. }\label{Mesh_InitS}
\end{figure}

\begin{table}[!htbp]\tabcolsep0.04in
\caption{Numerical convergence rates in the square domain on uniform meshes.}
\centering
\begin{tabular}{ |c| c|  c| c|| c |c |c|} \hline
 &  & {$H^1$ rate of $\phi_j$}& {$L^2$ rate of $\phi_j$}&{$H^1$  rate of $\mathbf{u}_j$}&{$L^2$ rate of $\mathbf{u}_j$ } & { $L^2$ rate of $p_j$}\\
\hline
\multirow{5}{*}{$k=2$}
&$j=4$ &1.96	&3.00	&1.99	&3.01& 2.03\\
&$j=5$ &1.99    &3.00   &2.00   &3.02& 2.02\\
&$j=6$ &2.00	&3.00	&2.00	&3.01 & 2.01\\
&$j=7$ &2.00	&3.00	&2.00	&3.00 & 2.00\\
&$j=8$ &2.00	&3.00	&2.00	&3.00 & 2.00\\
\hline
\end{tabular}\label{Sqr_Poi_Rates}
\end{table}

\noindent\textbf{Test case 2.} We then consider this problem in an L-shaped domain $\Omega=\Omega_0\setminus\Omega_1$ with $\Omega_0=(-1,1)^2$ and $\Omega_1=(0,1)\times(-1,0)$ based on the initial mesh given in Figure \ref{Mesh_InitL}a.
The error $\|\phi_R-\phi_j\|_{L^\infty(\Omega)}$ is given in Table \ref{MaxErrLshaped}. The finite element solution and its difference with the reference solution are shown in Figure \ref{Mesh_InitL}b and \ref{Mesh_InitL}c, respectively. These results indicate that the solutions of Algorithm \ref{femalg} converge to the exact solution in a nonconvex polygonal domain.

\begin{table}[!htbp]\tabcolsep0.04in
\caption{The $L^\infty$ error $\|\phi_R-\phi_j\|_{L^\infty(\Omega)}$ in the L-shaped domain  on quasi-uniform meshes.}
\begin{tabular}[c]{|c|c|c|c|c|c|c|}
\hline
\multirow{2}{*}{} & $j=3$ & $j=4$ & {$j=5$} & {$j=6$} &{$j=7$} &{$j=8$} \\
\hline
$k=2$  &  8.74987e-04 & 3.94122e-04 & 1.77980e-04 & 8.26205e-05 & 3.86434e-05 & 1.81330e-05\\
\cline{1-7}
\end{tabular}\label{MaxErrLshaped}
\end{table}

\begin{figure}[h]
\centering
\subfigure[]{\includegraphics[width=0.28\textwidth]{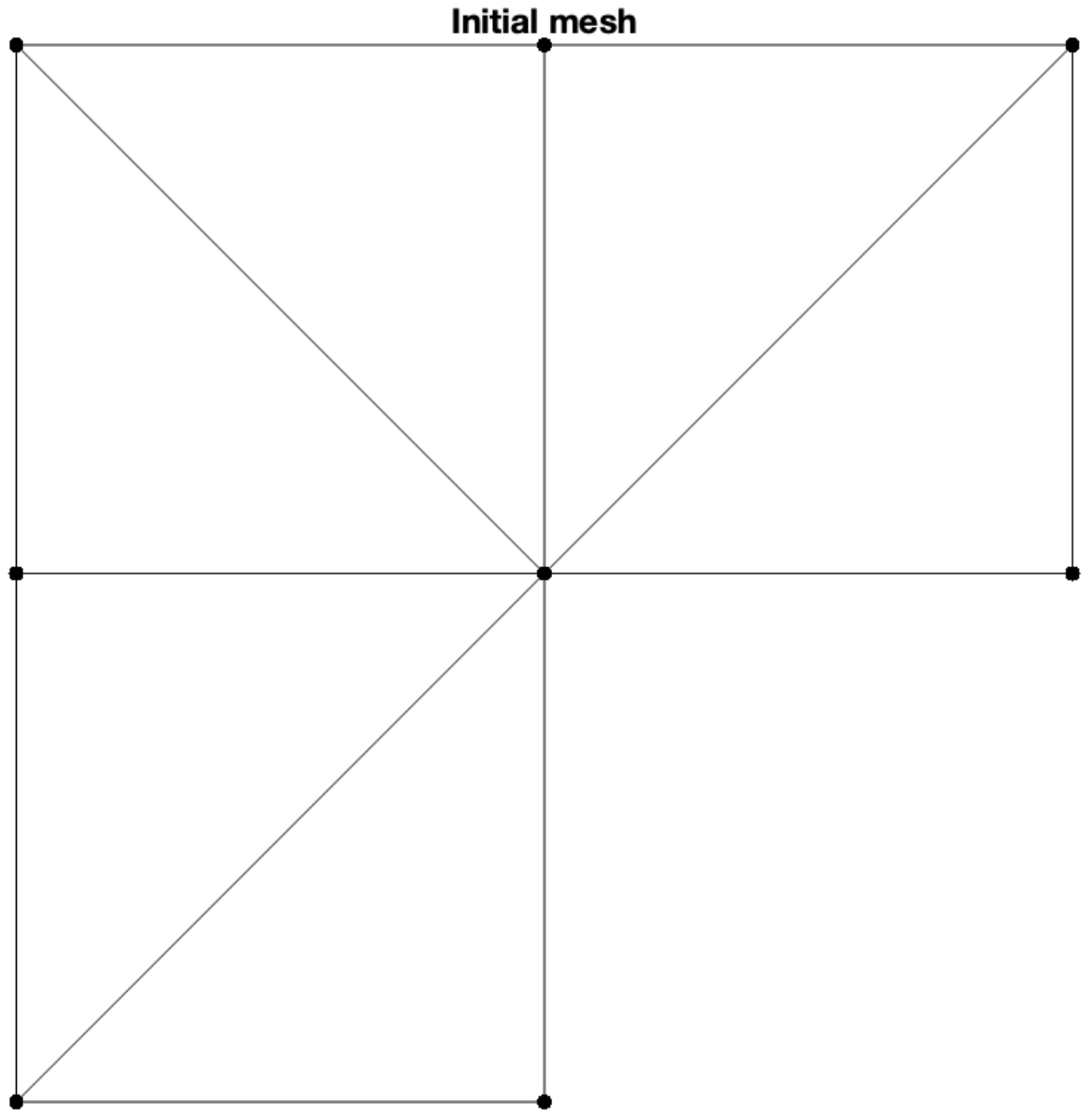}}
\subfigure[]{\includegraphics[width=0.325\textwidth]{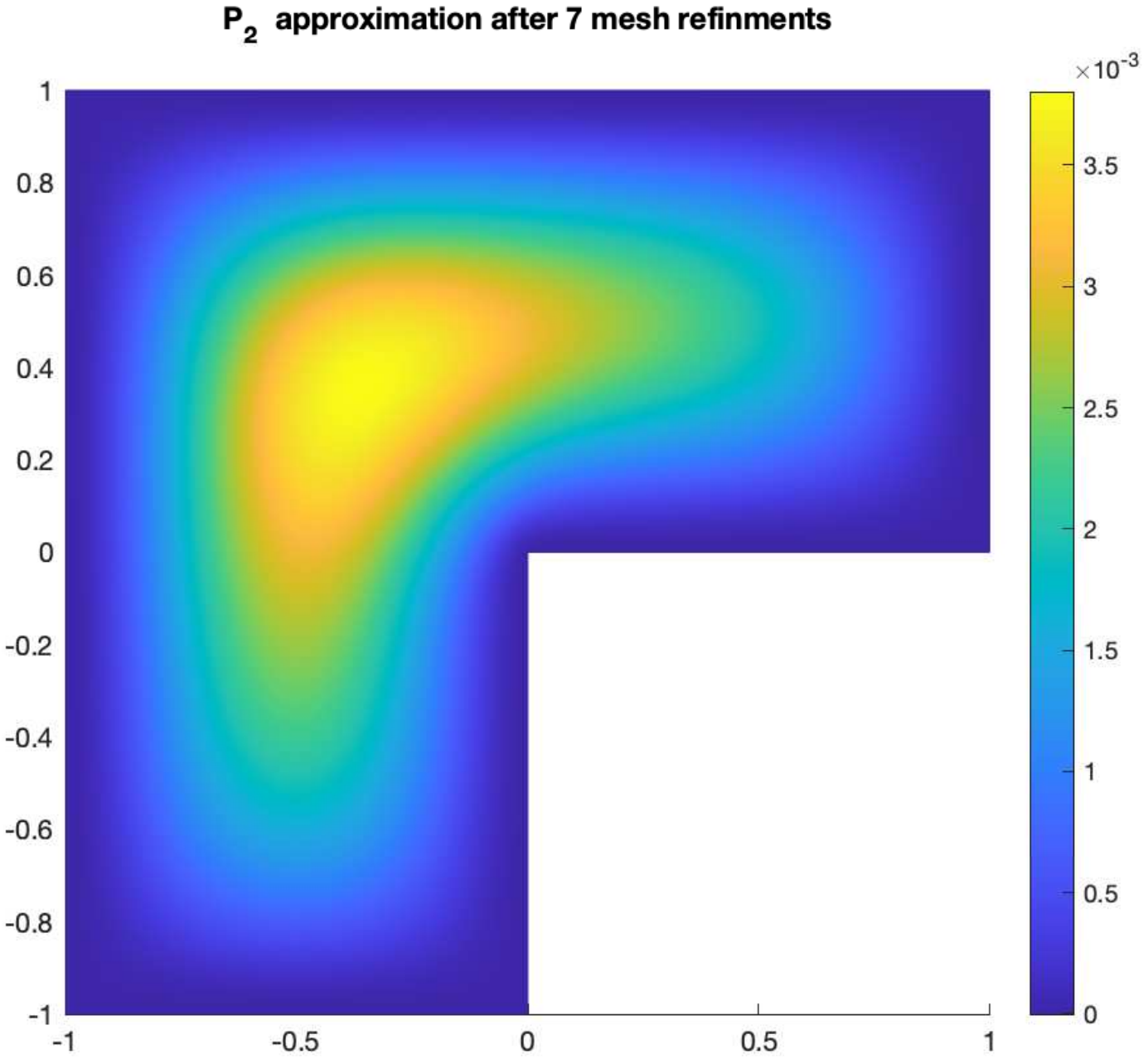}}
\subfigure[]{\includegraphics[width=0.325\textwidth]{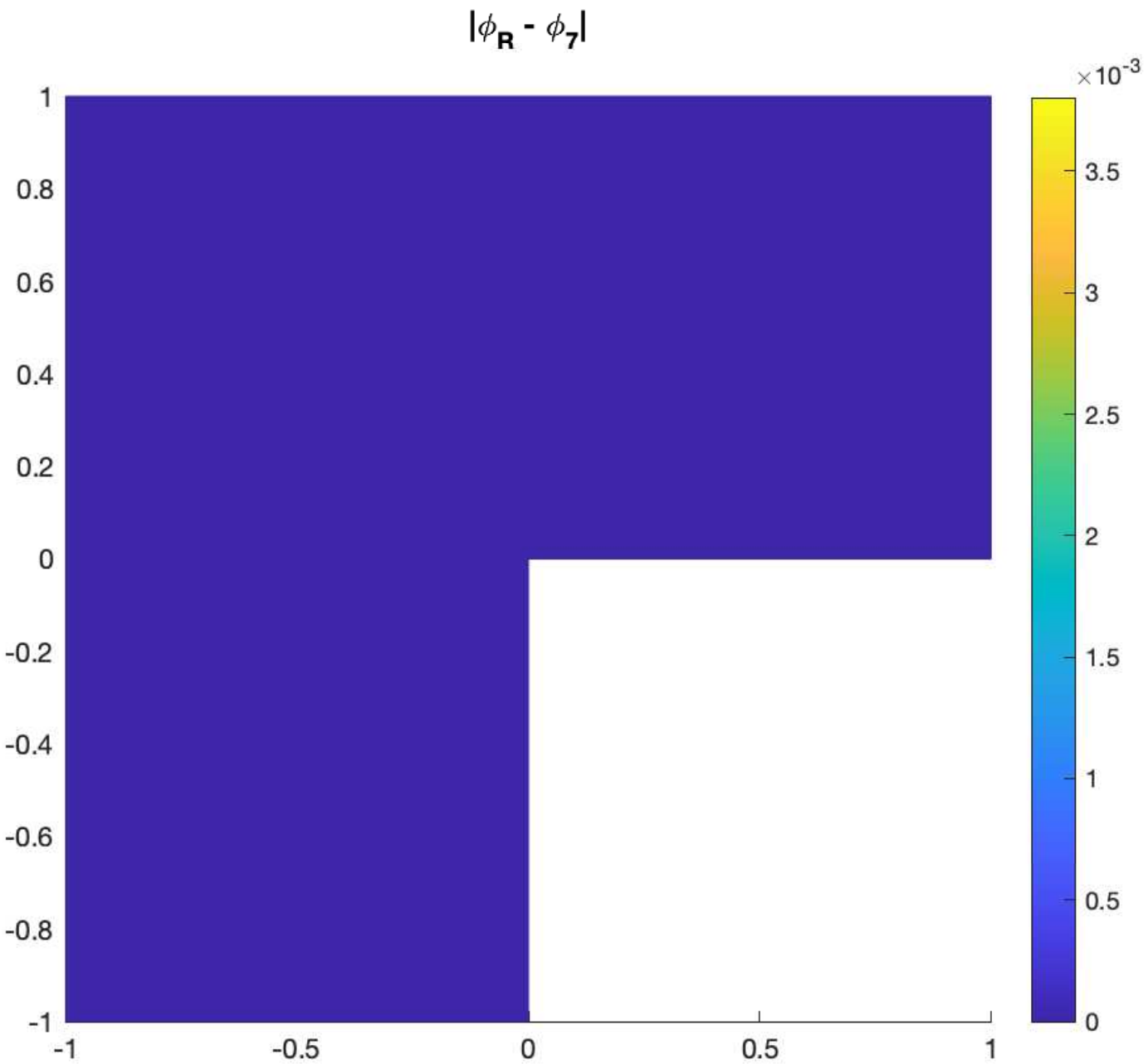}}

\caption{The L-shaped domain (Example \ref{ex5.1}):  (a) the initial mesh; (b) $P_2$ finite element approximation
$\phi_7$; (c) the difference $|\phi_R-\phi_7|$.}\label{Mesh_InitL}
\end{figure}

\end{example}

\begin{example}\label{ex5.3}
We test Example \ref{ex5.1} Test case 2 again using Algorithm \ref{femalg+} and Algorithm \ref{femalg}. For Algorithm \ref{femalg}, we will consider a different source term ${\mathbf F}=(-y, 0)^T$ for the involved Stokes problem, and the new source term also satisfies (\ref{curlFf}), namely, $\text{curl } \mathbf{F}=f=1$.
Here we denote the finite element approximations in Example \ref{ex5.1} Test case 2 (namely, Algorithm \ref{femalg} with $\mathbf{F}=(0,x)^T$) by $\tilde{\phi}_j$ for the biharmonic problem and by $\tilde{\mathbf{u}}_j$ and $\tilde{p}_j$ for the involved Stokes problem.

We compare the finite element approximations with those obtained in Example \ref{ex5.1} Test case 2 on uniform meshes. The errors of the finite element approximations of the biharmonic problem are shown in Table \ref{errratbi}, and these for the Taylor-Hood element approximations of the involved Stokes problem are reported in Table \ref{errratst}. From these results, we find that the errors of finite element approximations ${\phi}_j$ and Taylor-Hood element approximations ${\mathbf{u}}_j$ of both Algorithm \ref{femalg+} and Algorithm \ref{femalg} are converging to the solutions of solutions in Example \ref{ex5.1} Test case 2. Therefore, the finite element approximations $\phi_j$ from both Algorithm \ref{femalg+} and Algorithm \ref{femalg} converge to the exact solution of the biharmonic problem (\ref{eqnbi}).
All these results also indicate that the finite element approximations $\phi_j$ and the Taylor-Hood element approximations $\mathbf{u}_j$ are independent of the choice of the source term $\mathbf F$ as long as it satisfies (\ref{curlFf}), and these approximations are uniquely determined by $f$ in (\ref{eqnbi}). The results are consistent with the theoretical results in Lemma \ref{Stokeindep} and Corollary \ref{coro1}.
We also find that the errors between $\tilde{p}_j$ and $p_j$ are not converging, which implies that $p_j$ depends on the specific $\mathbf{F}$ as indicated in Lemma \ref{Stokeindep}.

\begin{table}[!htbp]\tabcolsep0.04in
\caption{Errors with the reference solution and convergence rates in the L-shaped domain.}
\centering
\begin{tabular}{ |c| c c| c c |}
\hline
 &  \multicolumn{2}{c|}{Algorithm \ref{femalg+}}  & \multicolumn{2}{c|}{Algorithm \ref{femalg}} \\
\hline
 & $\|\phi_j-\tilde \phi_j\|_{H^1(\Omega)}$ & $\|\phi_j-\tilde \phi_j\|$ & $\|\phi_j-\tilde \phi_j\|_{H^1(\Omega)}$ & $\|\phi_j-\tilde \phi_j\|$ \\
\hline
 $j=5$ & 1.73930e-06 & 3.83797e-07 & 7.97013e-09  &  3.72412e-10 \\
 $j=6$ & 3.73791e-07 & 8.38856e-08 & 5.88353e-10 &  2.10894e-11 \\
 $j=7$ & 8.05609e-08 & 1.82423e-08 & 4.56271e-11 &  1.24448e-12\\
 $j=8$ & 1.73829e-08 & 3.95457e-09 & 3.70076e-12 & 7.54393e-14 \\
 $j=9$ & 3.75263e-09 & 8.55733e-10 & 3.10543e-13 &  4.80570e-15\\
 $j=10$ & 8.10299e-10 & 1.84999e-10 & 2.81240e-14 & $--$  \\
\hline
\end{tabular}\label{errratbi}
\end{table}

\begin{table}[!htbp]\tabcolsep0.04in
\caption{Errors with the reference solution and convergence rates in the L-shaped domain.}
\centering
\begin{tabular}{ |c| c c c | c c c|}
\hline
 &  \multicolumn{3}{c|}{Algorithm \ref{femalg+}}  & \multicolumn{3}{c|}{Algorithm \ref{femalg}} \\
\hline
 & $\|\mathbf{u}_j-\tilde{ \mathbf{u}}_j\|_{[H^1(\Omega)]^2}$ & $\|\mathbf{u}_j-\tilde{ \mathbf{u}}_j\|_{[L^2(\Omega)]^2}$ &  $\|p_j-\tilde{p}_j\|$  & $\|\mathbf{u}_j-\tilde{ \mathbf{u}}_j\|_{[H^1(\Omega)]^2}$ & $\|\mathbf{u}_j-\tilde{ \mathbf{u}}_j\|_{[L^2(\Omega)]^2}$ &  $\|p_j-\tilde{p}_j\|$  \\
\hline
$j=5$ & 7.99655e-05 & 1.36408e-06 & 4.19809e-01 & 1.10073e-05 & 7.49412e-08 & 5.59017e-01 \\
$j=6$ & 2.48630e-05 & 2.79879e-07 & 4.19758e-01 & 1.95571e-06 &  6.67031e-09 & 5.59017e-01 \\
 $j=7$ & 7.79196e-06 & 5.87681e-08 & 4.19736e-01 & 3.46595e-07 &  5.91602e-10 & 5.59017e-01 \\
 $j=8$ & 2.44954e-06 & 1.24998e-08 & 4.19727e-01 & 6.13467e-08 & 5.23793e-11 & 5.59017e-01 \\
 $j=9$ & 7.70977e-07 & 2.67752e-09 & 4.19723e-01 & 1.08514e-08 & 4.63360e-12 & 5.59017e-01 \\
 $j=10$ & 2.42772e-07 & 5.75794e-10 & 4.19722e-01 & 1.91888e-09 & 4.09765e-13 & 5.59017e-01 \\
\hline
\end{tabular}\label{errratst}
\end{table}

\end{example}

\begin{example}\label{ex5.2}
We solve the problem in Example \ref{ex5.1} Test Case 2 using both Algorithm \ref{femalg+} and Algorithm \ref{femalg} with polynomials $k=1,2$ on a sequence of graded meshes (including uniform mesh). The initial mesh and the graded mesh after 2 mesh refinements are shown in Figure \ref{Mesh_InitL}a and Figure \ref{Mesh_Init_Graded}a, respectively.

\begin{figure}[h]
\centering
\subfigure[]{\includegraphics[width=0.28\textwidth]{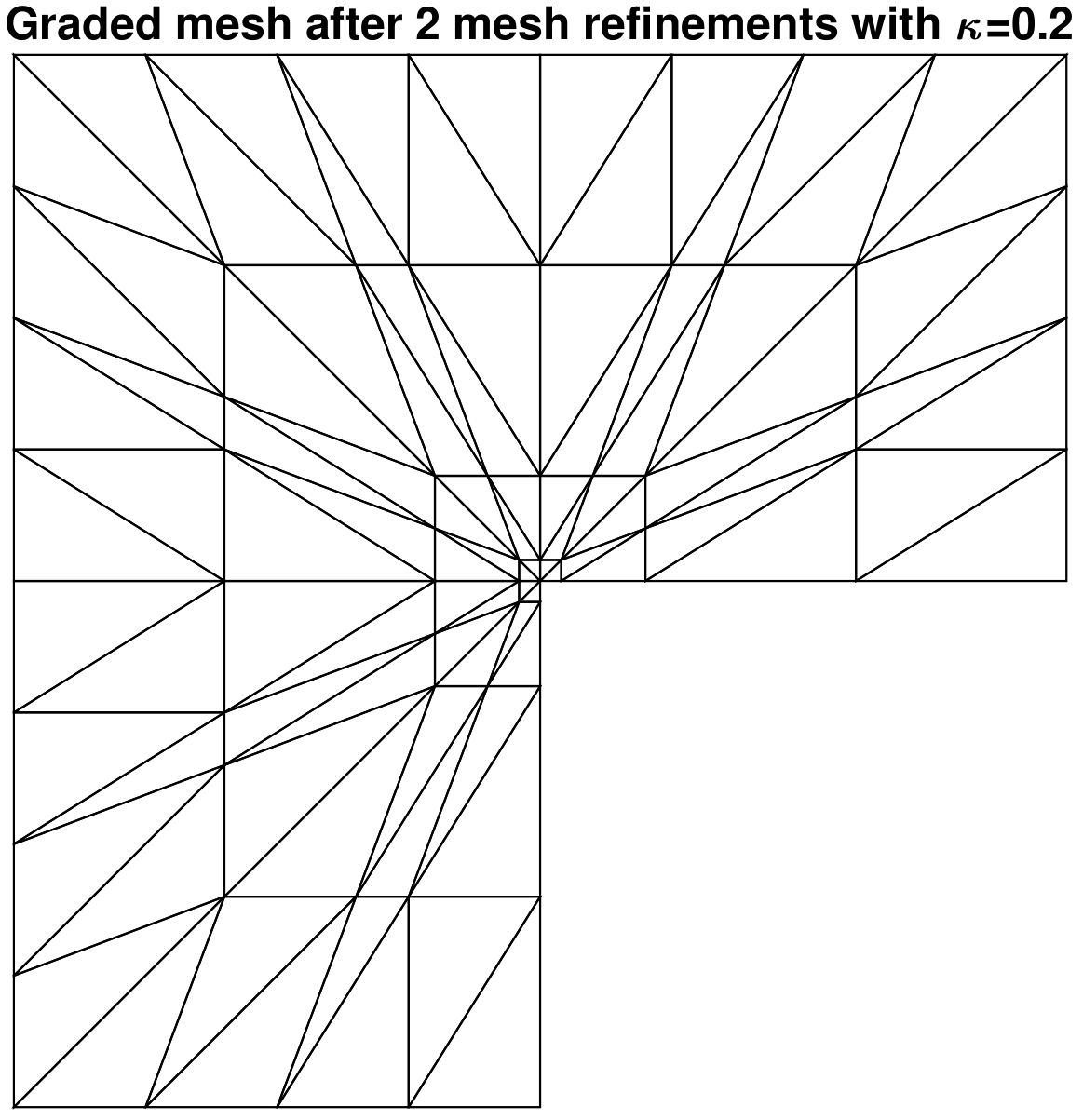}}
\subfigure[]{\includegraphics[width=0.325\textwidth]{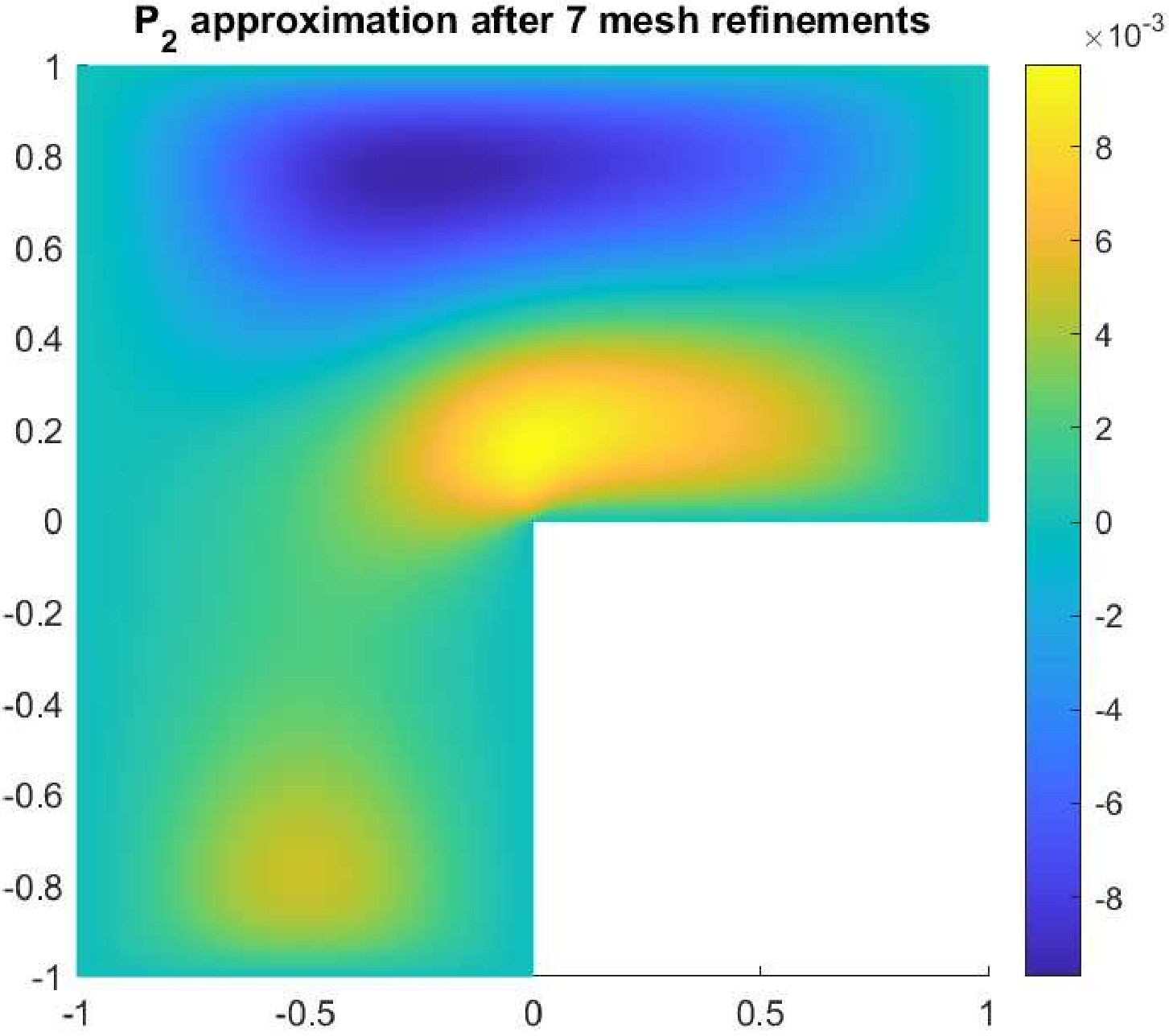}}\\
\subfigure[]{\includegraphics[width=0.325\textwidth]{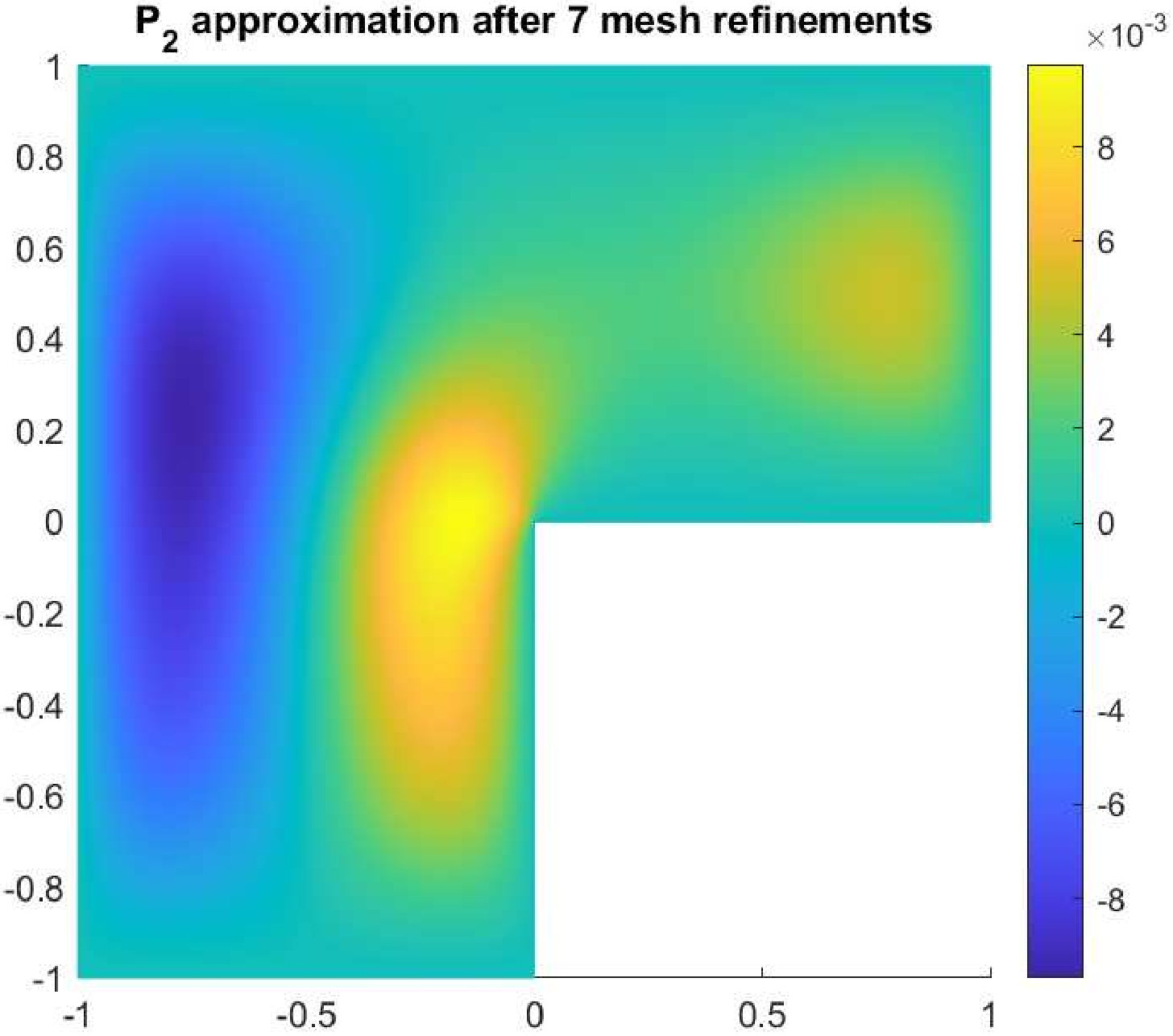}}
\subfigure[]{\includegraphics[width=0.325\textwidth]{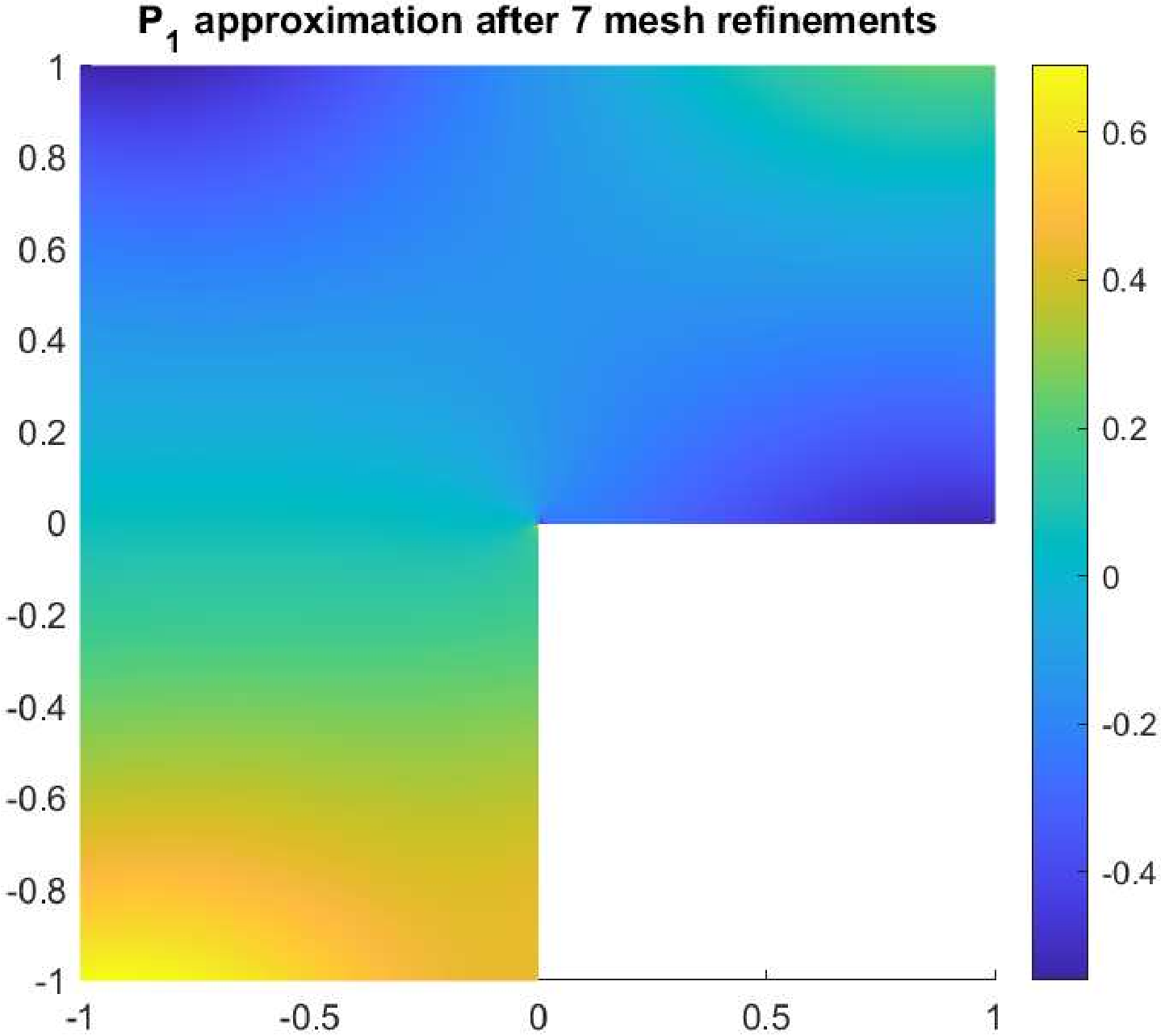}}
\caption{The L-shaped domain (Example \ref{ex5.2}):  (a) the graded mesh after two mesh refinements with $\kappa = 0.2$; (b) Taylor-Hood element approximation $u_1$ of $\mathbf{u}_7$; (e) Taylor-Hood element approximation $u_2$ of $\mathbf{u}_7$; (d) Taylor-Hood element approximation $p_7$.}\label{Mesh_Init_Graded}
\end{figure}

In Table \ref{L_Poi_Rates} (Algorithm \ref{femalg}) and Table \ref{L_Poi_Rates3.1} (Algorithm \ref{femalg+}), we show the numerical convergence rates of finite element approximation $\phi_j$ in both $H^1$ and $L^2$ norm to the solution of the biharmonic problem. We note that the convergence rates from both Algorithm \ref{femalg} and Algorithm \ref{femalg+} are almost the same.
From Table Table \ref{L_Poi_Rates} and Table \ref{L_Poi_Rates3.1}, we find on uniform meshes ($\kappa=0.5$) that the $H^1$ convergence rate of the $P_1$  finite element approximation is optimal with $\mathcal{R} = 1$, and that of the $P_2$  finite element approximations is suboptimal with $\mathcal{R} \approx  1.10$. Both of them are consistent the theoretical result in Theorem \ref{phierrthm3.1} in an L-shaped domain, that is $\mathcal{R}_{\text{exact}} =\min\{k,\alpha_0+1,2\alpha_0\}\approx \min\{k, 1.09\}$ for Algorithm \ref{femalg}, and $\mathcal{R}_{\text{exact}} =\min\{k,\beta_0+2,\alpha_0+1,2\alpha_0\}\approx \min\{k, 1.09\}$ for Algorithm \ref{femalg+}, where $\alpha_0$ is given in Table \ref{alpha0tab} with $\omega=\frac{3\pi}{2}$, and $\beta_0=\frac{\pi}{\omega}$. We also find that
the convergence rates of $P_1$ finite element approximations are optimal with $\mathcal{R}=1$ on graded meshes with $\kappa<0.5$, and that of the $P_2$ finite element approximations are optimal with $\mathcal{R}=2$ on graded meshes with  $\kappa\leq 0.3$, which are close to the theoretical result in Theorem \ref{phierrthmg}, namely, the optimal convergence rate can be obtained when $\kappa < 2^{-\frac{\alpha_0}{\alpha_0}} = 0.5$ for $P_1$ finite element approximations, and $\kappa < 2^{-\frac{1}{\alpha_0}} \approx 0.28$ for $P_2$ finite element approximations.

Again from Table Table \ref{L_Poi_Rates} and Table \ref{L_Poi_Rates3.1}, the $L^2$ convergence rates of both $P_1$ and $P_2$ finite element approximations on uniform meshes are suboptimal with $\mathcal{R} = 1.14$ and $\mathcal{R} \approx 1.09$, which are consistent with the theoretical result $\mathcal{R}_{\text{exact}}\approx \min\{k+1, 1.09\}$ for both Algorithm \ref{femalg}  and Algorithm \ref{femalg+} in Theorem \ref{phierrL2thm}.
On graded meshes, the convergence rates of $P_1$ finite element approximations are optimal with $\kappa\leq 0.2$, and those of the $P_2$ finite element approximations are optimal with  $\kappa\leq 0.1$, which are consistent with the theoretical result in Theorem \ref{phierrL2thm1}, namely, the optimal convergence rate $\mathcal{R}=2$ can be obtained when $\kappa < 2^{-\frac{1}{\alpha_0}} \approx 0.28$ for $P_1$ finite element approximations, and $\mathcal{R}=3$ can be achieved when $\kappa < 2^{-\frac{1.5}{\alpha_0}} \approx 0.15$ for $P_2$ finite element approximations.

\begin{table}[!htbp]\tabcolsep0.04in
\caption{Convergence history of finite element approximation $\phi_j$ of the biharmonic problem from Algorithm \ref{femalg} in the L-shaped domain.}
\centering
\begin{tabular}{ |c| c| c c c c c c| c c c c c c|} \hline
 &  & \multicolumn{6}{c|}{$H_1$ rate of $\phi_j$} & \multicolumn{6}{c|}{ $L_2$ rate of $\phi_j$} \\
\hline
& $\kappa$ & 0.05 & 0.1 & 0.2 & 0.3 & 0.4 &0.5 & 0.05 & 0.1 & 0.2 & 0.3& 0.4&0.5 \\
\hline
\multirow{5}{*}{$k=1$ }
&$j=6$ & $--$ & 1.00 & 1.00 & 1.00 & 1.00 & 1.00  & $--$ & 2.00 & 1.99 & 1.96 & 1.84 & 1.52  \\
&$j=7$ & $--$& 1.00 & 1.00 & 1.00 & 1.00 & 1.00   & $--$ & 2.00 & 2.00 & 1.96 & 1.79 & 1.37  \\
&$j=8$ & $--$ & 1.00 & 1.00 & 1.00 & 1.00 & 1.00  & $--$ & 2.00 & 2.00 & 1.96 & 1.72 & 1.26  \\
&$j=9$ & $--$ & 1.00 & 1.00 & 1.00 & 1.00 & 1.00   & $--$ & 2.00 & 2.00 & 1.95 & 1.66 & 1.18  \\
&$j=10$ & $--$ & 1.00 & 1.00 & 1.00 & 1.00 & 1.00   & $--$ & 2.00 & 2.00 & 1.95 & 1.60 & 1.14  \\
\hline
\multirow{4}{*}{$k=2$}
&$j=6$ & 1.98 & 1.99 & 1.99 & 1.99 & 1.89 & 1.37   & 3.05 & 3.02 & 2.99 & 1.94 & 1.43 & 1.08  \\
&$j=7$ & 1.99 & 2.00 & 2.00 & 2.00 & 1.81 & 1.23   & 3.03 & 3.01 & 2.98 & 1.90 & 1.43 & 1.08  \\
&$j=8$ & 2.00 & 2.00 & 2.00 & 2.00 & 1.71 & 1.15   & 3.02 & 3.00 & 2.96 & 1.89 & 1.44 & 1.08  \\
&$j=9$ & 2.00 & 2.00 & 2.00 & 2.00 & 1.61 & 1.12   & 3.01 & 3.00 & 2.92 & 1.89 & 1.44 & 1.08  \\
\hline
\end{tabular}\label{L_Poi_Rates}
\end{table}

\begin{table}[!htbp]\tabcolsep0.04in
\caption{Convergence history of finite element approximation $\phi_j$ of the biharmonic problem from Algorithm \ref{femalg+} in the L-shaped domain.}
\centering
\begin{tabular}{ |c| c| c c c c c c| c c c c c c|} \hline
 &  & \multicolumn{6}{c|}{$H_1$ rate of $\phi_j$} & \multicolumn{6}{c|}{ $L_2$ rate of $\phi_j$} \\
\hline
& $\kappa$ & 0.05 & 0.1 & 0.2 & 0.3 & 0.4 &0.5 & 0.05 & 0.1 & 0.2 & 0.3& 0.4&0.5 \\
\hline
\multirow{5}{*}{$k=1$ }
&$j=6$ & $--$ & 1.00 & 1.00 & 1.00 & 1.00 & 1.00  & $--$ & 1.99 & 1.99 & 1.96 & 1.85 & 1.54  \\
&$j=7$ & $--$& 1.00 & 1.00 & 1.00 & 1.00 & 1.00   & $--$ & 2.00 & 2.00 & 1.96 & 1.80 & 1.38  \\
&$j=8$ & $--$ & 1.00 & 1.00 & 1.00 & 1.00 & 1.00  & $--$ & 2.00 & 2.00 & 1.96 & 1.73 & 1.26  \\
&$j=9$ & $--$ & 1.00 & 1.00 & 1.00 & 1.00 & 1.00   & $--$ & 2.00 & 2.00 & 1.96 & 1.67 & 1.19  \\
&$j=10$ & $--$ & 1.00 & 1.00 & 1.00 & 1.00 & 1.00   & $--$ & 2.00 & 2.00 & 1.95 & 1.61 & 1.14  \\
\hline
\multirow{4}{*}{$k=2$}
&$j=6$ & 1.98 & 1.99 & 1.99 & 1.99 & 1.89 & 1.35   & 3.04 & 3.02 & 2.99 & 1.94 & 1.41 & 1.04  \\
&$j=7$ & 1.99 & 2.00 & 2.00 & 2.00 & 1.81 & 1.22   & 3.03 & 3.01 & 2.98 & 1.90 & 1.43 & 1.06  \\
&$j=8$ & 2.00 & 2.00 & 2.00 & 2.00 & 1.71 & 1.15   & 3.01 & 3.00 & 2.96 & 1.89 & 1.43 & 1.07  \\
&$j=9$ & 2.00 & 2.00 & 2.00 & 2.00 & 1.61 & 1.12   & 3.01 & 3.00 & 2.92 & 1.89 & 1.44 & 1.08  \\
\hline
\end{tabular}\label{L_Poi_Rates3.1}
\end{table}

The Taylor-Hood element approximations $\mathbf{u}_7$ and $p_7$ based on Algorithm \ref{femalg} on uniform meshes for the involved Stokes problem are shown in Figure \ref{Mesh_Init_Graded}b-\ref{Mesh_Init_Graded}d.
In Table \ref{L_Stok_Rates2}-\ref{L_Stok_Rates3.1}, we display numerical convergence rates of the Mini element approximations and the Taylor-Hood approximations from both Algorithm \ref{femalg} and Algorithm \ref{femalg+} for the involved Stokes problem. The $H^1$ convergence rates of $\mathbf{u}_j$ and the $L^2$ convergence rates of $p_j$ with $k=1,2$ are suboptimal on uniform meshes with convergence rates $\mathcal{R} \approx 0.54$, which are consistent with the theoretical result $\mathcal{R}_{\text{exact}} = \alpha_0 \approx 0.54$ in Lemma \ref{stokesestlem} and Lemma \ref{stokesestlem+} in an L-shaped domain. On graded meshes, the convergence rates are optimal 
with $\mathcal{R}=1$ when $\kappa \leq 0.2$ for $k=1$, and $\mathcal{R}=2$ when $\kappa \leq 0.05$ for $k=2$, these are consistent with the results in Theorem \ref{Stokesgraderr} and Remark \ref{stokesoptimal}, namely, the optimal convergence rate can be achieved when $\kappa<2^{-\frac{1}{\alpha_0}} \approx 0.28$ for $k=1$, and $\kappa<2^{-\frac{2}{\alpha_0}} \approx0.08$ for $k=2$.

The $L^2$ convergence rates of $\mathbf{u}_j$ are suboptimal on uniform meshes with convergence rates $\mathcal{R} \approx 1.13$ and $\mathcal{R} \approx 1.12$, which are consistent with the theoretical result $\mathcal{R}=2\alpha_0 \approx 1.09$ in Lemma \ref{stokesestlem} and Lemma \ref{stokesestlem+} in an L-shaped domain. On graded meshes, the convergence rates are optimal with $\mathcal{R}=2$ for $\kappa \leq 0.2$, and with $\mathcal{R}=3$ for $\kappa \leq 0.1$, which are close to the theoretical results in Theorem \ref{Stokesgraderr} and Remark \ref{stokesoptimal}, namely, the optimal convergence rate can be achieved when $\kappa<2^{-\frac{1}{\alpha_0}} \approx 0.28$ for $k=1$, and $\kappa<2^{-\frac{2}{\alpha_0}} \approx0.08$ for $k=2$.

\begin{table}[!htbp]\tabcolsep0.04in
\caption{Convergence history of the Mini element approximations ($k=1$) of Stokes problem from Algorithm \ref{femalg} in the L-shaped domain.}
\centering
\begin{tabular}{ |c| c c c c c| c c c c c | c c c c c |} \hline
 &  \multicolumn{5}{c|}{$H_1$ rate of $\mathbf{u}_j$} &  \multicolumn{5}{c|}{$L^2$ rate of $\mathbf{u}_j$}   &  \multicolumn{5}{c|}{ $L_2$ rate of $p_j$}  \\
\hline
 $\kappa$ & 0.1 & 0.2 & 0.3 & 0.4 &0.5 & 0.1 & 0.2 & 0.3& 0.4&0.5& 0.1 & 0.2 & 0.3& 0.4&0.5 \\
\hline
$j=6$ & 1.00 & 1.00 & 0.98 & 0.88 & 0.69   & 2.01 & 2.01 & 1.97 & 1.80 & 1.40  & 1.23 & 1.35 & 1.15 & 0.77 & 0.57  \\
$j=7$ & 1.00 & 1.00 & 0.98 & 0.85 & 0.64   & 2.01 & 2.01 & 1.96 & 1.71 & 1.28  & 1.10 & 1.35 & 1.07 & 0.74 & 0.56  \\
$j=8$ & 1.00 & 1.00 & 0.97 & 0.82 & 0.60   & 2.01 & 2.00 & 1.95 & 1.62 & 1.20  & 1.05 & 1.35 & 1.01 & 0.73 & 0.55  \\
$j=9$ & 1.00 & 1.00 & 0.97 & 0.80 & 0.58   & 2.01 & 2.00 & 1.94 & 1.55 & 1.16  & 1.15 & 1.34 & 0.98 & 0.72 & 0.55  \\
$j=10$ & 1.00 & 1.00 & 0.97 & 0.78 & 0.56 & 2.00 & 2.00 & 1.94 & 1.51 & 1.13  & 1.29 & 1.33 & 0.96 & 0.72 & 0.55  \\
\hline
\end{tabular}\label{L_Stok_Rates2}
\end{table}

\begin{table}[!htbp]\tabcolsep0.04in
\caption{Convergence history of the Taylor-Hood element approximations ($k=2$) of Stokes problem from Algorithm \ref{femalg} in the L-shaped domain.}
\centering
\begin{tabular}{ |c |c| c c c c c c| c c c c c c |} \hline
 &  &  \multicolumn{6}{c|}{$H_1$ rate}   &  \multicolumn{6}{c|}{ $L_2$ rate}  \\
\hline
 & $\kappa$ & 0.05 & 0.1 & 0.2 & 0.3 & 0.4 &0.5 & 0.05 & 0.1 & 0.2 & 0.3& 0.4&0.5 \\
\hline
\multirow{5}{*}{$\mathbf{u}_j$ }
& $j=6$ & 1.82 & 1.83 & 1.37 & 0.96 & 0.72 & 0.55 & 2.97 & 3.00 & 2.99 & 2.04 & 1.57 & 1.32  \\
& $j=7$& 1.90 & 1.84 & 1.31 & 0.95 & 0.72 & 0.54 & 3.03 & 3.01 & 2.99 & 1.93 & 1.50 & 1.24  \\
& $j=8$ & 1.95 & 1.83 & 1.28 & 0.95 & 0.72 & 0.54 & 3.03 & 3.01 & 2.98 & 1.90 & 1.47 & 1.19  \\
& $j=9$ & 1.97 & 1.83 & 1.27 & 0.95 & 0.72 & 0.54 & 3.01 & 3.00 & 2.95 & 1.89 & 1.45 & 1.15  \\
& $j=10$ & 1.98 & 1.82 & 1.27 & 0.95 & 0.72 & 0.54 & 3.01 & 3.00 & 2.91 & 1.89 & 1.45 & 1.12  \\
\hline
\multirow{5}{*}{$p_j$}
& $j=6$&\multicolumn{6}{c|}{}  & 1.81 & 1.82 & 1.31 & 0.95 & 0.72 & 0.55  \\
& $j=7$&\multicolumn{6}{c|}{}  & 1.91 & 1.82 & 1.28 & 0.95 & 0.72 & 0.55  \\
& $j=8$&\multicolumn{6}{c|}{}  & 1.97 & 1.82 & 1.27 & 0.95 & 0.72 & 0.54  \\
& $j=9$&\multicolumn{6}{c|}{}  & 1.99 & 1.82 & 1.27 & 0.95 & 0.72 & 0.54  \\
& $j=10$&\multicolumn{6}{c|}{}  & 2.00 & 1.81 & 1.27 & 0.95 & 0.72 & 0.54  \\
\hline
\end{tabular}\label{L_Stok_Rates}
\end{table}

\begin{table}[!htbp]\tabcolsep0.04in
\caption{Convergence history of the Mini element approximations ($k=1$) of Stokes problem from Algorithm \ref{femalg+} in the L-shaped domain.}
\centering
\begin{tabular}{ |c| c c c c c| c c c c c | c c c c c |} \hline
 &  \multicolumn{5}{c|}{$H_1$ rate of $\mathbf{u}_j$} &  \multicolumn{5}{c|}{$L^2$ rate of $\mathbf{u}_j$}   &  \multicolumn{5}{c|}{ $L_2$ rate of $p_j$}  \\
\hline
 $\kappa$ & 0.1 & 0.2 & 0.3 & 0.4 &0.5 & 0.1 & 0.2 & 0.3& 0.4&0.5& 0.1 & 0.2 & 0.3& 0.4&0.5 \\
\hline
$j=6$ & 1.00 & 1.00 & 0.98 & 0.88 & 0.69   & 2.01 & 2.00 & 1.97 & 1.81 & 1.41  & 1.22 & 1.35 & 1.15 & 0.76 & 0.54  \\
$j=7$ & 1.00 & 1.00 & 0.98 & 0.85 & 0.64   & 2.01 & 2.00 & 1.97 & 1.73 & 1.29  & 1.10 & 1.35 & 1.07 & 0.74 & 0.54  \\
$j=8$ & 1.00 & 1.00 & 0.97 & 0.82 & 0.60   & 2.01 & 2.00 & 1.96 & 1.64 & 1.21   & 1.04 & 1.35 & 1.01 & 0.73 & 0.54  \\
$j=9$ & 1.00 & 1.00 & 0.97 & 0.80 & 0.58   & 2.00 & 2.00 & 1.95 & 1.57 & 1.16  & 1.15 & 1.34 & 0.98 & 0.72 & 0.54  \\
$j=10$ & 1.00 & 1.00 & 0.97 & 0.78 & 0.56  & 2.00 & 2.00 & 1.94 & 1.52 & 1.13  & 1.29 & 1.35 & 0.96 & 0.72 & 0.54  \\
\hline
\end{tabular}\label{L_Stok_Rates23.1}
\end{table}

\begin{table}[!htbp]\tabcolsep0.04in
\caption{Convergence history of the Taylor-Hood element approximations ($k=2$) of Stokes problem from Algorithm \ref{femalg+} in the L-shaped domain.}
\centering
\begin{tabular}{ |c |c| c c c c c c| c c c c c c |} \hline
 &  &  \multicolumn{6}{c|}{$H_1$ rate}   &  \multicolumn{6}{c|}{ $L_2$ rate}  \\
\hline
 & $\kappa$ & 0.05 & 0.1 & 0.2 & 0.3 & 0.4 &0.5 & 0.05 & 0.1 & 0.2 & 0.3& 0.4&0.5 \\
\hline
\multirow{5}{*}{$\mathbf{u}_j$ }
& $j=6$ & 1.82 & 1.83 & 1.37 & 0.96 & 0.72 & 0.53 & 2.97 & 2.99 & 2.98 & 2.03 & 1.56 & 1.29  \\
& $j=7$ & 1.90 & 1.84 & 1.31 & 0.95 & 0.72 & 0.54 & 3.03 & 3.01 & 2.99 & 1.93 & 1.50 & 1.23  \\
& $j=8$ & 1.95 & 1.83 & 1.28 & 0.95 & 0.72 & 0.54 & 3.03 & 3.01 & 2.97 & 1.90 & 1.47 & 1.18  \\
& $j=9$ & 1.97 & 1.83 & 1.27 & 0.95 & 0.72 & 0.54 & 3.01 & 3.00 & 2.95 & 1.89 & 1.45 & 1.14  \\
& $j=10$ & 1.98 & 1.82 & 1.27 & 0.95 & 0.72 & 0.54 & 3.00 & 2.98 & 2.93 & 1.89 & 1.45 & 1.12  \\
\hline
\multirow{5}{*}{$p_j$}
& $j=6$&\multicolumn{6}{c|}{}  & 1.75 & 1.71 & 1.27 & 0.94 & 0.71 & 0.53  \\
& $j=7$&\multicolumn{6}{c|}{}  & 1.89 & 1.75 & 1.27 & 0.95 & 0.72 & 0.53  \\
& $j=8$&\multicolumn{6}{c|}{}  & 1.96 & 1.77 & 1.27 & 0.95 & 0.72 & 0.54  \\
& $j=9$&\multicolumn{6}{c|}{}  & 1.99 & 1.78 & 1.26 & 0.95 & 0.72 & 0.54  \\
& $j=10$&\multicolumn{6}{c|}{}  & 2.00 & 1.79 & 1.26 & 0.95 & 0.72 & 0.54  \\
\hline
\end{tabular}\label{L_Stok_Rates3.1}
\end{table}

\end{example}

\begin{example}\label{ex5.4}
We consider the biharmonic problem (\ref{eqnbi}) with $f=1$ in the polygonal domain $\Omega$ (see Figure \ref{Mesh_Init_Gradedex2}a) enclosed by line segments $Q_1Q_2$, $Q_2Q_3$, $Q_3Q_4$, and $Q_4Q_1$ with $Q_1(0, 0)$, $Q_2( \frac{2}{c_1/c_2+1}, -\frac{2c_1}{c_1/c_2+1})$, $Q_3(2, 0)$, and  $Q_4(\frac{2}{c_1/c_2+1}, \frac{2c_1}{c_1/c_2+1})$, here $c_1=\tan\left( \frac{11\pi}{24} \right)$, and $c_2=\tan\left(\frac{13\pi}{72}\right)$. The largest interior angle $\omega=\angle Q_2Q_1Q_4 = \frac{11\pi}{12}$. 
We solve this problem by Algorithm \ref{femalg+} and Algorithm \ref{femalg} on both graded meshes and uniform meshes based on polynomials with $k=2$.
We take $\mathbf{F}=(0,x)^T$ as the source term of involved Stokes problem (\ref{stokes}) in Algorithm \ref{femalg}. 
The finite element approximations $w_7$ and $\phi_7$ from Algorithm \ref{femalg+} are shown in Figure \ref{Mesh_Init_Gradedex2}b and Figure \ref{Mesh_Init_Gradedex2}c, respectively. We show both $H^1$ and $L^2$ convergence rates of the finite element approximation $\phi_j$ from both Algorithm \ref{femalg+} and Algorithm \ref{femalg} in Table \ref{L_Poi_Rates25.4}. We find the the convergence rates are almost the same for these two algorithms.

For $H^1$ convergence rate of $\phi_j$ from both algorithms, we find that  the convergence rates on uniform meshes are $\mathcal{R}=2.00$, which is consistent with the expected convergence rate $\mathcal{R}_{\text{exact}}=\min\{k,\alpha_0+1, 2\alpha_0\}=2$ in Theorem \ref{phierrthm3.1}, where $\alpha_0\approx 1.20$ corresponding to the interior angle $\omega=\frac{11\pi}{12}$ shown in Table \ref{alpha0tab}. On graded meshes, the optimal convergence rate $\mathcal{R}=2.0$ can be observed on graded meshes for $\kappa < 0.5$, which is consistent with the theoretical result $\kappa<2^{-\frac{\alpha_0}{\alpha_0}} = 0.5$ in Theorem \ref{phierrthmg}. For $L^2$ convergence rate, we find that on uniform meshes the convergence rates $\mathcal{R} \approx 2.98$, which is a little bit larger than the expected rate $\mathcal{R}_{\text{exact}} \approx 2\alpha_0 = 2.40$. On graded meshes, we find that the optimal convergence rates $\mathcal{R}=3$ can be obtained when $\kappa\leq 0.4$, which is consistent with the theoretical result $\kappa<2^{-\frac{1.5}{\alpha_0}} \approx 0.42$ in Theorem \ref{phierrL2thm1}.

\begin{figure}[h]
\centering
\subfigure[]{\includegraphics[width=0.24\textwidth]{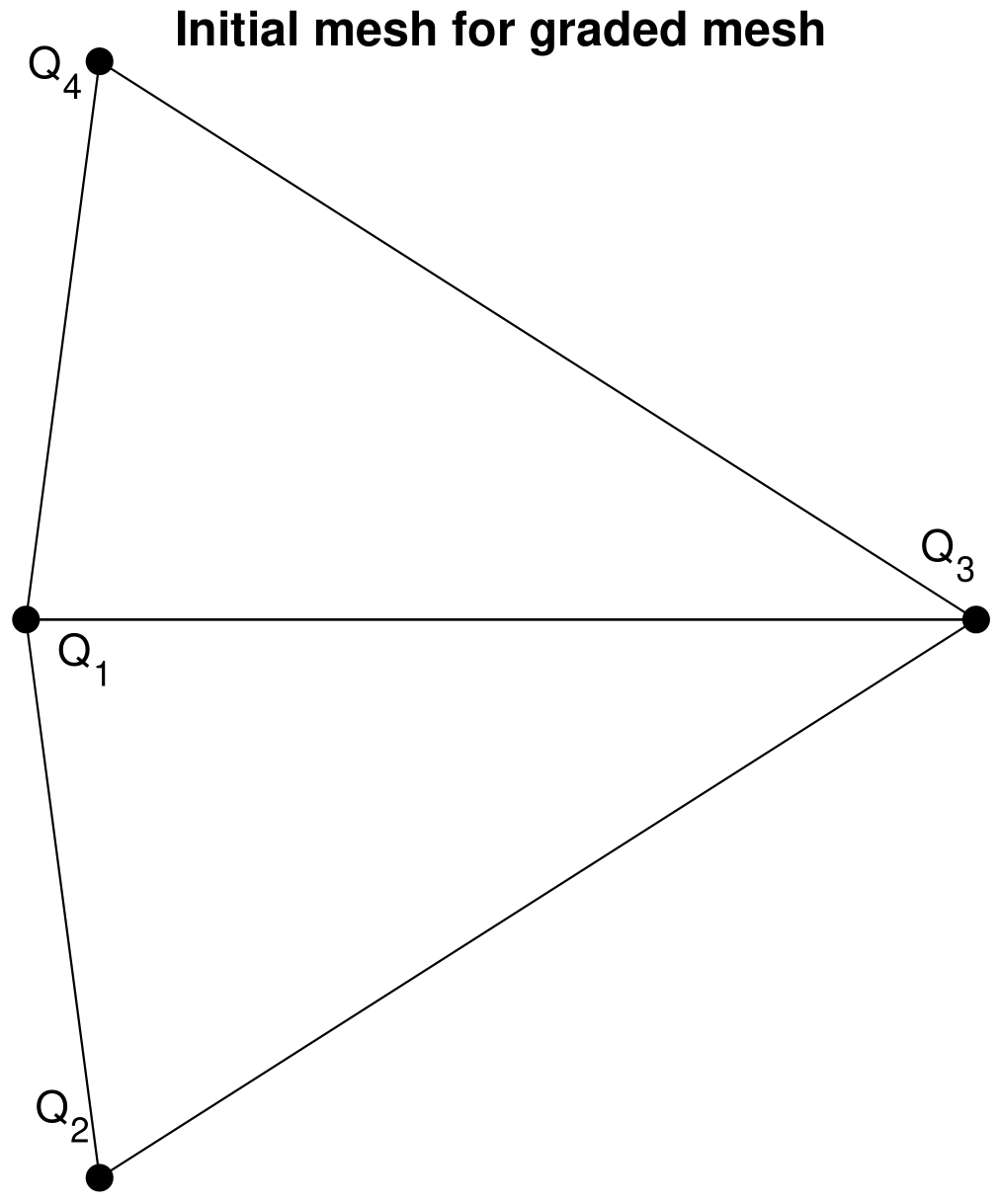}}
\subfigure[]{\includegraphics[width=0.29\textwidth]{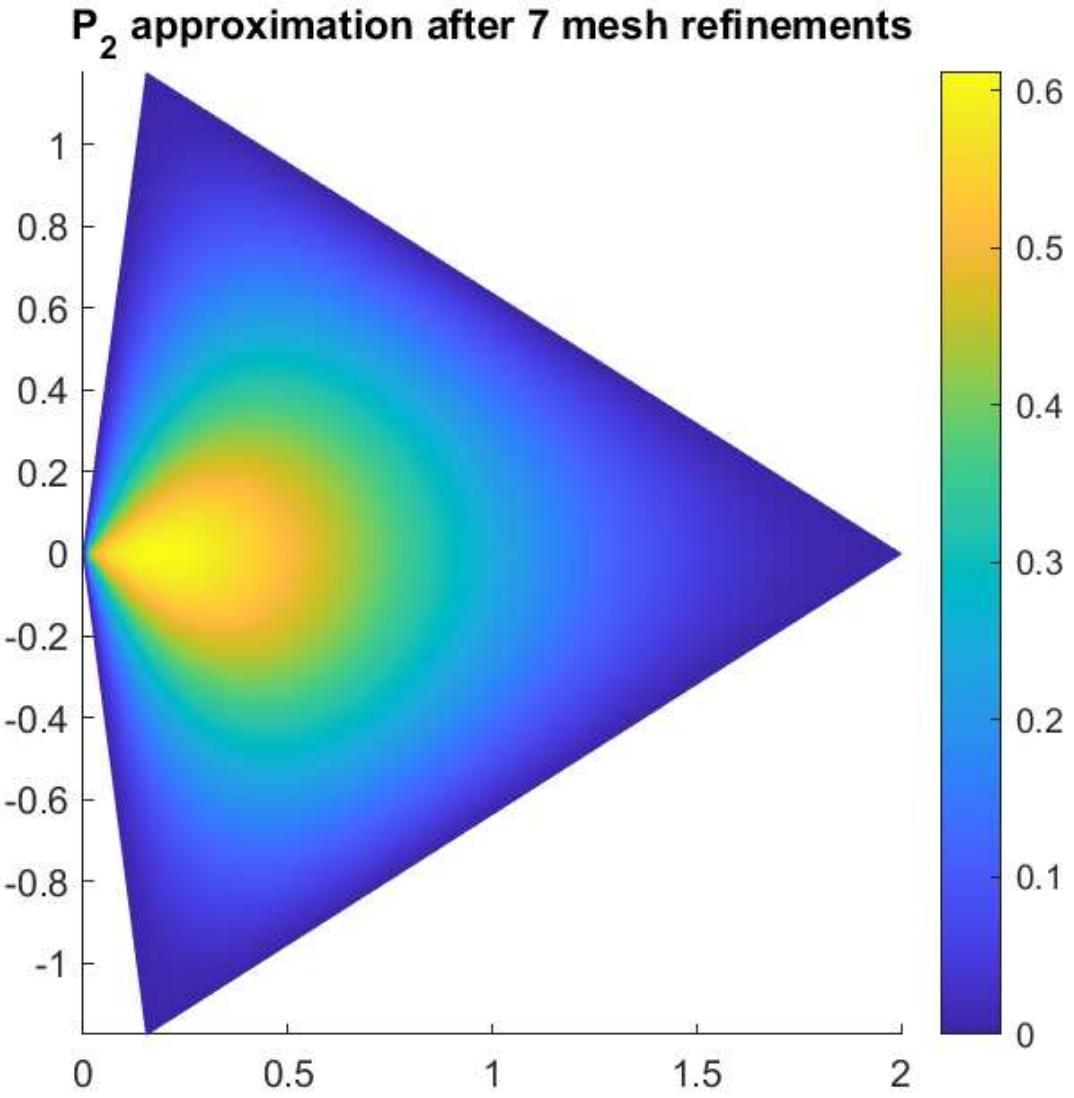}}
\subfigure[]{\includegraphics[width=0.30\textwidth]{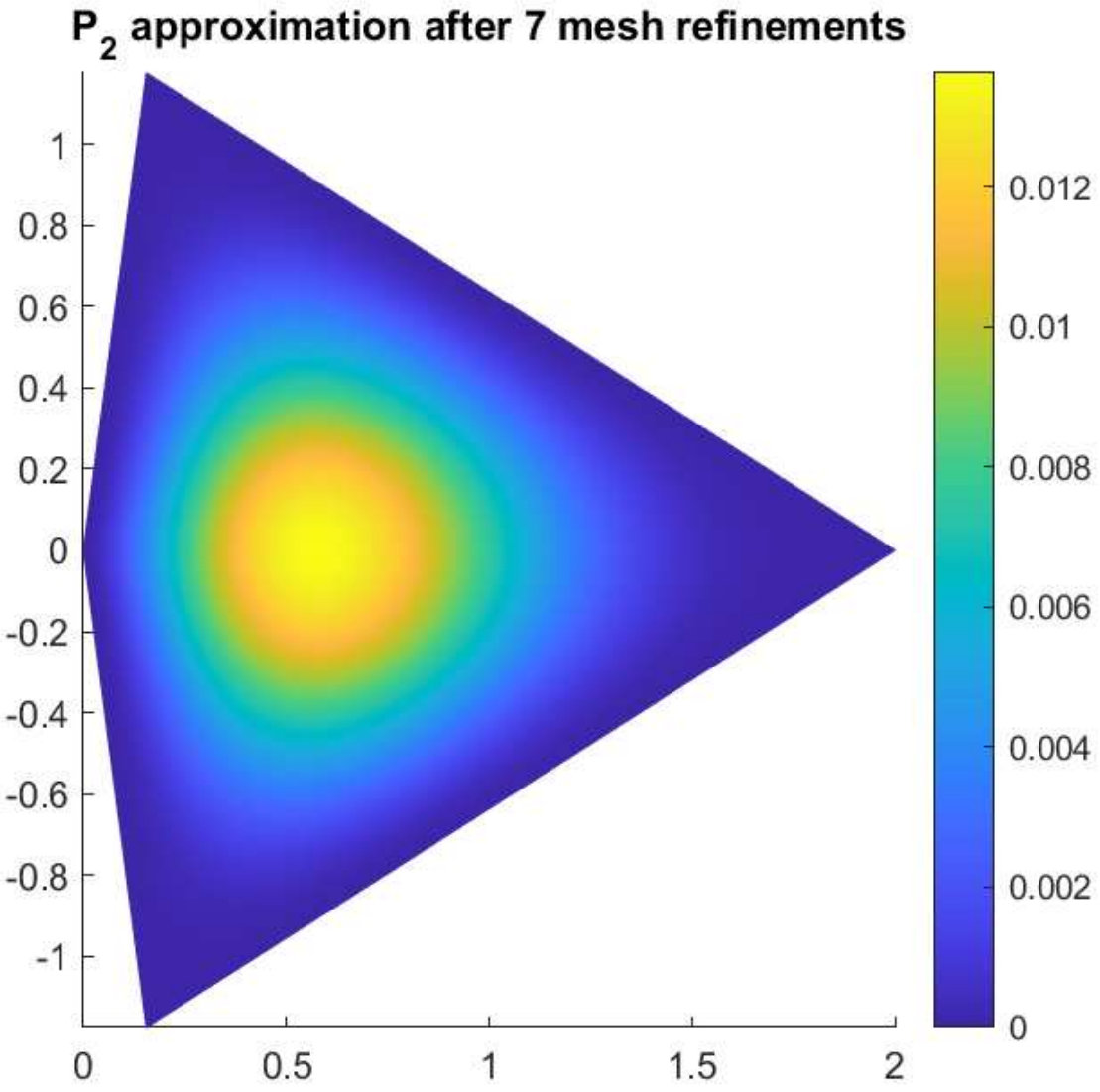}}
\subfigure[]{\includegraphics[width=0.30\textwidth]{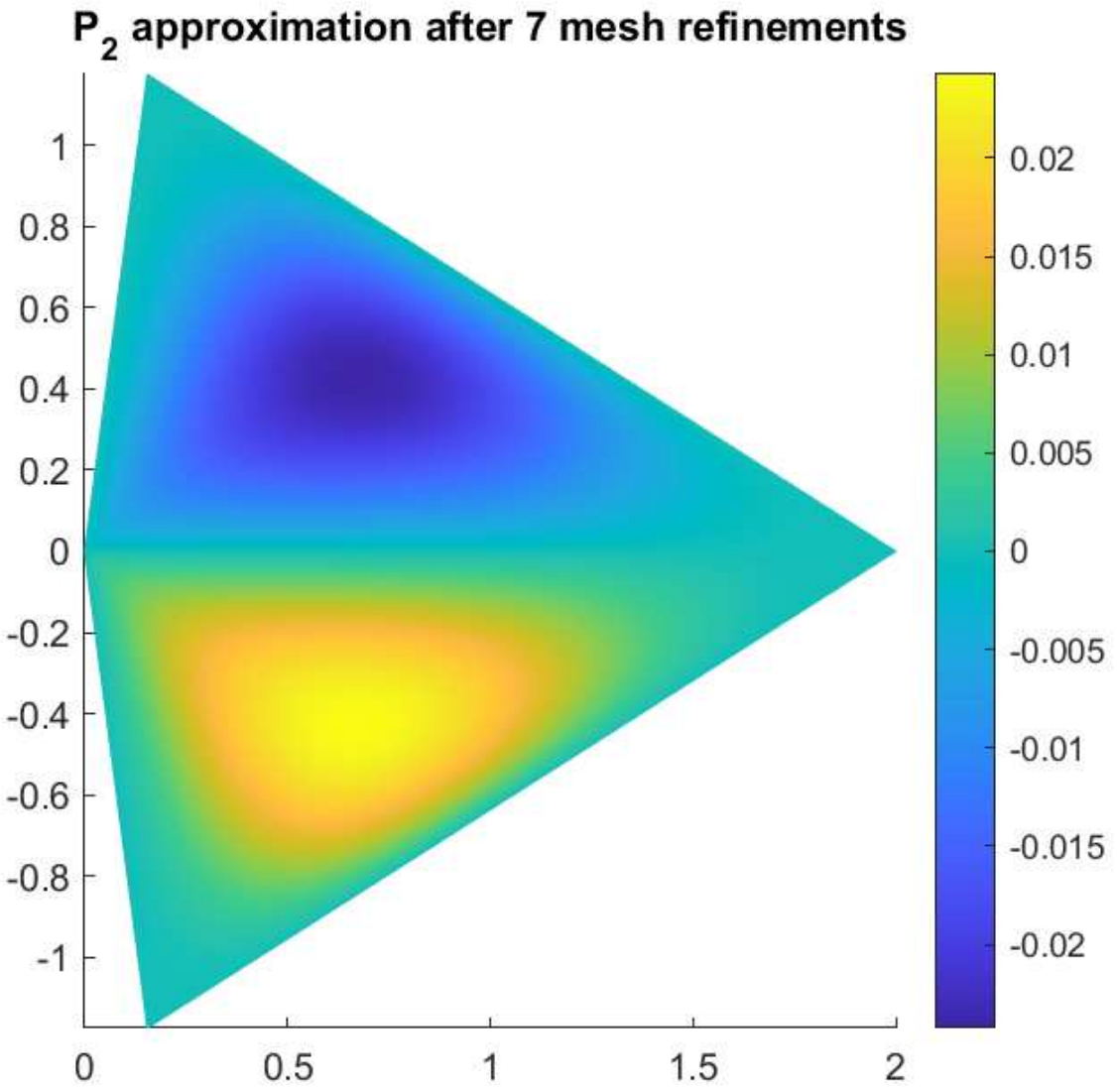}}
\subfigure[]{\includegraphics[width=0.295\textwidth]{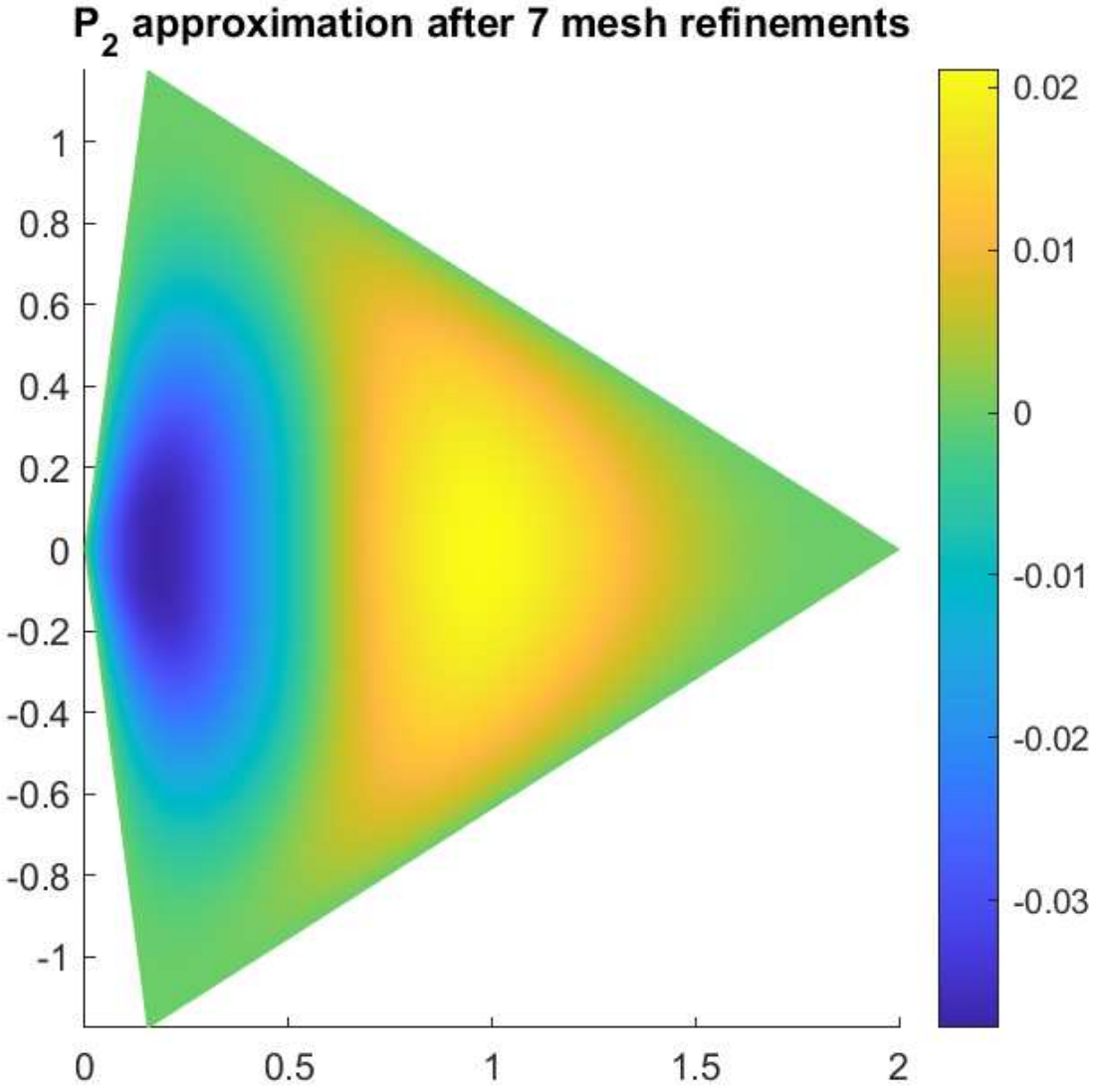}}
\subfigure[]{\includegraphics[width=0.295\textwidth]{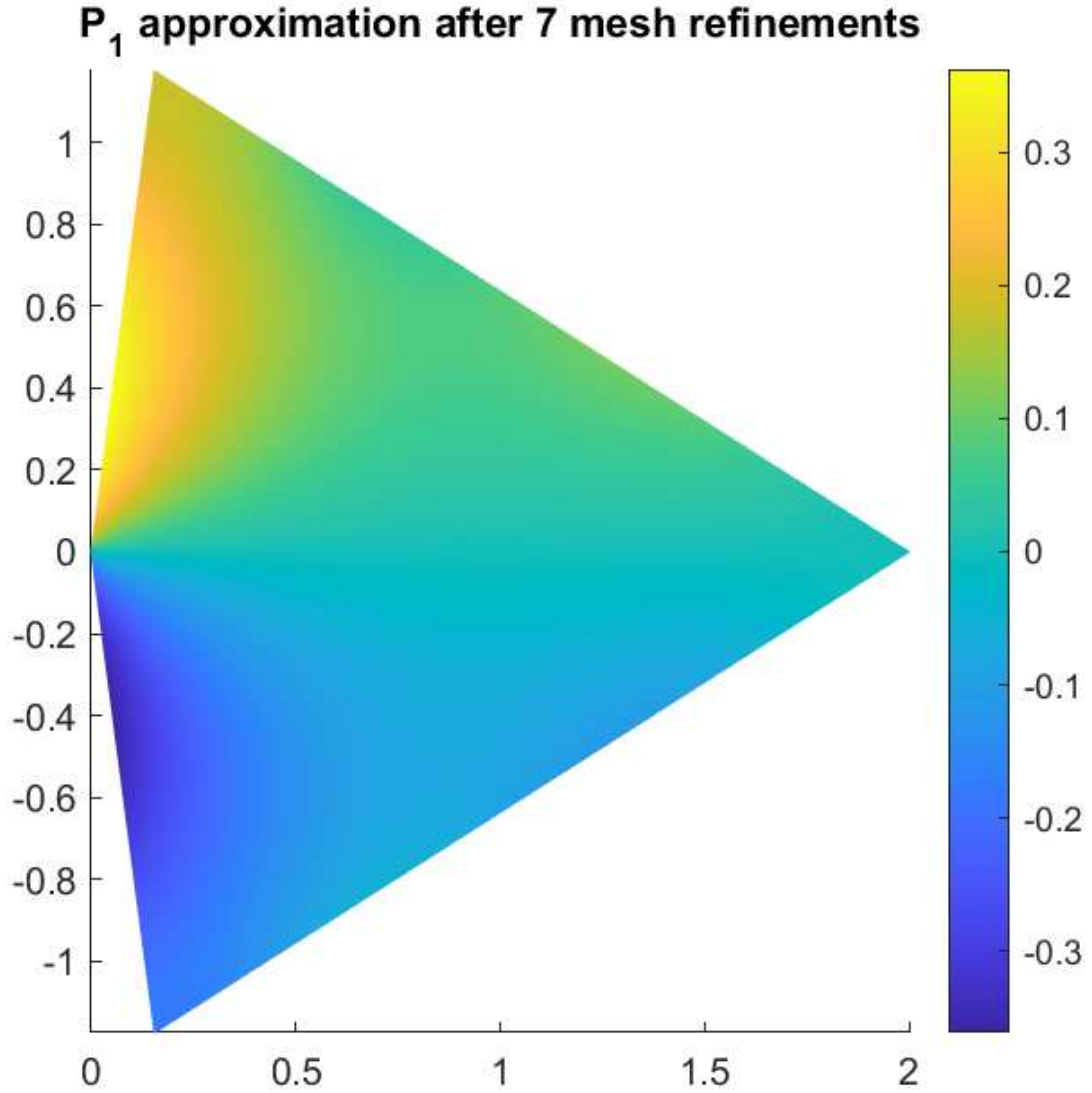}}
\caption{A convex domain with $\omega=\frac{11\pi}{12}$ (Example \ref{ex5.4}):  (a) the initial mesh; (b) $P_2$ finite element approximation $w_7$; (c) $P_2$ finite element approximation $\phi_7$; (d) Taylor-Hood element approximation $u_1$ of $\mathbf{u}_7$; (e) Taylor-Hood element approximation $u_2$ of $\mathbf{u}_7$; (f) Taylor-Hood element approximation $p_7$.}\label{Mesh_Init_Gradedex2}
\end{figure}

\begin{table}[!htbp]\tabcolsep0.04in
\caption{Convergence history of finite element approximation of the biharmonic problem with $\omega=\frac{11\pi}{12}$.}
\centering
\begin{tabular}{ |c| c| c c c c c| c c c c c|} \hline
 &  & \multicolumn{5}{c|}{$H_1$ rate of $\phi_j$} & \multicolumn{5}{c|}{ $L_2$ rate of $\phi_j$} \\
\hline
& $\kappa$ & 0.1 & 0.2 & 0.3 & 0.4 &0.5 & 0.1 & 0.2 & 0.3& 0.4&0.5 \\
\hline
\multirow{5}{*}{Algorithm \ref{femalg}}
&$j=6$ & 1.99 & 1.99 & 1.99 & 2.00 & 1.99   & 3.02 & 3.01 & 3.00 & 3.00 & 2.994  \\
&$j=7$ & 2.00 & 2.00 & 2.00 & 2.00 & 1.99   & 3.00 & 3.00 & 3.00 & 3.00 & 3.995  \\
&$j=8$ & 2.00 & 2.00 & 2.00 & 2.00 & 2.00   & 3.00 & 3.00 & 3.00 & 3.00 & 2.994  \\
&$j=9$ & 2.00 & 2.00 & 2.00 & 2.00 & 2.00   & 3.00 & 3.00 & 3.00 & 3.00 & 2.990  \\
&$j=10$ & 2.00 & 2.00 & 2.00 & 2.00 & 2.00  & 3.00 & 3.00 & 3.00 & 3.00 & 2.992  \\
\hline
\multirow{5}{*}{Algorithm \ref{femalg+}}
&$j=6$ & 1.99 & 1.99 & 1.99 & 2.00 & 1.99   & 3.02 & 3.01 & 3.00 & 3.00 & 2.994  \\
&$j=7$ & 2.00 & 2.00 & 2.00 & 2.00 & 1.99   & 3.01 & 3.00 & 3.00 & 3.00 & 2.995  \\
&$j=8$ & 2.00 & 2.00 & 2.00 & 2.00 & 2.00   & 3.00 & 3.00 & 3.00 & 3.00 & 2.994  \\
&$j=9$ & 2.00 & 2.00 & 2.00 & 2.00 & 2.00   & 3.00 & 3.00 & 3.00 & 3.00 & 2.990  \\
&$j=10$ & 2.00 & 2.00 & 2.00 & 2.00 & 2.00  & 3.00 & 3.00 & 3.00 & 3.00 & 2.984  \\
\hline
\end{tabular}\label{L_Poi_Rates25.4}
\end{table}

We show the Taylor-Hood element approximations $\mathbf{u}_7$ and $p_7$ from Algorithm \ref{femalg+} on uniform meshes in Figure \ref{Mesh_Init_Gradedex2}d-\ref{Mesh_Init_Gradedex2}f.
We also report the convergence rate of the Taylor-Hood element approximations of the involved Stokes problem in Table \ref{L_Stok_Rates25.4} (Algorithm \ref{femalg}) and Table \ref{L_Stok_Rates25.42} (Algorithm \ref{femalg+}). The $H^1$ convergence rate of $\mathbf{u}_j$ and the $L^2$ convergence rate of $p_j$ are converging with the rate $\mathcal{R}=1.20$ on uniform meshes, and it is consistent with the theoretical result $\mathcal{R}_{\text{exact}}=\alpha_0 \approx 1.20$ in Lemma \ref{stokesestlem} and Lemma \ref{stokesestlem+}. The optimal convergence rate $\mathcal{R}=2$ can be observed for $\kappa \leq 0.3$, which is consistent with the theoretical result $\kappa<2^{-\frac{2}{\alpha_0}} \approx 0.31$ in Theorem \ref{Stokesgraderr} and Remark \ref{stokesoptimal}.
The $L^2$ convergence rate of $\mathbf{u}_j$ on uniform meshes is $\mathcal{R}=2.20$, which is consistent with the expected rate $\mathcal{R}_{\text{exact}}=\alpha_0+1 \approx 2.20$. On graded meshes, the optimal convergence rate $\mathcal{R}=3$ is observed with $\kappa \leq 0.3$, which is also consistent with the theoretical result $\kappa<2^{-\frac{2}{\alpha_0}} \approx 0.31$ in Theorem \ref{StokesgradL2err} and Remark \ref{stokesoptimal}.

\begin{table}[!htbp]\tabcolsep0.04in
\caption{Convergence history of the Taylor-Hood element approximations ($k=2$) of Stokes problem from Algorithm \ref{femalg} with $\omega=\frac{11\pi}{12}$.}
\centering
\begin{tabular}{ |c| c c c c c| c c c c c | c c c c c |} \hline
 &  \multicolumn{5}{c|}{$H_1$ rate of $\mathbf{u}_j$} &  \multicolumn{5}{c|}{$L^2$ rate of $\mathbf{u}_j$}   &  \multicolumn{5}{c|}{ $L_2$ rate of $p_j$}  \\
\hline
 $\kappa$ & 0.1 & 0.2 & 0.3 & 0.4 &0.5 & 0.1 & 0.2 & 0.3& 0.4&0.5& 0.1 & 0.2 & 0.3& 0.4&0.5 \\
\hline
$j=6$ & 2.00 & 2.00 & 2.00 & 1.88 & 1.32   & 3.01 & 3.00 & 3.01 & 2.99 & 2.26  & 2.01 & 2.00 & 2.00 & 1.94 & 1.42  \\
$j=7$ & 2.00 & 2.00 & 2.00 & 1.83 & 1.24   & 3.00 & 3.00 & 3.00 & 2.98 & 2.22  & 2.00 & 2.00 & 2.00 & 1.90 & 1.29 \\
$j=8$ & 2.00 & 2.00 & 2.00 & 1.77 & 1.21   & 3.00 & 3.00 & 3.00 & 2.98 & 2.21  & 2.00 & 2.00 & 2.00 & 1.85 & 1.23  \\
$j=9$ & 2.00 & 2.00 & 2.00 & 1.71 & 1.20   & 3.00 & 3.00 & 3.00 & 2.98 & 2.20  & 2.00 & 2.00 & 2.00 & 1.79 & 1.21  \\
$j=10$ & 2.00 & 2.00 & 2.00 & 1.67 & 1.20 & 3.00 & 3.00 & 3.00 & 2.97 & 2.20  & 2.00 & 2.00 & 2.00 & 1.73 & 1.20  \\
\hline
\end{tabular}\label{L_Stok_Rates25.4}
\end{table}

\begin{table}[!htbp]\tabcolsep0.04in
\caption{Convergence history of the Taylor-Hood element approximations ($k=2$) of Stokes problem from Algorithm \ref{femalg+} with $\omega=\frac{11\pi}{12}$.}
\centering
\begin{tabular}{ |c| c c c c c| c c c c c | c c c c c |} \hline
 &  \multicolumn{5}{c|}{$H_1$ rate of $\mathbf{u}_j$} &  \multicolumn{5}{c|}{$L^2$ rate of $\mathbf{u}_j$}   &  \multicolumn{5}{c|}{ $L_2$ rate of $p_j$}  \\
\hline
 $\kappa$ & 0.1 & 0.2 & 0.3 & 0.4 &0.5 & 0.1 & 0.2 & 0.3& 0.4&0.5& 0.1 & 0.2 & 0.3& 0.4&0.5 \\
\hline
$j=6$ & 2.00 & 2.00 & 1.99 & 1.88 & 1.31   & 3.01 & 3.00 & 3.00 & 2.99 & 2.26  & 2.04 & 2.02 & 2.00 & 1.75 & 1.20  \\
$j=7$ & 2.00 & 2.00 & 2.00 & 1.83 & 1.24   & 3.00 & 3.00 & 3.00 & 2.98 & 2.22  & 2.02 & 2.01 & 2.00 & 1.69 & 1.19 \\
$j=8$ & 2.00 & 2.00 & 2.00 & 1.77 & 1.21   & 3.00 & 3.00 & 3.00 & 2.98 & 2.20  & 2.01 & 2.00 & 1.99 & 1.65 & 1.19  \\
$j=9$ & 2.00 & 2.00 & 2.00 & 1.71 & 1.20   & 3.00 & 3.00 & 3.00 & 2.98 & 2.20  & 2.00 & 2.00 & 1.99 & 1.63 & 1.19  \\
$j=10$ & 2.00 & 2.00 & 2.00 & 1.67 & 1.20 & 3.00 & 3.00 & 3.00 & 2.97 & 2.20  & 2.00 & 2.00 & 2.00 & 1.61 & 1.20  \\
\hline
\end{tabular}\label{L_Stok_Rates25.42}
\end{table}

\end{example}

\begin{example}

In this example, we compare the CPU time and the memory usage of the proposed finite element algorithms (Algorithm \ref{femalg} and Algorithm \ref{femalg+}) with those of the $H^2$-conforming Argyris finite element method by solving the biharmonic problem (\ref{eqnbi}) in Example \ref{ex5.1} on the same meshes. The results of the CPU time comparison (in seconds) are shown in Table \ref{TabCPU}.
The results of the memory usage comparison (in GB) are shown in Table \ref{TabM}.
All results are tested on MATLAB R2021a in Linux with 16 GB memory and Intel\circledR~Core$^{\text{TM}}$ i7-6600U processors.

\begin{table}[!htbp]\tabcolsep0.03in
\caption{The CPU time (in seconds) of Algorithm \ref{femalg}, Aglorithm \ref{femalg+}, and the Argyris finite element method. (``$--$" represents running out of memory)}
\begin{tabular}[c]{|c|ccccc|ccccc|}
\hline
\multirow{2}{*}{method$\backslash j$}  & \multicolumn{5}{|c|}{Example \ref{ex5.1} Test case 1} & \multicolumn{5}{|c|}{Example \ref{ex5.1} Test case 2}  \\
\cline{2-11}
& $j=4$ & $j=5$ & $j=6$ & $j=7$ & $j=8$ & $j=4$ & $j=5$ & $j=6$ & $j=7$ & $j=8$\\
\hline
Argyris FEM & 7.09 & 28.57 & 118.10 & 492.43 & $--$ & 5.27 & 21.62 & 87.70 & 367.76 & $--$ \\
\hline
Algorithm \ref{femalg}(k=1) & 0.08 & 0.39 & 2.82 & 15.92 & 90.94 & 0.05 & 0.24 & 1.13 & 9.59 & 58.84 \\
\hline
Algorithm \ref{femalg+}(k=1) & 0.10 & 0.42 & 3.01 & 16.09 & 96.83 & 0.08 & 0.30 & 1.36 & 11.37 & 66.22 \\
\hline
Algorithm \ref{femalg}(k=2) & 0.88 & 2.29 & 9.05 & 37.72 & 150.40 & 0.60 & 1.90 & 7.03 & 28.21 & 120.95 \\
\hline
Algorithm \ref{femalg+}(k=2) & 0.91 & 2.74 & 10.79 & 43.29 & 181.42 & 0.77 & 2.27 & 8.25 & 34.25 & 147.97 \\
\hline
\end{tabular}\label{TabCPU}
\end{table}

\begin{table}[!htbp]\tabcolsep0.03in
\caption{The memory usage (in GB) of Algorithm \ref{femalg}, Aglorithm \ref{femalg+}, and the Argyris finite element method.  (``$--$" represents running out of memory)}
\begin{tabular}[c]{|c|ccccc|ccccc|}
\hline
\multirow{2}{*}{method$\backslash j$}  & \multicolumn{5}{|c|}{Example \ref{ex5.1} Test case 1} & \multicolumn{5}{|c|}{Example \ref{ex5.1} Test case 2}  \\
\cline{2-11}
& $j=4$ & $j=5$ & $j=6$ & $j=7$ & $j=8$ & $j=4$ & $j=5$ & $j=6$ & $j=7$ & $j=8$  \\
\hline
Argyris FEM & 0.06 & 0.27 & 1.00 & 5.10 & $--$ & 0.04 & 0.13 & 0.54 & 3.00 & $--$ \\
\hline
Algorithm \ref{femalg}(k=1) & $<$0.01 & 0.01 & 0.02 & 0.41 & 3.03 & $<$0.01 & 0.01 & 0.02 & 0.27 & 2.20 \\
\hline
Algorithm \ref{femalg+}(k=1) & $<$0.01 & 0.01 & 0.02 & 0.46 & 2.98 & $<$0.01 & 0.01 & 0.02 & 0.27 & 2.19 \\
\hline
Algorithm \ref{femalg}(k=2) & 0.01 & 0.05 & 0.23 & 1.07 & 4.14 & $<$0.01 & 0.04 & 0.20 & 0.80 & 3.02 \\
\hline
Algorithm \ref{femalg+}(k=2) & 0.01 & 0.04 & 0.21 & 0.93 & 4.19 & $<$0.01 & 0.03 & 0.19 & 0.80 & 3.14 \\
\hline
\end{tabular}\label{TabM}
\end{table}

From Table \ref{TabCPU}, we find that Algorithm \ref{femalg} and Algorithm \ref{femalg+} are much faster than the Argyris finite element method due to the availability of fast Stokes solvers and Poisson solvers. Moreover, Algorithm \ref{femalg} is faster than Algorithm \ref{femalg+}, since Algorithm \ref{femalg+} has one extra Poisson problem to compute. The results in Table \ref{TabM} indicate that Algorithm \ref{femalg} and Algorithm \ref{femalg+} use much less memory compared with the Argyris finite element method. We notice that there are not too many differences in memory usage for both Algorithm \ref{femalg} and Algorithm \ref{femalg+}, which is because both algorithms are serial, and the Stokes solver accounts for the maximum memory usage.  

\end{example}

\appendix

\section{Error bounds of finite element approximations to the Stokes problem}

In this section, we only prove the error bound of the Taylor-Hood approximations to the Stokes problem, the error bound of the the Mini element approximations can be proved similarly. We introduce the following operators.
\be
\bal
\mathcal{B} = & -\text{div} \ : \ [H_0^1(\Omega)]^2 \rightarrow (L_0^2(\Omega))'=L_0^2(\Omega), \quad \langle \mathcal{B} \mathbf{v}, q \rangle = -(\text{div } \mathbf{v}, q), \\
\mathcal{B}' = & \nabla \ : \ L_0^2(\Omega) \rightarrow [H^{-1}(\Omega)]^2, \quad \langle \mathbf{v}, \mathcal{B}'q  \rangle = -(\text{div } \mathbf{v}, q).
\eal
\ee
We denote $\mathcal{B}_n  :  [V_n^{k}]^2 \rightarrow (S_{n}^{k-1})'$ be the discrete counterpart of the operator $\mathcal{B}$ and it satisfies 
$$
\langle \mathcal{B}_n \mathbf{v}, q \rangle = \langle \mathcal{B} \mathbf{v}, q \rangle = -(\text{div } \mathbf{v}, q) \quad \forall (\mathbf{v}, q) \in [V_n^{k}]^2 \times S_{n}^{k-1}.
$$
The nullspace of $\mathcal{B}_n$ is given by
\be\label{kerbn}
\ker(\mathcal{B}_n) = \{ \mathbf{v} \in  [V_n^{k}]^2 \ | \ \forall q \in S_{n}^{k-1}, (\text{div } \mathbf{v}, q) = 0 \}.
\ee
\begin{lem}\label{stokesTHbdd}
Let $(\mathbf{u}, p)$ be the solution of the Stokes problem (\ref{stokesweak}), and $(\mathbf{u}_n, p_n)$ be the Taylor-Hood element solution ($k\geq 2$) of (\ref{stokesfem}) satisfying the LBB condition (\ref{skinfsupTH}), then it follows
\bes
\|\mathbf{u}-\mathbf{u}_n\|_{[H^1(\Omega)]^2} + \|p-p_n\| \leq C \left( \inf_{\mathbf{v} \in [V_n^{k}]^2 } \|\mathbf{u}-\mathbf{v}\|_{[H^1(\Omega)]^2} + \inf_{ q \in S_{n}^{k-1} } \|p-q\| + \|\mathbf{F}-\mathbf{F}_n\|_{[H^{-1}(\Omega)]^2} \right).
\ees
\end{lem}
\begin{proof}
By the LBB condition (\ref{skinfsupTH}a), we have that the operator $B_n$ is surjective. Thus, for given $\mathbf{v} \in [V_n^{k}]^2$, 
there exists $\tilde{\mathbf{v}} \in [V_n^{k}]^2$ such that 
\be\label{Bnproj}
\mathcal{B}_n \tilde{\mathbf{v}} = \mathcal{B}_n(\mathbf{u}_n-\mathbf{v})
\ee
and
\be\label{tvubb}
\tilde{\gamma}_1 \|\tilde{\mathbf{v}}\|_{[H^1(\Omega)]^2} \leq \sup_{q \in S_{n}^{k-1}}\frac{-(\text{div } (\mathbf{u}_n-\mathbf{v}), q)}{\|q\|} = \sup_{q \in S_{n}^{k-1}}\frac{-(\text{div } (\mathbf{u}-\mathbf{v}), q)}{\|q\|}  \leq C_1 \|\mathbf{u}-\mathbf{v}\|_{[H^1(\Omega)]^2},
\ee
where we have used (\ref{notgo}b).

Set $\mathbf{w} = \mathbf{v}+\tilde{\mathbf{v}} \in [V_n^{k}]^2$, then it follows by (\ref{Bnproj}),
$$
\mathcal{B}_n(\mathbf{u}_n-\mathbf{w}) = \mathcal{B}_n(\mathbf{u}_n-\mathbf{v}-\tilde{\mathbf{v}})  = 0,
$$
which implies that $\mathbf{u}_n-\mathbf{w} \in \ker(\mathcal{B}_n)$. 

By (\ref{skinfsupTH}b), we have
\bes
\bal
\tilde\gamma_2\|\mathbf{u}_n-\mathbf{w}\|_{[H^1(\Omega)]^2} \leq & \frac{(\nabla (\mathbf{u}_n-\mathbf{w}), \nabla (\mathbf{u}_n-\mathbf{w}))}{\|\mathbf{u}_n-\mathbf{w}\|_{[H^1(\Omega)]^2}} \leq  \sup_{\tilde{\mathbf{w}}\in \ker(\mathcal{B}_n)} \frac{(\nabla (\mathbf{u}_n-\mathbf{w}), \nabla \tilde{\mathbf{w}})}{\|\tilde{\mathbf{w}}\|_{[H^1(\Omega)]^2}}\\
= & \sup_{\tilde{\mathbf{w}}\in \ker(\mathcal{B}_n)} \frac{(\nabla (\mathbf{u}_n-\mathbf{u}), \nabla \tilde{\mathbf{w}}) + (\nabla (\mathbf{u}-\mathbf{w}), \nabla \tilde{\mathbf{w}})}{\|\tilde{\mathbf{w}}\|_{[H^1(\Omega)]^2}}\\
= & \sup_{\tilde{\mathbf{w}}\in \ker(\mathcal{B}_n)} \frac{-(\text{div } \tilde{\mathbf{w}}, p-p_n) - \langle\mathbf{F}-\mathbf{F}_n, \tilde{\mathbf{w}} \rangle + (\nabla (\mathbf{u}-\mathbf{w}), \nabla \tilde{\mathbf{w}})}{\|\tilde{\mathbf{w}}\|_{[H^1(\Omega)]^2}}\\
= & \sup_{\tilde{\mathbf{w}}\in \ker(\mathcal{B}_n)} \frac{-(\text{div } \tilde{\mathbf{w}}, p-p_n) + (\nabla (\mathbf{u}-\mathbf{w}), \nabla \tilde{\mathbf{w}})}{\|\tilde{\mathbf{w}}\|_{[H^1(\Omega)]^2}}+\sup_{\tilde{\mathbf{w}}\in \ker(\mathcal{B}_n)} \frac{ - \langle\mathbf{F}-\mathbf{F}_n, \tilde{\mathbf{w}} \rangle}{\|\tilde{\mathbf{w}}\|_{[H^1(\Omega)]^2}},
\eal
\ees
where we have used (\ref{notgo}a). For any $q \in S_{n}^{k-1}$, we have by (\ref{kerbn}) with $\tilde{\mathbf{w}} \in \ker(\mathcal{B}_n)$,
$$
-(\text{div } \tilde{\mathbf{w}}, p_n-q) = 0.
$$
Therefore, it follows
\be\label{unwubb}
\bal
\tilde\gamma_2\|\mathbf{u}_n-\mathbf{w}\|_{[H^1(\Omega)]^2} \leq &  \sup_{\tilde{\mathbf{w}}\in \ker(\mathcal{B}_n)} \frac{-(\text{div } \tilde{\mathbf{w}}, p-q) - \langle\mathbf{F}-\mathbf{F}_n, \tilde{\mathbf{w}} \rangle + (\nabla (\mathbf{u}-\mathbf{w}), \nabla \tilde{\mathbf{w}})}{\|\tilde{\mathbf{w}}\|_{[H^1(\Omega)]^2}}\\
\leq & C_1 \|p-q\| + C_2 \|\mathbf{u}-\mathbf{w}\|_{[H^1(\Omega)]^2} + \|\mathbf{F}-\mathbf{F}_n\|_{[H^{-1}(\Omega)]^2},
\eal
\ee

Note that using triangle inequality and (\ref{tvubb}) yields
\be\label{uwtouv}
\|\mathbf{u}-\mathbf{w}\|_{[H^1(\Omega)]^2} \leq \|\mathbf{u}-\mathbf{v}\|_{[H^1(\Omega)]^2} + \|\tilde{\mathbf{v}}\|_{[H^1(\Omega)]^2} \leq \left(1+\frac{C_1}{\tilde\gamma_1} \right)\|\mathbf{u}-\mathbf{v}\|_{[H^1(\Omega)]^2}.
\ee
Thus, we have by the triangle inequality and (\ref{uwtouv})
\be\label{uubdd}
\bal
\|\mathbf{u}-\mathbf{u}_n\|_{[H^1(\Omega)]^2} \leq& \left(1+ \frac{C_2}{\tilde\gamma_2}\right)\|\mathbf{u}-\mathbf{w}\|_{[H^1(\Omega)]^2} + \frac{C_2}{\tilde\gamma_2} \|p-q\| + \frac{1}{\tilde\gamma_2} \|\mathbf{F}-\mathbf{F}_n\|_{[H^{-1}(\Omega)]^2}\\
\leq & C_3 \|\mathbf{u}-\mathbf{v}\|_{[H^1(\Omega)]^2} + \frac{C_2}{\tilde\gamma_2} \|p-q\| + \frac{1}{\tilde\gamma_2} \|\mathbf{F}-\mathbf{F}_n\|_{[H^{-1}(\Omega)]^2},
\eal
\ee
where $C_3=\left(1+\frac{C_1}{\tilde\gamma_1} \right)\left(1+ \frac{C_2}{\tilde\gamma_2}\right)$.

Next, we need to obtain the estimate for $\|p-p_n\|$. From (\ref{notgo}a), we have 
$$
- (\text{div } \mathbf{v}, q-p_n) = - (\nabla (\mathbf{u}-\mathbf{u}_n), \nabla \mathbf{v}) +  \langle\mathbf{F}-\mathbf{F}_n, \mathbf{v} \rangle - (\text{div } \mathbf{v}, q-p).
$$
By the LBB condition (\ref{skinfsupTH}a) and the boundedness of the bilinear forms, we have
\bes
\bal
\tilde \gamma_1 \|q-p_n\| \leq C_2 \|\mathbf{u}-\mathbf{u}_n\|_{[H^1(\Omega)]^2} +  \|\mathbf{F}-\mathbf{F}_n\|_{[H^{-1}(\Omega)]^2} + C_1\|p-q\|.
\eal
\ees
Thus, we have
\be\label{ppnupp}
\bal
\|p-p_n\| \leq \left(1+\frac{C_1}{\tilde\gamma_1}\right)\|p-q\|+\frac{C_2}{\tilde\gamma_1}\|\mathbf{u}-\mathbf{u}_n\|_{[H^1(\Omega)]^2} +\frac{1}{\tilde\gamma_1} \|\mathbf{F}-\mathbf{F}_n\|_{[H^{-1}(\Omega)]^2}.
\eal
\ee
The conclusion follows from (\ref{uubdd}) and (\ref{ppnupp}).
\end{proof}

If $\mathbf{F}_n$ is the $L^2$ projection of $\mathbf{F}$, i.e., $\langle\mathbf{F}-\mathbf{F}_n, \mathbf{v} \rangle=0$ for $\forall \mathbf{v} \in [V_n^k]^2$, then
(\ref{notgo}) becomes the Galerkin orthogonality of a general Taylor-Hood method, and the result in Lemma \ref{stokesTHbdd} degenerates to the well known estimate bound in \cite{brezzi1974}.
\begin{corollary}\label{stokesTHbdd+}
Let $(\mathbf{u}, p)$ be the solution of the Stokes problem (\ref{stokesweak}), and $(\mathbf{u}_n, p_n)$ be the Taylor-Hood element solution ($k\geq 2$) in Algorithm \ref{femalg} satisfying the LBB condition (\ref{skinfsupTH}), then it follows
\bes
\|\mathbf{u}-\mathbf{u}_n\|_{[H^1(\Omega)]^2} + \|p-p_n\| \leq C \left( \inf_{\mathbf{v} \in [V_n^{k}]^2 } \|\mathbf{u}-\mathbf{v}\|_{[H^1(\Omega)]^2} + \inf_{ q \in S_{n}^{k-1} } \|p-q\| \right).
\ees
\end{corollary}

For Mini element approximations to the Stokes problem, we have the following error bounds.
\begin{lem}\label{stokesMinibdd}
Let $(\mathbf{u}, p)$ be the solution of the Stokes problem (\ref{stokesweak}), and $(\mathbf{u}_n, p_n)$ be the Mini element solution ($k=1$) in Algorithm \ref{femalg+} satisfying the LBB condition (\ref{skinfsupMini}), then it follows
\bes
\|\mathbf{u}-\mathbf{u}_n\|_{[H^1(\Omega)]^2} + \|p-p_n\| \leq C \left( \inf_{\mathbf{v} \in [V_n^{1}]^2 } \|\mathbf{u}-\mathbf{v}\|_{[H^1(\Omega)]^2} + \inf_{ q \in S_{n}^{1} } \|p-q\| + \|\mathbf{F}-\mathbf{F}_n\|_{[H^{-1}(\Omega)]^2} \right).
\ees
Moreover, if 
$(\mathbf{u}_n, p_n)$ is the Mini element solution ($k=1$) in Algorithm \ref{femalg} satisfying the LBB condition (\ref{skinfsupMini}), then it follows
\bes
\|\mathbf{u}-\mathbf{u}_n\|_{[H^1(\Omega)]^2} + \|p-p_n\| \leq C \left( \inf_{\mathbf{v} \in [V_n^{1}]^2 } \|\mathbf{u}-\mathbf{v}\|_{[H^1(\Omega)]^2} + \inf_{ q \in S_{n}^{1} } \|p-q\| \right).
\ees
\end{lem}

\section*{Acknowledgments}
H. Li was supported in part by the National Science Foundation Grant DMS-1819041 and by the Wayne State University Faculty Competition for Postdoctoral Fellows Award.


\bibliography{LWY21}

\bibliographystyle{plain}

\end{document}